\documentclass[a4paper,reqno]{amsart}

\usepackage{setspace}
\usepackage[totalwidth=14cm,totalheight=22cm]{geometry}
\usepackage[T1]{fontenc}
\usepackage[dvips]{graphicx}
\usepackage{epsfig}
\usepackage{amsmath,amssymb,amscd,amsthm}
\usepackage{subfigure}
\usepackage[latin1]{inputenc}
\usepackage{latexsym}
\usepackage{graphics}
\usepackage{colortbl}
\usepackage{color}
\usepackage{ae}
\usepackage{enumitem}
\usepackage[all]{xy}

\newtheorem{cor}{Corollary}[section]
\newtheorem{theorem}[cor]{Theorem}
\newtheorem{prop}[cor]{Proposition}
\newtheorem{lemma}[cor]{Lemma}

\theoremstyle{definition}
\newtheorem{defi}[cor]{Definition}
\theoremstyle{remark}
\newtheorem{remark}[cor]{Remark}
\newtheorem{example}[cor]{Example}

\DeclareMathOperator\arctanh{arctanh}

\newcommand{\C}{{\mathbb C}}

\newcommand{\R}{{\mathbb R}}

\newcommand{\dev}{\mbox{dev}}

\newcommand{\hess}{\mbox{Hess}}

\newcommand{\Hyp}{\mathbb{H}}

\newcommand{\dS}{\mathrm{d}\mathbb{S}}

\newcommand{\co}{\mathrm{co}\,}

\newcommand{\grad}{\operatorname{grad}}
\newcommand{\isom}{\mathrm{Isom}}

\newcommand{\Ip}{\mathrm{I}^+}
\newcommand{\Ipm}{\mathrm{I}^-}

\newcommand{\rar}{\rightarrow}

\newcommand{\war}{\rightharpoonup}

\newcommand{\id}{\mathrm{I}}
\newcommand{\SO}{\mathrm{SO}}
\newcommand{\so}{\mathfrak{so}}
\newcommand{\ddt}{\left.\frac{d}{dt}\right|_{t=0}}

\newcommand{\D}{\mathbb{D}}

\def\Hess{\mathrm{Hess}}

\begin{document}

\setcounter{secnumdepth}{3}
\setcounter{tocdepth}{2}

\title{Spacelike convex surfaces with prescribed curvature in (2+1)-Minkowski space}


\author[Francesco Bonsante]{Francesco Bonsante}
\address{Francesco Bonsante: Dipartimento di Matematica ``Felice Casorati", Universit\`{a} degli Studi di Pavia, Via Ferrata 1, 27100, Pavia, Italy.} \email{bonfra07@unipv.it} 
\author[Andrea Seppi]{Andrea Seppi}
\address{Andrea Seppi: Dipartimento di Matematica ``Felice Casorati", Universit\`{a} degli Studi di Pavia, Via Ferrata 1, 27100, Pavia, Italy.} \email{andrea.seppi01@ateneopv.it}

\thanks{The authors were partially supported by FIRB 2010 project ``Low dimensional geometry and topology'' (RBFR10GHHH003). The first author was partially supported by
PRIN 2012 project ``Moduli strutture algebriche e loro applicazioni''.
The authors are members of the national research group GNSAGA}

\begin{abstract}
We prove existence and uniqueness of solutions to the Minkowski problem in any domain of dependence $D$ in $(2+1)$-dimensional Minkowski space, provided $D$ is contained in the future cone over a point. Namely, it is possible to find a smooth convex Cauchy surface with prescribed curvature function on the image of the Gauss map. This is related to solutions of the Monge-Amp\`ere equation $\det D^2 u(z)=(1/\psi(z))(1-|z|^2)^{-2}$ on the unit disc, with the boundary condition $u|_{\partial\D}=\varphi$, for $\psi$ a smooth positive function and $\varphi$ a bounded lower semicontinuous function.

We then prove that a domain of dependence $D$ contains a convex Cauchy surface with principal curvatures bounded from below by a positive constant  if and only if
the corresponding function $\varphi$ is in the  Zygmund class. Moreover in this case the surface of constant curvature $K$ contained in $D$ has bounded
principal curvatures, for every $K<0$. In this way we get a full classification of isometric immersions of the hyperbolic plane in Minkowski space with bounded shape operator in terms of
Zygmund functions of $\partial \D$.

Finally, we prove that every domain of dependence as in the hypothesis of the Minkowski problem is foliated by the surfaces of constant curvature $K$, as $K$ varies in $(-\infty,0)$.
\end{abstract}

\maketitle




\section{Introduction}

A classical theorem of Riemannian geometry states that if $\sigma: \mathbb{S}^2\to\R^3$ is an isometric immersion of the round sphere into Euclidean space,
then it is the standard inclusion up to an isometry of $\R^3$. On the other hand, Hano and Nomizu (\cite{hanonomizu}) first proved that the analogous statement does not hold in 
Minkowski space $\R^{2,1}$. Namely there are isometric embeddings of the hyperbolic plane in $\R^{2,1}$ which are not equivalent to the standard inclusion
of the hyperboloid model of $\Hyp^2$ into $\R^{2,1}$.
 
 A possible way to understand the rigidity in the Euclidean case makes use of the so-called support function: basically if $\sigma:\mathbb{S}^2\to\R^3$ is an isometric immersion, 
 the image of $\sigma$ must be a locally convex surface by the Gauss equation. In particular it turns out that the Gauss map $G:\mathbb{S}^2\to \mathbb{S}^2$ is bijective, and the support function is defined
 as $\bar u: \mathbb{S}^2\to \R$, $\bar u(x)=\langle x, G^{-1}(x)\rangle$. A simple computation shows that $\det(\Hess^{\mathbb{S}^2}\bar u+\bar u\id)=1$, where $\hess^{\mathbb{S}^2}$ is the covariant Hessian on
 the sphere.   By a wise use of the comparison principle it turns out that the difference of any two solutions must be the restriction on $\mathbb{S}^2$ of a linear form on $\R^3$.
  This allows to conclude that every solution is of the form $\bar u(x)=1+\langle x, \xi\rangle$ for some $\xi\in\R^3$. Therefore the surface $\sigma(\mathbb{S}^2)$ is the round sphere of radius $1$ centered at $\xi$.
 
 The \emph{support function} $\bar u$ can be analogously defined for a spacelike convex immersion $\sigma:\Hyp^2\to\R^{2,1}$. If $\sigma$ is an isometric immersion, then $\bar u$ satisfies
 the equation $\det(\hess^{\Hyp^2}\bar u-\bar u\id)=1$. However the maximum principle cannot be directly used in this context by the non-compactness of $\Hyp^2$. This is a general
 indication that some boundary condition must be taken into account to determine the solution and  the immersion $\sigma$.
 
 The classical Minkowski problem in Euclidean space  can be also formulated for Minkowski space. Given a  smooth spacelike strictly convex surface $S$ in $\R^{2,1}$, the 
 curvature function is defined as $\psi:G(S)\to \R$, $\psi(x)=-K_S(G^{-1}(x))$, where $G:S\to\Hyp^2$ is the Gauss map and $K_S$ is the scalar intrinsic curvature on $S$.
 Minkowski problem consists in finding a convex surface in Minkowski space whose curvature function is a prescribed positive function $\psi$.
Using the support function technology, the problem  turns out to  be equivalent to solving the equation 
\begin{equation}\label{eq:monhyp}
\det(\hess^{\Hyp^2}\bar u-\bar u\id)=\frac{1}{\psi}\,.
\end{equation}
Also in this case for the well-posedness of the problem some boundary conditions must be imposed.

Using the Klein model of $\Hyp^2$,
Equation \eqref{eq:monhyp} can be reduced to a standard Monge-Amp\`ere equation over the unit disc $\D$.
In particular solutions of \eqref{eq:monhyp} explicitly correspond to solutions $u:\D\to\R$ of the equation 
\begin{equation}\label{eq:moneucl}
\det D^2u(z)=\frac{1}{\psi(z)}(1-|z|^2)^{-2}\,.
\end{equation}

It should be remarked that the correspondence can degenerate in some sense. Indeed given a convex surface $S$, the support function $u$
is  defined only on a convex subset of $\D$. On the other hand, any convex function over $\D$ corresponds to some convex surface $S$ in Minkowski
space, but in general $S$ might contain lightlike rays. 
We say that a convex surface is a spacelike entire graph if $S=\{(p, f(p))\,|\,p\in\R^2\}$, where $f:\R^2\to\R$ is a $C^1$ function on the horizontal plane 
such that $||D f(p)||<1$ for all $p\in\R^2$. 

In \cite{Li}, Li studied the Minkowski problem in Minkowski space in any dimension showing the existence and uniqueness
of the solution of \eqref{eq:moneucl} imposing $u|_{\partial\D}=\varphi$, for a given smooth $\varphi$. 
The result was improved in dimension $2+1$ by Guan, Jian and Schoen in \cite{schoenetal}, where the existence of the solution is proved assuming that
the boundary data is only Lipschitz.
The solutions obtained in both cases correspond to spacelike entire graphs.

A remarkable result in \cite{Li} is that under the assumption that the boundary data is smooth, the corresponding convex surface $S$ has principal curvatures
bounded from below by a positive constant. As a partial converse statement, if $S$ has principal curvatures bounded from below by a positive constant, then the corresponding function
$u:\D\to\R$ extends to a continuous function of the boundary of $\D$.

In a different direction Barbot, B\'eguin and Zeghib (\cite{barbotzeghib}) solved the Minkowski problem for surfaces invariant by an affine deformation of a cocompact Fuchsian group.
Let $G$ be a cocompact Fuchsian group, and $\Gamma$ an affine deformation of $G$ ($\Gamma$ is a group of affine transformations whose elements  are obtained
by adding a translation part to elements of $G$). If $S$ is a $\Gamma$-invariant surface, its curvature function $\psi$ is $G$-invariant.
Barbot, B\'eguin and Zeghib proved that, given a positive $G$-invariant function $\psi$, there exists a unique solution of Minkowski problem which is $\Gamma$-invariant.

If $u:\D\to\R$ is the support function corresponding to some $\Gamma$-invariant surface $S$, combining the result by Li and the cocompactness of $\Gamma$, it turns out that
$u$ extends to the boundary of $\D$. It is not difficult to see that the extension on the boundary only depends on $\Gamma$ and in particular it is independent of the curvature function.
However the result in \cite{barbotzeghib} is not a consequence of the results in \cite{Li, schoenetal}, as it is not likely that $u|_{\partial\D}$ is Lipschitz continuous. This gives an indication that in dimension $2+1$ results
in \cite{Li, schoenetal} are not sharp.

One of the goals of the paper is to determine the exact regularity class of the extension on $\partial\D$ of functions $u:\D\to\R$ corresponding to
surfaces with principal curvatures bounded from below.

\begin{theorem}\label{thm:main1}
Let  $\varphi:\partial\D\to\R$ be a continuous function. Then there exists a spacelike entire graph in $\R^{2,1}$ whose principal curvatures are bounded from below by a positive constant
and whose support function $u$ extends $\varphi$ if and only if $\varphi$ is in the Zygmund class.
\end{theorem}
 
 Recall that a function $\varphi:S^1\to\R$ is in the \emph{Zygmund class}  if there is a constant $C$ such that, for every $\theta, h\in\R$,
 \[
 |\varphi(e^{i(\theta+h)})+\varphi(e^{i(\theta-h)})-2\varphi(e^{i\theta})|<C|h|\,.
 \]
 Functions in the Zygmund class are $\alpha$-H\"older for every $\alpha\in(0,1)$, but in general they are not Lipschitz.
 
 Theorem \ref{thm:main1} implies  that spacelike entire graphs of constant curvature $-1$ and with a uniform bound on the principal curvatures
 correspond to functions $u$ whose extension to $\partial\D$ is Zygmund. We prove that also the converse holds. This gives  a complete classification
 of such surfaces in terms of Zygmund functions.
 
 \begin{theorem}\label{thm:main2}
 Let $\varphi:\partial\D\to\R$ be a function in the Zygmund class.
 For every $K<0$ there is a unique spacelike entire graph $S$ in $\R^{2,1}$ of constant curvature $K$ and with bounded principal curvatures  whose corresponding  function $u$
 extends $\varphi$.
 \end{theorem} 
 
 The proof of Theorem \ref{thm:main2} relies on a general statement we prove about solvability of Minkowski problem.
 We precisely prove the following Theorem.
 
 \begin{theorem}\label{thm:main3}
 Let $\varphi:\partial \D\to\R$ be a lower semicontinuous and bounded function and $\psi: \D\to [a,b]$ for some $0<a<b<+\infty$.
 Then there exists a unique spacelike graph $S$ in $\R^{2,1}$ whose support function $u$ extends $\varphi$ and whose
 curvature function is $\psi$. 
 \end{theorem}
 
In general we say that a convex function $u$ extends $\varphi$ if $\varphi(z_0)=\liminf_{z\to z_0} u(z)$ for every $z_0\in\partial\D$.
By convexity, if $\varphi$ is continuous this condition is  equivalent to requiring that $u$ is continuous up to the boundary and its boundary value
coincides with $\varphi$.

Let us explain the geometric meaning of the boundary value of the support function of $S$.
As $\partial\D$, regarded as the set of lightlike directions, parameterizes lightlike linear planes, 
the restriction of the support function on $\partial\D$ gives the height function of 
lightlike support planes of $S$, where $u(\eta)=+\infty$ means that there is no lightlike support plane orthogonal to $\eta$.
It can be checked that, when $S$ is the graph of a convex function $f:\R^2\to\R$, the condition $u|_{\partial\D}=\varphi$ is also equivalent to requiring that
$$\lim_{r\to+\infty}\left(r-f(rz)\right)=\varphi(z)$$
for every $z\in\partial\D$. The asymptotic condition is stated in the latter fashion for instance in \cite{treibergs} and \cite{choitreibergs}, where the existence problem for constant mean curvature surfaces is treated.

In this paper we will consider future convex surfaces with bounded support function on $\D$. Geometrically this means that $S$ is contained in the future cone
of some point $p\in\R^{2,1}$.

It is also useful to consider convex objects more general than spacelike surfaces. A \emph{future-convex domain} is defined as an open domain in $\R^{2,1}$ which is
the intersection of a family of future half-spaces with spacelike boundary planes. Given a future convex set, $D_0$ we consider a bigger domain $D$ obtained as the intersection of the future of the  lightlike support planes of $D_0$.
From the Lorentzian point of view $D$ is the Cauchy development of the boundary of $D_0$. A domain $D$ obtained in this way is called a \emph{domain of dependence}.
An immediate consequence of Theorem \ref{thm:main3} is that for any domain of dependence and any $K<0$ there exists a unique convex surface $S$
in $\R^{2,1}$ of constant curvature $K$ whose Cauchy development coincides with $D$.
More precisely we prove the following Theorem.

\begin{theorem}\label{thm:main4}
If $D$ is a domain of dependence contained in the future cone of a point, then $D$ is foliated by surfaces of constant curvature $K\in(-\infty,0)$.
\end{theorem}

A problem that remains open is to characterize spacelike entire graphs  with bounded curvature which are complete for the induced metric.
If a surface has bounded principal curvatures, the Gauss map turns to be bi-Lipschitz, hence the surface is automatically complete.
On the opposite side we construct an example of non complete entire graphs which constant curvature; in our example the boundary extension
$\varphi$ of the support function is bounded but not continuous. This result underlines a remarkable difference with respect to CMC surfaces.
Indeed it was proved in \cite{chengyaumax, treibergs} that an entire CMC spacelike graph is automatically complete.

The surfaces we construct are invariant under a one-parameter parabolic group of isometries of $\R^{2,1}$ fixing the origin, and are isometric to a half-plane in $\Hyp^2$. This strategy goes back to Hano and Nomizu, who first exhibited non-standard immersions of the hyperbolic plane in $\R^{2,1}$ as \emph{surfaces of revolutions}, namely surfaces invariant under a one-parameter hyperbolic group fixing the origin.

\subsection*{Ingredients in the proofs}
We will use a description of domains of dependence due to Mess.
The key fact is that it is possible to associate to a domain of dependence $D$ a \emph{dual} measured geodesic lamination of $\Hyp^2$.
It turns out that if $D$ is a domain of dependence, its support function $u$ is
 the convex envelope of its boundary value on $\partial\D$.
 Heuristically, the graph of $u$ is  a piece-wise linear convex pleated surface. The bending lines provide a geodesic lamination over
$\Hyp^2$, whereas  a transverse measure encodes the amount of  bending.

This correspondence will be crucial in the present work.
Solving the Minkowski problem with a given boundary value $\varphi$ is equivalent to finding a convex entire graph $S$ in $\R^{2,1}$ whose
curvature function is prescribed and whose Cauchy development is the domain of dependence $D$ determined by $\varphi$.
If $\mu$ is the dual measured geodesic lamination of $D$, we 
construct a sequence of measured geodesic laminations $\mu_n$ and
Fuchsian groups $G_n$ such that $\mu_n$ is $G_n$-invariant and $\mu_n$ converges to $\mu$ on compact subsets of $\Hyp^2$ in an appropriate sense.
By a result of \cite{barbotzeghib} for every $n$ it is possible to solve the Minkowski problem. The proof of Theorem \ref{thm:main3} is obtained by taking
$\Gamma_n$-invariant curvature functions $\psi_n$ converging to the prescribed curvature function, and 
showing that solutions $S_n$ of Minkowski problem  for the domain $D_n$ dual to $\mu_n$ with curvature function $\psi_n$ converge to a solution of the original Minkowski problem for $D$.

The convergence of solutions is obtained first by showing the convergence of the support functions and then by proving that the surface dual to the limit support function is an entire graph.

For the first step a simple application of the maximum principle implies some a priori bounds of the support functions of $S_n$ in terms of the support function of $D_n$.
This allows to conclude that the support functions of $S_n$ converge to a convex function $u:\D\to\R$ extending $\varphi$. Applying standard theory of Monge-Amp\`ere equation we have
that $u$ is a generalized solution of our problem. On the other hand, in dimension $2$,  Alexander-Heinz theorem implies that $u$ is strictly convex and by standard regularity theorem
we conclude that $u$ is a classical solution.

The second step is more geometric. The key idea is to use - as \emph{barriers} - the already mentioned constant curvature surfaces which are invariant under a $1$-parameter parabolic group
fixing a point $\eta_0\in\partial\D$.
A similar approach (also in higher dimension) using surfaces invariant for a hyperbolic group is taken in \cite{schoenetal}. Since the parabolic-invariant surfaces are entire graphs and 
their support function is constant on  $\partial \D\setminus\{\eta_0\}$, they are very appropriate to
 show (by applying the comparison principe) that the boundary of the domain dual to a solution of \eqref{eq:moneucl}
cannot contain lightlike rays. The argument works well  under the hypothesis that the boundary value of the support function is bounded, leading to the proof
of Theorem \ref{thm:main3}.

The proof of Theorem \ref{thm:main1} is based on a relation we point out between convex geometry in Minkowski space and the theory of infinitesimal earthquakes
introduced and studied in \cite{gardbending, garhulakic, saricweak, saric_earthquakes}.
In particular given $\varphi:\partial \D\rar\R$ we prove that the convex envelope of $\varphi$ is explicitly related to the infinitesimal earthquake extending the field $\varphi\frac{\partial\,}{\partial\theta}$.
From this correspondence we see that the dual lamination associated by Mess to the domain defined by $\varphi$ is equal to the earthquake lamination.
Using a result of Gardiner, Hu and Lakic (\cite{garhulakic}) we deduce that $\varphi$ is Zygmund if and only if the Thurston norm of the dual lamination is finite.
Given a convex entire graph with principal curvatures bounded from below by a positive constant, we point out by a direct geometric construction in Minkowski space an
explicit estimate on the Thurston norm of the dual lamination. By the above correspondence, this proves one direction of the statement of Theorem \ref{thm:main1}.

Conversely, we show that a spacelike entire graph of constant curvature with bounded dual lamination has bounded principal curvatures.
This proves Theorem \ref{thm:main2} and shows the other implication of Theorem \ref{thm:main1}.
The proof is obtained by contradiction. The key fact is the following: if $D$ is a domain of dependence with bounded
dual lamination and $\alpha_n$ is any sequence of isometries of Minkowski space such that $\alpha_n(D)$ contains a fixed point $p$ with horizontal
support plane at $p$, then $\alpha_n(D)$ converges to a domain of dependence with bounded dual lamination.
Now let $S$ be an entire graph of constant curvature whose domain of dependence has bounded dual lamination.
If the principal curvatures of $S$ were not bounded, we could construct a sequence of isometries $\alpha_n$ bringing back to a fixed point
a sequence $p_n\in S$ where the principal curvature are degenerating. The statement above implies that the domain of dependence of $\alpha_n(S)$
is converging to a domain with bounded dual lamination. 
Applying standard regularity theory of Monge-Amp\`ere equations to the support function of $\alpha_n(S)$ one obtains an a priori bound
on the second derivatives of the support function of $\alpha_n(S)$, thus leading to a contradiction.

\subsection*{Discussions and possible developments}
Recently several connections between  Teichm\"uller theory and the geometry of spacelike surfaces in the Anti-de Sitter space
have been exploited \cite{Mess, notes, bon_schl}. The key idea is that graphs of orientation-preserving homeomorphisms of $\partial\D$ are naturally curves in the asymptotic
 boundary of the Anti-de Sitter space spanned by spacelike surfaces. A natural construction allows to associate, to any convex surface $\Sigma$ spanning the
 graph of a homeomorphism $f:\partial \D\rar\partial\D$, a homeomorphism $F:\D\to\D$ extending $f$.
 It turns out that $F$ is quasi-conformal if and only if the principal curvatures of $\Sigma$ are bounded, showing that in this case $f$ is quasi-symmetric.
 
 Theorem \ref{thm:main1} can be regarded as an infinitesimal version of this property. 
 More generally we believe that the relation between the theory of convex surfaces in Minkowski space and the infinitesimal Teichm\"uller theory
 is an infinitesimal evidence of the above connection for Anti-de Sitter space.
 The correct geometric setting to understand this idea goes through the geometric transition introduced recently in \cite{danciger, dancigerideal, dancigertransition}.
In fact Danciger introduced the half-pipe geometry, which is a projective blow-up of a a spacelike geodesic plane in Anti-de Sitter space.
This model turns out to be a natural parameter space for spacelike affine planes in Minkowski space.
Regarding the Minkowski space as the blow-up of a point in Anti-de Sitter space, the correspondence between spacelike affine planes
in Minkowski space and points in half-pipe geometry is the infinitesimal version of the projective duality between points and spacelike planes in the Anti-de Sitter space.
Using this connection the graph of support functions $u:\D\to\R$ are naturally convex surfaces in half-pipe geometry.

 Let  $\Sigma_t$ be a smooth family of convex surfaces, for $t\in[0,\epsilon)$, such that $\Sigma_0$ is a totally geodesic plane and the principal curvatures
 of $\Sigma_t$ are $O(t)$. Using the correspondence above one can prove that the boundary of $\Sigma_t$ is the graph
 of a  differentiable family   $f_t$  of quasi-symmetric homeomorphisms of $\partial \D$ such that $f_0=\textrm{id}$. 
 
 The rescaled limit of $\Sigma_t$ in the half-pipe model is the graph of the support function of a convex surface $S$ in Minkowski space.
 It turns out that the principal curvatures of $S$ are the inverse of the derivatives of the principal curvatures of $\Sigma_t$.
 In particular they are bounded from below by a positive constant is the principal curvatures of $\Sigma_t$ are $O(t)$.
 Moreover the support function at infinity of $S$ corresponds to the vector field $\dot f$ on $\partial\D$. 
 
This gives an evidence of the fact that Theorem \ref{thm:main1} is a rescaled version of the correspondence between quasi-symmetric homeomorphisms
and convex surfaces in Anti de Sitter space with bounded principal curvatures.

We believe that even Theorem \ref{thm:main2} is a rescaled version in this sense of the existence of a  $K$-surface (actually, a foliation by $K$-surfaces, compare Theorem \ref{thm:main4}) with bounded principal curvature 
spanning a prescribed quasi-symmetric homeomorphism (for $K\in(-\infty,-1)$). 
However we leave this problem for a further investigation.
 
Finally we mention some questions still open in the Minkowski case, which are left for future developments.
\begin{itemize}
\item As already remarked, an interesting problem is to characterize complete spacelike entire graphs in Minkowski space in terms of the regularity of the boundary value
of the support function. From Theorem \ref{thm:main2} we know that the surface is complete if this regularity is Zygmund.
On the other hand it is not difficult  to construct an example of isometric embedding of $\Hyp^2$ in Minkowski space with unbounded principal curvatures,
for which the regularity of the support function will necessarily be weaker than Zygmund.
\item Another interesting question is to solve Minkowski problem for domains of dependence which are not contained in the future of a point.
This would include the case of domains of dependence whose support function is finite only on some subset  of $\overline{\D}$ which is obtained as the convex hull
of a subset $E$ of $\partial\D$.
It is simple to check that there is no solution for the constant curvature problem when $E$ contains $0$, $1$ or $2$ points (the corresponding domains are the whole space,
the future of a lightlike plane, or the future of a spacelike line).
An existence result in case where $E$ is an interval is given in \cite{schoenetal} with some  assumption on the smoothness of the support function on $E$.
In our opinion the construction of the support function in this setting is not difficult to generalize, but the barriers we use to prove that the corresponding surfaces
are entire graph seem to be ineffective.
\end{itemize}
 
 \subsection*{Organization of the paper}
In Section \ref{preliminaries} we give a short review of the different techniques used in the paper. We first recall the main definitions of the theory of future convex
sets in Minkowski space, introducing the support function and the generalized Gauss map. We relate those notions to Mess' description of domains of dependence
and dual laminations. Minkowski problem is then formulated in terms of Monge-Amp\`ere equations and we collect the main result used in the paper.
Finally we discuss the notion of infinitesimal earthquakes and how it is related to the infinitesimal theory of Teichm\"uller spaces.

In Section \ref{sec:convex existence} we solve Minkowski problem for domains of dependence contained in the future cone of a point, proving Theorem \ref{thm:main3}.
 In the following section we prove the existence of the foliation by constant curvature surfaces of the same domains, as in Theorem \ref{thm:main4}.
In Section \ref{sec:convex zygmund} we finally prove Theorems \ref{thm:main1} and \ref{thm:main2}.
Those theorems are collected in a unique statement, see Theorem \ref{big theorem zygmund finite dual lamination bounded curvatures}.

Finally in Appendix \ref{appendix parabolic} we construct constant curvature surfaces invariant by a parabolic group. The discussion leads to consider two classes of surfaces: the first class
(used in the proof of Theorem \ref{thm:main3}) includes surfaces which are entire graphs, whereas the second class gives surfaces which develop a lightlike ray.
Although we do not need surfaces of the latter type in this paper, we include a short description for completeness.
   
\subsection*{Acknowledgements}
The authors would like to thank Thierry Barbot, Fran{\c{c}}ois Fillastre and Jean-Marc Schlenker for their interest and encouragement during several fruitful discussions.

\clearpage

\section{Preliminaries and motivation} \label{preliminaries}

\subsection{Convex surfaces in Minkowski space} \label{preliminaries subsec convex surfaces}

We denote by $\R^{2,1}$ the $(2+1)$-dimensional Minkowski space, namely $\R^3$ endowed with the bilinear quadratic form
$$\langle x,y\rangle =x^1 y^1+x^2 y^2-x^3 y^3\,.$$
We denote by $\isom(\R^{2,1})$ the group of orientation preserving and time preserving isometries of $\R^{2,1}$. It is know that there is an isomorphism
$$\isom(\R^{2,1})\cong \SO_0(2,1)\rtimes \R^{2,1}$$
where $\SO(2,1)$ is the group of linear isometries of Minkowski product, $\SO_0(2,1)$ is the connected component of
the identity, and $\R^{2,1}$ acts on itself by translations.

A vector $v\in T_x \R^{2,1}\cong \R^{2,1}$ is called \emph{timelike} (resp. \emph{lightlike}, \emph{spacelike}) if $\langle v,v\rangle<0$ (resp. $\langle v,v\rangle=0$, $\langle v,v\rangle>0$). A plane $P$ is \emph{spacelike} (resp. \emph{lightlike}, \emph{timelike}) if its orthogonal vectors are timelike (resp. lightlike, spacelike).
An immersed differentiable surface $S$ in $\R^{2,1}$ is \emph{spacelike} if its tangent plane $T_x S$ is spacelike for every point $x\in S$. In this case, the symmetric 2-tensor induced on $S$ by the Minkowski product is a Riemannian metric. For instance, the hyperboloid 
$$\Hyp^2=\{x\in \R^{2,1}:\langle x,x\rangle = -1,x^3>0\}$$
is a spacelike embedded surface, and is indeed an isometric copy of hyperbolic space embedded in $\R^{2,1}$. 
An example of an embedded surface which is not spacelike - but that will still be important in the following - is the de Sitter space
$$\dS^2=\{x\in \R^{2,1}:\langle x,x\rangle = +1\}\,.$$
It is easy to see that $\dS^2$ parametrizes oriented geodesics of $\Hyp^2$. Indeed, for every point $x$ in $\dS^2$, the orthogonal complement $x^\perp$ is a timelike plane in $\R^{2,1}$, which intersects $
\Hyp^2$ in a complete geodesic, and coincides as a set with $(-x)^\perp$.

Given a point $x_0\in\R^{2,1}$, $\Ip(x_0)$ denotes the \emph{future cone} over $x_0$, namely the set of points $x$ of $\R^{2,1}$ which are connected to $x_0$ by a timelike differentiable path (namely, a path with timelike tangent vector at every point) along which the $x^3$ coordinate is increasing. 
It is easy to see that $\Ip(x_0)$ is a translate of
$$\Ip(0)=\{x\in\R^{2,1}:(x^1)^2+(x^2)^2<(x^3)^2,x^3>0\}\,.$$

It is clear from the definition that $\Hyp^2$ parametrizes spacelike linear planes in $\R^{2,1}$. 
Hence for every spacelike surface $S$, there is a well-defined \emph{Gauss map} $$G:S\rar\Hyp^2$$ which maps $x\in S$ to the normal of $S$ at $x$, i.e. the future unit timelike vector orthogonal to $T_x S$.

As we stated in the introduction, one of
 the aims of this paper is to classify isometric (or more generally homothetic) immersions of the hyperbolic plane into $\R^{2,1}$ which are contained in the future cone over some point.
Mess proved \cite{Mess} that if the first fundamental form of a spacelike immersion is complete, then the image of the immersion is a spacelike entire graph.
So we will deal with convex surfaces in $\R^{2,1}$ which are of the form $\{(x^1, x^2, x^3)\,|\,x^3=f(x^1,x^2)\}$, where $f:\R^2\rar\R$ is a convex function satisfying the spacelike condition
$||Df||^2<1$, where $Df$ is the Euclidean gradient of $f$.
Notice however that if $S$ is a spacelike entire graph in general it might not be complete.

It is convenient to extend the theory to the case of convex entire graphs which are not smooth and possibly contain lightlike rays.
Those correspond to convex functions $f:\R^2\rar\R$ such that $||Df||\leq 1$ almost everywhere. 
 We will extend the notion of Gauss map to this more general class.

A \emph{future-convex domain} in $\R^{2,1}$ is a closed convex set which is obtained as the intersection of future half-spaces bounded by spacelike planes. 
If $f:\R^2\rar\R$ is a convex function satisfying the condition $||Df||\leq 1$, then the epigraph of $f$ is a future-convex domain, and conversely the boundary of
any future-convex domain is the graph of a convex function as above.

A \emph{support plane} for a future-convex domain $D$ is a plane $P=y+x^\perp$ such that $P\cap \mathrm{int}(D)=\emptyset$ and every translate $P'=P+v$, for $v$ in the future of $x^\perp$, intersects $\mathrm{int}(D)$. A future-convex domain can admit spacelike and lightlike support planes. 
We define the \emph{spacelike boundary} of $D$ as the subset $$\partial_s D=\{p\in \partial D:p\text{ belongs to a spacelike support plane of }D\}\,.$$
It can be easily seen that $\partial D\setminus \partial_s D$ is a union of lightlike geodesic rays.
So $\partial D$ is a spacelike entire graph if and only if it does not contain lightlike rays.
 
We can now define the Gauss map for the spacelike boundary of a future-convex set. 
We allow the Gauss map to be set-valued, namely $$G(p)=\{x\in\Hyp^2:p+x^\perp\text{ is a support plane of }D \}\,.$$
By an abuse of notation, we will treat the Gauss map as a usual map with values in $\Hyp^2$. 
The following proposition (see \cite{minty1,minty2}) has to be interpreted in this sense.
\begin{prop}
Given a future-convex domain $D$ in $\R^{2,1}$, the Gauss map of $\partial_s D$ has image a convex subset of $\Hyp^2$. 
If $S$ is a strictly convex embedded spacelike surface, then its Gauss map is a homeomorphism onto its image.
\end{prop}

\subsection{The support function} \label{preliminaries subsec support}
We now introduce another important tool for this work, which is the Lorentzian analogue of the support function of Euclidean convex bodies. Roughly speaking, the support function encodes the information about the support planes of a future-convex domain. This is essentially done by associating to a unit future timelike (or lightlike) vector $x\in \Hyp^2$ the height of the support plane with normal unit vector $x$.

\begin{defi}
Given a future-convex domain $D$ in $\R^{2,1}$, the \emph{support function} of $D$ is the function $U:\overline{\Ip(0)}\rar \R\cup\{\infty\}$ defined by
$$U(x)=\sup_{p\in D}\langle p,x\rangle\,. $$
\end{defi}

The following is an immediate property of support functions.

\begin{lemma} \label{lemma disuguaglianza funzioni supporto}
Let $D_1$ and $D_2$ be future-convex domains with support functions $U_1$ and $U_2$. Then $U_1\geq U_2$ if and only if $D_1\subseteq D_2$.
\end{lemma}

It is clear from the definition that $U$ is 1-homogeneous, namely $U(\lambda x)=\lambda U(x)$ for every $\lambda>0$. Moreover, $U$ is lower semicontinuous, since it is defined as the supremum of continuous functions. 
It is straightforward to check that, given an isometry in $\isom(\R^{2,1})$ with linear part $A$ and translation part $t$,
 the support function of $D'=AD+t$ is
 \begin{equation} \label{transf rule support function}
 U'(x)=U(A^{-1}x)+\langle x,t\rangle
 \end{equation}

We will mostly consider the restriction of $U$ to the Klein projective model of hyperbolic space, which is the disc
$$\D=\{(z,1)\in\R^{2,1}:|z|<1\}\,.$$
This restriction will be denoted by lower case letters,
$u=U|_{\D}$, and uniquely determines the 1-homogeneous extension $U$. We will generally write $u(z)$ instead of $u(z,1)$. 
Analogously, also the restriction of $U$ to the hyperboloid, denoted $\bar u=U|_{\Hyp^2}$, can be uniquely extended to a 1-homogeneous function, and will be often used in the following. 

\begin{remark}
It is easy to relate the restrictions $u$ and $\bar u$ of the support function to $\D$ and $\Hyp^2$ respectively. Let us consider the radial projection $\pi:\Hyp^2\rar\D$ defined by 
$$\pi(x^1,x^2,x^3)=(x^1/x^3,x^2/x^3,1)\,.$$
Its inverse is given by
$$\pi^{-1}(z,1)=\left(\frac{z}{\sqrt{1-|z|^2}},\frac{1}{\sqrt{1-|z|^2}}\right)\,.$$
Since $U$ is 1-homogeneous, we obtain
$$u(z)=\sqrt{1-|z|^2}\,\bar u(\pi^{-1}(z))\,.$$
\end{remark}

A 1-homogeneous convex function is called \emph{sublinear}.

\begin{lemma}[{\cite[Lemma 2.21]{Fillastre}}]
Given a future-convex domain $D$ in $\R^{2,1}$, the support function $U:\overline{\Ip(0)}\rar\R$ is sublinear and lower semicontinuous. 
Conversely, given a sublinear function $\hat U$ on $\Ip(0)$ (or equivalently every convex function $u$ on $\D$), consider the lower semicontinuous extension $U:\overline{\Ip(0)}\rar\R\cup\{+\infty\}$, which is defined on $\partial\Ip(0)$ as $$U(x)=\liminf_{y\to x}\hat U(y)\,.$$
Then $U$ is the support function of a future-convex domain, defined by
$$D=\{p\in \R^{2,1}:\langle p,x\rangle\leq U(x)\text{ for every }x\in\Ip(0)\}\,.$$
\end{lemma}


The support function of a future-convex domain $D$ is finite on the image of the Gauss map of $\partial_s D$, since for every point $x$ in $G(\partial_s D)$ there exists a support plane with normal vector $x$. Observe that $\bar u(x)=+\infty$ if $x\in \Hyp^2\setminus \overline{G(\partial_s D)}$. We will call \emph{support function at infinity} the restriction of $U$ to $\partial\D=\{(z,1):|z|=1\}$. Given $z\in\partial\D$, $u(z)<+\infty$ if and only if there exists a lightlike support plane $P$ orthogonal to the lightlike vector $(z,1)$. In this case $-u(z)$ is the intercept of $P$ on the $x^3$-axis. 

\begin{example}
The support function of the hyperboloid $\Hyp^2$ is $U(x)=-\sqrt{|\langle x,x\rangle|}$. Hence its restriction $\bar u:\Hyp^2\to\R$ is constant, $\bar u\equiv -1$. The support function at infinity is finite, $u|_{\partial\D}\equiv 0$, by an easy computation. For some less elementary examples, see Remark \ref{remark counterexample constant curvature not entire} and Appendix \ref{appendix parabolic}.

In this paper, we are mostly concerned with domains of dependence for which the support function at infinity is finite, and is actually bounded. Geometrically, this means that the domain is contained in the future cone over some point. 
\end{example}

The following lemma will be useful to compute the value of support functions at infinity.

\begin{lemma}[{\cite[Theorem 7.4,7.5]{rockafellar}}] \label{lemma limit radial geodesics}
Let $U:\overline{\Ip(0)}\rar\R$ be a sublinear and lower semincontinuous function. Let $c:[0,1]\to \overline{\Ip(0)}$ be a spacelike line such that $x=c(1)\in\partial\Ip(0)$. Then $U(x)=\lim_{t\to1}U(c(t))$.
\end{lemma}

Some of the geometric invariants of $\partial_s D$ can be directly recovered from the support function. This is the content of next lemma. Recall that, given a $C^2$ embedded spacelike surface $S$ in $\R^{2,1}$, its \emph{shape operator} is a $(1,1)$-tensor which can be defined as 
$$B(v)=\nabla_v N\,,$$
where $\nabla$ is the Levi-Civita connection of the first fundamental form of $\R^{2,1}$, $N$ is the future unit normal vector field, and $v$ is any tangent vector in $T_p S$. The Gauss Theorem in this Lorentzian setting gives the following relation between the intrinsic curvature $K$ of the first fundamental form and its shape operator, which holds for every point $p\in S$:
$$K=-\det B\,.$$
Finally, we define the \emph{hyperbolic Hessian} of a function $\bar u:\Hyp^2\rar \R$ as the $(1,1)$-tensor
$$\Hess \bar u(v)=\nabla^{\Hyp^2}_v\grad\bar u\,,$$
where $\nabla^{\Hyp^2}$ is the Levi-Civita connection of $\Hyp^2$. We denote by $D^2 u$ the Euclidean Hessian of a function $u$ defined on an open subset of $\R^2$. In the following, the identity operator is denoted by $\id$.

\begin{lemma}[{\cite[§2.10, 2.13]{Fillastre}}] \label{lemma formulae minkowski}
Let $D$ be a future-convex domain in $\R^{2,1}$ and let $G:\partial_s D\rar\Hyp^2$ be its Gauss map. 
\begin{itemize}

\item If the support function is $C^1$, then the intersection of $\partial_s D$ with any spacelike support plane consists of exactly one point.
The inverse of the Gauss map $G$ is well defined and is related to the support function of $D$ by the formula
\begin{equation}
G^{-1}(x)=\grad \bar u(x)-\bar u(x)x\,, \label{formula inverse gauss map}
\end{equation}
where $\bar u:\Hyp^2\rar \R$ is the support function restricted to $\Hyp^2$. 

\item If the support function is $C^2$ and the operator $\Hess \bar u-\bar u\id$ is positive definite, then $\partial_s D$ is a convex $C^2$-surface. The inverse of its shape operator and its curvature are
\begin{gather}
B^{-1}=\Hess \bar u-\bar u\id\,, \label{formula shape operator euclidean hessian}
\\
 -\frac{1}{K(G^{-1}(x))}=\det (\Hess \bar u-\bar u\id)(x)=(1-|z|^2)^2\det D^2u (z)\,,
\end{gather}
where $z=\pi(x)$ is the point of $\D$ obtained from $x$ by radial projection.
\end{itemize}
\end{lemma}
We will often abuse notation and write $G^{-1}(z)$ in place of $G^{-1}(\pi^{-1}(z))$ for $z\in\D$.

\subsection{The boundary value of the support function of an entire graph}
Let $S$ be the boundary of a future-convex domain in $\R^{2,1}$. Denote by $f:\R^2\to\R$ the function defining $S$ as a graph and $u:\D\to\R$ the support function.
We want to show that 
\[
    \lim_{r\to\infty} (f(rz)-r)=-u(z)    
\]
for every unitary vector $z\in\partial\D$. This clarifies that the asymptotic conditions defined for instance in \cite{treibergs, choitreibergs} coincide with those considered here and in \cite{schoenetal, Li}.

First consider the case where $D=\Ip(p)=graph(f_p)$ where $p=(w_0, a_0)$ and 
$f_p(z)=|z-w_0|+a_0$. In this case the support function is $u_p(z)=\langle z, w_0\rangle-a_0$. A simple computation shows
\begin{align*}
   f_p(rz)-r=&\sqrt{|w_0|^2-2r\langle z, w_0\rangle+r^2 }-r+a_0\\ =&
   \frac{-2r\langle z, w_0\rangle+|w_0|^2}{\sqrt{|w_0|^2-2r\langle z, w_0\rangle+r^2 }+r} +a_0
   \longrightarrow -\langle z, w_0\rangle+a_0=-u_p(z)\,.
 \end{align*}
 
Now consider the general case.
Imposing that the point $(rz,f(rz))$ lies in the future of the support plane
$\{q\in\R^{2,1}\,|\,\langle q, (z,1)\rangle=u(z)\}$ we get
 $f(rz)-r\geq -u(z)$. So it is sufficient to prove that $\limsup (f(rz)-r)\leq -u(z)$.
 
Fix $\epsilon>0$ and consider the lightlike plane $P=\{q\in\R^{2,1}\,|\,\langle q, (z,1)\rangle =u(z)-\epsilon\}$.
This plane must intersect the future of $S$. Let $p=(w_0, a_0)$ be a point in this intersection.
The cone $\Ip(p)$ is contained in the future of $S$, hence $f_p\geq f$, where $f_p$ is the graph function
for $\Ip(p)$ as above.

In particular, using the computation above for $\Ip(p)$,
$$\limsup_{r\to+\infty} (f(rz)-r)\leq \lim_{r\to+\infty} (f_p(rz)-r)=-\langle z, w_0\rangle+a_0\,.$$
Imposing that $p$ lies on the plane $P$, $$-\langle z, w_0\rangle+a_0=-\langle (z,1), p\rangle=-u(z)+\epsilon\,.$$
Therefore, for any $\epsilon>0$, $$\limsup_{r\to+\infty} (f(rz)-r)\leq-u(z)+\epsilon$$
and this concludes our claim, since $\epsilon$ is arbitrary.

\subsection{Cauchy surfaces and domains of dependence}
Given a future-convex domain $D$ in $\R^{2,1}$, a \emph{Cauchy surface} for $D$ is a spacelike embedded surface $S\subseteq D$ such that every differentiable inextensible causal path in $D$ (namely, such that its tangent vector is either timelike or lightlike at every point) intersects $S$ in exactly one point. Given an embedded surface $S$ in $\R^{2,1}$, the maximal future-convex domain $D(S)$ such that $S$ is a Cauchy surface for $D(S)$ is the \emph{domain of dependence} of $S$. It turns out that $D(S)$ is obtained as intersection of future half-spaces bounded by lightlike planes which do not disconnect $S$.

It is easy to prove the following lemma.
\begin{lemma} \label{remark support functions and Cauchy}
Let $h:\overline\D\to\R$ be the support function of a future-convex domain $D$, with $h|_{\partial\D}<\infty$. Let $S\subseteq D$ be a convex embedded surface and let $u:\overline\D\to\R$ be the support function of $S$. Then $S$ is a Cauchy surface for $D$ if and only if $h|_{\partial \D}=u|_{\partial \D}$.
\end{lemma}


Domains of dependence can be characterized in terms of the support function, see \cite[Proposition 2.21]{bonfill}.

\begin{lemma}
Let $D$ be a domain of dependence in $\R^{2,1}$, whose lightlike support planes are determined by the function $\varphi:\partial\D\to\R\cup\{\infty\}$. Then the support function $h:\overline\D\to\R$ of $D$ is the convex envelope $h=\co(\varphi)$, namely:
$$h(z)=\sup\{f(z):f\text{ is an affine function on }\D,f|_{\partial\D}\leq\varphi\}\,.$$
\end{lemma}

A useful example of support functions of Cauchy surfaces can be obtained by looking at the leaves of the cosmological time of the domain of dependence. Observe that a \emph{timelike distance} can be defined for two points $x_1$ and $x_2\in \Ip(x)$ in $\R^{2,1}$, by means of the definition
$$d(x_1,x_2)=\sup_{\gamma} \int_{\gamma}\sqrt{|\langle\gamma'(t),\gamma'(t)\rangle|}dt\,,$$ 
where the supremum is taken over all causal paths $\gamma$ from $x_1$ to $x_2$. This is not a distance though, because it satisfies a reverse triangle inequality; however, $d(x_1,x_2)$ is achieved along the geodesic from $x_1$ to $x_2$. Given an embedded spacelike surface $S$, consider the equidistant surface $$S_d=\{x\in\R^{2,1}:x\in \Ip(S), d(x,S)=d\}\,,$$
where of course $d(x,S)=\sup_{x'\in S}d(x,x')$. If the support function of $S$ restricted to $\Hyp^2$ is $\bar u$, then $S_d$ has support function (see for instance \cite{Fillastre})
$$\bar u_d(x)=\bar u(x)-d\,.$$
This can be applied also for $\partial_s D$, instead of an embedded surface. In this way, we obtain the level sets of the \emph{cosmological time}, namely the function $T:D\to\R$ defined by
$$T(x)=\sup_{\gamma}\int_{\gamma}\sqrt{|\langle\gamma'(t),\gamma'(t)\rangle|}dt\,,$$
where the supremum is taken over all causal paths $\gamma$ in $D$ with future endpoint $x$. If $\bar h:\Hyp^2\to\R$ is the support function of $D$, the level sets $L_d=\{T=d\}$ of the cosmological time have support function on the disc $\bar h_d(x)=\bar h(x)-d$. It can be easily seen that all leaves of the cosmological time of $D$ are Cauchy surfaces for $D$ (although only $C^{1,1}$). Indeed, the support functions $h_d:\D\to\R$ can be computed:
$$h_d(z)=h(z)-d\sqrt{1-|z|^2}\,.$$
Therefore they all agree with $h$ on $\partial\D$. 

\subsection{Dual lamination}

In this paper, we will adopt the following definition of measured geodesic lamination. The equivalence with the most common definition is discussed for instance in \cite{saricweak}. Let $\mathcal{G}$ be the set of (unoriented) geodesics of $\Hyp^2$. The space $\mathcal{G}$ is identified to $((S^1\times S^1)\setminus diag)/\sim$ where the equivalence relation is defined by $(a,b)\sim(b,a)$. Note that $\mathcal{G}$ has the topology of an open M\"{o}bius strip. Given a subset $B\subset \Hyp^2$, we denote by $\mathcal{G}_B$ the set of geodesics of $\Hyp^2$ which intersect $B$.

\begin{defi} \label{defi mgl}
A geodesic lamination on $\Hyp^2$ is a closed subset of $\mathcal{G}$ such that its elements are pairwise disjoint geodesics of $\Hyp^2$. A measured geodesic lamination is a locally finite Borel measure on $\mathcal{G}$ such that its support is a geodesic lamination.
\end{defi}

A measured geodesic lamination is called \emph{discrete} if its support is a discrete set of geodesics. A measured geodesic lamination $\mu$ is \emph{bounded} if
$$\sup_I \mu(\mathcal{G}_I)<+\infty\,,$$
where the supremum is taken over all geodesic segments $I$ of lenght at most 1 transverse to the support of the lamination.
The \emph{Thurston norm} of a bounded measured geodesic lamination is 
$$||\mu||_{Th}=\sup_I \mu(\mathcal{G}_I)\,.$$

Elements of a geodesic lamination are called \emph{leaves}. \emph{Strata} of the geodesic lamination are either leaves or connected components of the complement of the geodesic lamination in $\Hyp^2$.




In his groundbreaking work, Mess associated a domain of dependence $D$ to every measured geodesic lamination $\mu$, in such a way that the support function $h:\D\to\R$ of $D$ is linear on every stratum of $\mu$. Although we do not enter into details here, the measure of $\mu$ determines the bending of $h$. (Recall $h$ is the convex envelope of some lower semicontinuous function $\varphi:\partial\D\to\R$.) The domain $D$ is determined up to translation in $\R^{2,1}$, and $\mu$ is called \emph{dual lamination} of $D$.

Given $y_0\in \R^{2,1}$ and $x_0\in\Hyp^2$ 
, we will denote by $D(\mu,x_0,y_0)$ the domain of dependence having $\mu$ as dual laminations and $P=y_0+x_0^\perp$ as a support plane tangent to the boundary at $y_0$.

We sketch here the explicit construction of $D(\mu,x_0,y_0)$. In the following, given the oriented geodesic interval $[x_0,x]$ in $\Hyp^2$, $\boldsymbol{\sigma}:\mathcal{G}[x_0,x]\rar\R^{2,1}$ is the function which assigns to a geodesic $l$ (intersecting $[x_0,x]$) the corresponding point in $\dS^2$, namely, the spacelike unit vector in $\R^{2,1}$ orthogonal to $l$ for the Minkowski product, pointing outward with respect to the direction from $x_0$ to $x$.
Then 
\begin{equation}
y(x)=y_0+\int_{\mathcal{G}[x_0,x]} \boldsymbol{\sigma} d\mu \label{eq: boundary domain mess}
\end{equation}
is a point of the regular boundary of $D(\mu,x_0,y_0)$ such that $y(x)+x^\perp$ is a support plane for  $D(\mu,x_0,y_0)$, for every $x\in\Hyp^2$ such that the expression in Equation \eqref{eq: boundary domain mess} is integrable. The image of the Gauss map of the regular boundary of  $D(\mu,x_0,y_0)$ is composed precisely of those $x\in\Hyp^2$ which satisfy this integrability condition.


In the following proposition, we give an explicit expression for the support function of the domain of dependence $D(\mu,x_0,y_0)$ we constructed.
By an abuse of notation, given two points $x_0,x\in\Ip(0)$, we will denote by $[x_0,x]$ the geodesic interval of $\Hyp^2$ obtained by projecting to the hyperboloid $\Hyp^2\subset\R^{2,1}$ the line segment from $x_0$ to $x$. 

\begin{prop} \label{support function from lamination}
Suppose $D$ is a domain of dependence in $\R^{2,1}$ with dual lamination $\mu$ and such that the plane $P=y_0+(x_0)^{\perp}$ is a support plane for $D$, for $y_0\in\partial_s D$ and $x_0\in \Hyp^2$. Then the support function $H:\overline{\Ip(0)}\rar\R$ of $D$ is:
\begin{equation} \label{support function formula}
H(x)=\langle x,y_0\rangle+\int_{\mathcal{G}[x_0,x]} \langle x,\boldsymbol{\sigma}\rangle d\mu\,.
\end{equation}
\end{prop}
Indeed, the expression in Equation \eqref{support function formula} holds for $x\in\Hyp^2$ by Equation \eqref{eq: boundary domain mess}. Since the expression is 1-homogeneous, it is clear that it holds for every $x\in\Ip(0)$. Using Lemma \ref{lemma limit radial geodesics}, the formula holds also for the lower semicontinuous extension to $\partial \Ip(0)$.

It is easily seen from Equation \eqref{support function formula} that the support function $h:\D\to\R$ (which is the restriction of $H$ to $\D$) is affine on each stratum of $\mu$. In \cite{Mess, bebo} it was proved that every domain of dependence can be obtained by the above construction. Hence a dual lamination is uniquely associated to every domain of dependence.





The work of Mess (\cite{Mess}) mostly dealt with domains of dependence which are invariant under a discrete group of isometries $\Gamma<\isom(\R^{2,1})$, whose linear part is a cocompact Fuchsian group. We resume here some results.


\begin{prop} \label{proposition translation part mess}
Let $D$ be a domain of dependence in $\R^{2,1}$ with dual lamination $\mu$. The measured geodesic lamination $\mu$ is invariant under a cocompact Fuchsian group $G$ if and only if $D$ is invariant under a discrete group $\Gamma<\isom(\R^{2,1})$ such that the projection of $\Gamma$ to $\SO(2,1)$ is an isomorphism onto $G$. In this case, assuming $P=y_0+(x_0)^{\perp}$ is a support plane for $D$, for $y_0\in\partial_s D$ and $x_0\in \Hyp^2$, the translation part of an element $g\in G$ is:
$$t_g=\int_{\mathcal{G}[x_0,g(x_0)]}  \boldsymbol{\sigma} d\mu\,.$$
\end{prop}

\subsection{Monge-Amp\`ere equations}
Given a smooth strictly convex spacelike surface $S$ in $\R^{2,1}$, let $U:\Ip(0)\to\R$ be the support function of $S$ and let $u$ be its restriction to $\D=\Ip(0)\cap\left\{x^3=1\right\}$. Given a point $z\in D$, let $x=\pi^{-1}(z)\in\Hyp^2$. The curvature of $S$ is given by (see Lemma \ref{lemma formulae minkowski})
$$-\frac{1}{K(G^{-1}(x))}=(1-|z|^2)^2\det D^2u(z)\,,$$
where $G:S\rar\Hyp^2$ is the Gauss map, which is a diffeomorphism. For $K$-surfaces, namely surfaces with constant curvature equal to $K\in(-\infty,0)$, the support function satisfies the Monge-Amp\`ere equation
\begin{equation} \label{monge ampere constan curvature}
\det D^2u(z)=\frac{1}{|K|}(1-|z|^2)^{-2}\,.
\end{equation}

More generally, the Minkowski problem consists of finding a convex surface with prescribed curvature function on the image of the Gauss map. Given a smooth function $\psi:\D\to\R$, the support function of a surface with curvature $K(G^{-1}(z))=-\psi(z)$ solves the Monge-Amp\`ere equation
\begin{equation} 
\det D^2u(z)=\frac{1}{\psi(z)}(1-|z|^2)^{-2} \tag{\ref{monge ampere constan curvature}}
\end{equation}

We review here some key facts of Monge-Amp\`ere theory. Given a convex function $u:\Omega\rar\R$ for $\Omega$ a convex domain in $\R^2$, we define the normal mapping of $u$ as the set-valued function $N_u$ whose value at a point $\bar w\in\Omega$ is:
$$N_u(\bar w)=\left\{Df:f\text{ affine; }graph(f)\text{ is a support plane for }graph(u),(\bar w,u(\bar w))\in graph(f)\right\}\,.$$

In general $N_u(\bar w)$ is a convex set; if $u$ is differentiable at $\bar w$, then $N_u(\bar w)=\left\{Du(\bar w)\right\}$. We define the Monge-Amp\`ere measure on the collection of Borel subsets $\omega$ of $\R^2$:
$$M\!A_u(\omega)=\mathcal{L}(N_u(\omega))$$
where $\mathcal{L}$ denote the Lebesgue measure on $\R^2$.

\begin{lemma}[{\cite[Lemma 2.3]{trudwang}}]
If $u$ is a $C^2$ function, then $$M\!A_u(\omega)=\mathcal{L}(Du(\omega))=\int_\omega (\det D^2 u)d\mathcal{L}\,.$$
In general, $\int_\omega (\partial^2 u) d\mathcal{L}$ is the regular part of the Lebesgue decomposition of $M\!A_u(\omega)$, where we set
$$\partial^2 u(\bar w)=\begin{cases} \det D^2_{\bar w} u & \text{if }u\text{ is twice-differentiable at }\bar w \\ 0 & \text{otherwise} \end{cases}\,.$$
\end{lemma}

\begin{defi}
Given a nonnegative measure $\nu$ on $\Omega$, we say a convex function $u:\Omega\rar\R$ is a generalized solution to the Monge-Amp\`ere equation
\begin{equation} \label{general monge ampere}
\det D^2u=\nu
\end{equation}
if $M\!A_u(\omega)=\nu(\omega)$ for all Borel subsets $\omega$. In particular, given an integrable function $f:\Omega\rar\R$, $u$ is a generalized solution to the equation $\det D^2u=f$ if and only if, for all $\omega$, $$M\!A_u(\omega)=\int_\omega fd\mathcal{L}\,.$$
\end{defi}

We collect here, without proofs, some facts which will be used in the following. Unless explicitly stated, the results hold in $\R^n$, although we are only interested in $n=2$. Recall that, by Aleksandrov Theorem, a convex function $u$ on $\Omega$ is twice-differentiable almost everywhere.

\begin{lemma}[{\cite[Lemma 2.2]{trudwang}}] \label{convergence of solutions}
Given a sequence of convex functions $u_n$ which converges uniformly on compact sets to $u$, the Monge-Amp\`ere measure $M\!A_{u_n}$ converges weakly to $M\!A_{u}$. 
\end{lemma}

\begin{theorem}[Comparison principle, \cite{trudwang,gutierrez}] \label{comparison principle}
Given a bounded convex domain $\Omega$ and two convex functions $u,v$ defined on $\overline\Omega$, if $M\!A_u(\omega)\leq M\!A_v(\omega)$ for every Borel subset $\omega$, then
$$\min_{\overline\Omega}(u-v)=\min_{\partial\Omega}(u-v)\,.$$  
\end{theorem}

\begin{cor}
Given two generalized solutions $u_1,u_2\in C^{0}(\overline \Omega)$ to the Monge-Amp\`ere equation $\det D^2 (u_i)=\nu$ on a bounded convex domain $\Omega$, if $u_1\equiv u_2$ on $\partial\Omega$, then $u_1\equiv u_2$ on $\Omega$.
\end{cor}

\begin{theorem}[{\cite[Lemma 3]{chengyau}}, \cite{pogorelov}] \label{monge ampere boundedness second derivative}
Given a bounded convex domain $\Omega$, let $u$ be a $C^4$ solution to $\det D^2u=f$ defined on $\overline\Omega$ which is constant on $\partial\Omega$. There is an estimate on the second derivatives of $u$ at $x\in \Omega$ which depends only on $$\max_{\Omega}\left\{|u|,||Du||^2,||D\log(f)||^2,\sum_{i,j}\partial_{ij}(\log(f))^2\right\}$$
and on the distance of $x$ to $\partial\Omega$.
\end{theorem}

The following property will be used repeatedly in the paper, and is a peculiar property of dimension $n=2$.

\begin{theorem}[Aleksandrov-Heinz] \label{solution strictly convex dimension 2}
A generalized solution to $\det D^2 u =f$ on a domain $\Omega\subset\R^2$ with $f>0$ must be strictly convex.
\end{theorem}

We will use this theorem to prove an a priori estimate of the $C^2$-norm of  support functions of surfaces of constant curvature in Minkowski space in terms
of the $C^0$-norm. Although this result is  well-known we sketch a proof for the convenience of the reader.

\begin{lemma} \label{lemma boundedness second derivatives}
Let $u_n:\D\rar\R$ be a sequence of smooth solutions of the Monge-Amp\`ere equation $$\det D^2 (u_n)=\frac{1}{|K|}(1-|z|^2)^{-2}$$ uniformly bounded on $\D$. 
Then $||u_n||_{C^2(\Omega)}$ is uniformly bounded on any compact domain $\Omega$ contained in $\D$.
\end{lemma}
\begin{proof}
Assume that the conclusion is false and that there is a subsequence (which we still denote by $u_n$ by a slight abuse of notation) for which the $C^2$-norm goes to infinity. Hence it suffices to show that there exists a further subsequence $u_{n_k}$ for which $||u_{n_k}||_{C^2(\Omega)}$ is bounded. 
Take $\Omega'$ such that $\Omega\subset\!\subset\Omega' \subset\!\subset\D$.
Using the uniform bound on $||u_n||_{C^0(\D)}$ and the convexity, one can derive that the $C^1$-norms $||u_n||_{C^1(\overline\Omega')}$ are uniformly bounded by a constant $C$. 
By Ascoli-Arzel\`a theorem, we can extract a subsequence which converges uniformly on compact subsets of $\Omega'$. Let $u_\infty$ be the limit function. By Lemma \ref{convergence of solutions}, $u_\infty$ is a generalized solution to 
$$\det D^2(u_\infty)=\frac{1}{|K|}(1-|z|^2)^{-2}$$
and $u_\infty$ is strictly convex, by Theorem \ref{solution strictly convex dimension 2}. 

For any $z\in\Omega$ and $n\geq 0$ we can fix an affine function $f_{n,z}$ such that $v_{n,z}=u_n+f_{n,z}$ takes its minimum at $z$ and $v_{n,z}(z)=0$.
We claim that there are $\epsilon_0>0$ and $r_0>0$ such that
\begin{itemize}
\item   $\min_{\partial\Omega'} v_{n,z}\geq 2\epsilon_0$ for any $n\geq 0$ and $z\in\Omega$.
\item   $\max_{B(z, r_0)}v_{n,z}\leq \epsilon_0$ for any $n\geq 0$ and $z\in\Omega$.
\end{itemize}

First let us show how the claim implies the statement.
Indeed for any $z$ and $n$ consider the domain $U_{n,z}=\{z\in\Omega'\,|\,v_{n,p}(z)\leq \epsilon_0\}$.
We have that $U_{n,z}\subset\!\subset\Omega'$ by the first point of the claim. In particular, $v_{n,z}$ is constant equal to $\epsilon_0$ along
the boundary of $U_{n,z}$. On the other hand the second point of the claim implies that the distance of $z$ from $\partial U_{n,z}$ is at least $r_0$.
So by Theorem \ref{monge ampere boundedness second derivative} there is a constant $C'$ depending on $C$ and $r_0$ such that
$||D^2 u_n(z)||=||D^2 v_{n,z}(z)||<C'$ for all $z\in\Omega'$ and $n\geq 0$.

To prove the claim we argue by contradiction.
Suppose there exist sequences $z_n, z'_n\in\Omega$ such that, defining $2\epsilon_n=\min_{\partial\Omega'} v_{n,z_n}$,
\begin{itemize}
\item $||z_n-z'_n||\to 0$;
\item $v_{n, z_{n}}(z'_n)> \epsilon_n$. 
\end{itemize}
Up to passing to a subsequence we may suppose that $z_n\to z_\infty$, so that $z'_n\to z_\infty$ as well.
As the $C^1$-norm of $u_n$ is bounded, the $C^1$-norm of $f_{n,z}$ is uniformly bounded for any $z\in\Omega$ and $n\geq 0$,
so we may suppose that $f_{n,z_{n}}$ converges to an affine function $f_\infty$. Therefore $v_{n, z_n}$ converges to $v_\infty= u_\infty+f_\infty$. Since, as already observed, $u_\infty$ is strictly convex, $v_\infty$ is strictly convex as well.

As $\lim v_{n, z_{n}}(z'_n)=v_\infty(z_\infty)=\lim v_{n, z_{n}}(z_n)= 0$ we conclude that $\epsilon_n\to 0$, so that $\min_{\partial\Omega'} v_\infty=0$. This gives a 
contradiction with the strict convexity of $v_\infty$, and thus concludes the proof.
\end{proof}

A refinement of the above arguments leads to the proof of the regularity of strictly convex solutions of Monge-Amp\`ere equation.

\begin{theorem}[{\cite[Theorem 3.1]{trudwang}}] \label{solution smooth}
Let $u$ be a strictly convex generalized solution to $\det D^2 u=f$ on a bounded convex domain $\Omega$ with smooth boundary. If $f>0$ and $f$ is smooth, then $u$ is smooth.
\end{theorem}

\subsection{Universal Teichm\"{u}ller space and Zygmund fields} \label{zygmund fields}

In this section, we want to introduce the notion of Zygmund class, which will be the relevant boundary regularity for the support functions of convex surfaces with bounded principal curvatures. Since support functions are defined as 1-homogeneous functions, we first show that a vector field on $S^1$ defines a 1-homogeneous function on the boundary of the null-cone in a natural way.

We consider the boundary at infinity $\partial_\infty\Hyp^2$ of $\Hyp^2$, as $\mathbb{P}(N)\cong S^1$, where $N=\partial\Ip(0)\setminus\{0\}$. In particular, we will use vector fields on $S^1$ to define 1-homogeneous functions on $\partial\Ip(0)$. We want to show that this is well-defined, i.e. does not depend on the choice of a section $S^1\to N$.
\begin{lemma} \label{lemma campo e funzione omogenea}
There is a $1$-to-$1$ correspondence between vector fields $X$ on $S^1$ and $1$-homogeneous functions
$H:N\rar\R$ satisfying the following property:
if $\gamma:S^1\rar N$ is any $C^1$ spacelike section of the projection $N\rar S^1$ and $v$ is  the unit tangent vector field to $\gamma$, then
\begin{equation} \label{definizione campo e funzione omogenea}
\gamma_*(X(\xi))=H(\gamma(\xi))v(\gamma(\xi))\,.
\end{equation}
\end{lemma}
\begin{proof}
Consider coordinates $(x^1,x^2,x^3)$ on $\R^{2,1}$ and $z,\theta$ on $N$ given by $$\phi:(r,\theta)\rar(r\cos\theta,r\sin\theta,r)\in N.$$ In these coordinates, the restriction of the Minkowski metric to $N$ takes the (degenerate) form 
\begin{equation} \label{metrica cono} 
g=r^2 d\theta^2
\end{equation}
We take $\gamma_1$ to be the section $\gamma_1(\theta)=(1,\theta)$. Namely, the image of $\gamma_1$ is $N\cap\{x^3=1\}$. Any other section $\gamma_2$ is of the form $\gamma_2(\theta)=(r(\theta),\theta)$ and is obtained as $\gamma_2=f\circ\gamma_1$ by a radial map $f(1,\theta)=(r(\theta),\theta)$. 
Let $X$ be a vector field on $S^1$. 
We define a $1$-homogeneous function $H$ such that $(\gamma_1)_*(X(\theta))=H(1,\theta)v_1$ and compute 
$$(\gamma_2)_*(X(\theta))=f_*(H(1,\theta)v_1)=H(1,\theta)f_*(v_1).$$
Now $f_*(v_1)$ is a tangent vector to $\gamma_2(S^1)$ whose norm (recall the form (\ref{metrica cono}) of the metric) is $r(\theta)$.
Therefore $$(\gamma_2)_*(X(\theta))=H(1,\theta)r(\theta)v_2=H(r(\theta),\theta)v_2$$ where $v_2$ is the unit tangent vector. Conversely, given any 1-homogeneous function, (\ref{definizione campo e funzione omogenea}) defines a vector field on $S^1$ which does not depend on the choice of $\gamma$.
\end{proof}

We now introduce a class of vector fields on $S^1$ whose regularity is specially interesting for this article, namely the Zygmund fields. To do so, we first introduce the notion of earthquake of $\Hyp^2$ and quasi-symmetric homeomorphisms of the circle. A Zygmund field is a vector field which corresponds to an infinitesimal deformation of quasi-symmetric homeomorphisms, in a suitable sense.

\begin{defi} \label{definizione terremoto}
A surjective map $E:\Hyp^2\to\Hyp^2$ is a left earthquake if it is an isometry on the strata of a geodesic lamination of $\Hyp^2$ and, for every pair of strata $S$ and $S'$, the composition
$$(E|_S)^{-1}\circ(E|_{S'})$$
is a hyperbolic translation whose axis weakly separates $S$ and $S'$ and such that $S'$ is translated on the left as seen from $S$. 
\end{defi}

Let us observe that $E$ is in general not continuous. Given an earthquake $E$, there is a measured geodesic lamination associated to $E$, called the \emph{earthquake measure}. See \cite{thurstonearth}. The earthquake measure $\mu$ determines $E$ up to post-composition with an hyperbolic isometry (in other words, up to the choice of the image of one stratum), and up to the ambiguity on the weighted leaves of the lamination. Hence an earthquake whose earthquake measure is $\mu$ will be denoted by $E^\mu$.

The following theorem was proved by Thurston \cite{thurstonearth}.

\begin{theorem}
Any earthquake $E:\Hyp^2\to\Hyp^2$ extends to an orientation-preserving homeomorphism of $S^1=\partial_\infty\Hyp^2$. Conversely, every orientation-preserving homeomorphism of $S^1$ is induced by a unique earthquake of $\Hyp^2$. 
\end{theorem}

We are now going to present a characterization of earthquakes whose measured geodesic lamination is bounded in terms of their extension to the boundary at infinity, due to Gardiner-Hu-Lakic in  \cite{garhulakic}. 
Given an orientation preserving homeomorphism $T:S^1\rar S^1$, we consider a lifting of  $\tilde T:\R\rar\R$ to the universal covering so that
\[
    T(e^{i\theta})=e^{i\tilde T(\theta)}
\]


\begin{defi}
An orientation-preserving homeomorphism $T:S^1\to S^1$ is quasi-symmetric if   there is a constant $C$ such that
\begin{equation}\label{cross-ratio norm}
\frac{1}{C}<\left|\frac{\tilde T(\theta+h)-\tilde T(\theta)}{\tilde T(\theta)-\tilde T(\theta-h)}\right|< C\,,
\end{equation}
for all $\theta, h\in\R$.
\end{defi}


The space of quasi-symmetric homeomorphisms of $S^1$, up to the action of M\"obius transformation by post-compositions, is called \emph{Universal Teichm\"uller space}. It is endowed with a smooth (actually, complex) structure, see for instance \cite[§16]{gardiner2}. 
We are now going to discuss briefly the tangent space to Universal Teichm\"uller space. 
We will use the notation $\hat\varphi$ for a vector field on $S^1$ and $\varphi$ for the function from $S^1$ to $\R$ which corresponds to $\hat\varphi$ under the standard trivialization of $TS^1$.
In the following definition, we regard tangent vectors to $S^1$ as elements of $\C$. Hence, $\hat\varphi(z)=iz\varphi(z)$ for every $z\in\partial\D$.

\begin{defi}
A function $\varphi:S^1\rar\R$ is in the Zygmund class if
there is a constant $C$ such that
\begin{equation}\label{cross-ratio infinitesimal}
|\varphi(e^{i(\theta+h)})+\varphi(e^{i(\theta-h)})-2\varphi(e^{i\theta})|<C|h|
\end{equation}
for all $\theta, h\in\R$.
A vector field $\hat\varphi$ on $S^1$ is a Zygmund field if 
the associated function $\varphi$ is in the Zygmund class.
\end{defi}

Functions in the Zygmund class are $\alpha$-H\"older for any $\alpha\in(0,1)$, but in general they are not Lipschitz.
The vector space of Zygmund fields, quotiented by the subspace of vector fields which are extensions on $S^1$ of Killing vector fields on $\Hyp^2$, is precisely the tangent space at the identity of Universal Teichm\"uller space (\cite{gardiner2}).

\begin{example} \label{infinitesimal earthquake one leaf}
Let $\mu$ be the measured geodesic lamination whose support consists of a single geodesic $l$, with weight $1$. Then, once a point $x_0\in\Hyp^2\setminus l$ is fixed, it is easy to describe the earthquake along $\mu$:
\begin{equation}
E_l ([\eta])=\begin{cases}
[\eta] & \text{if }x_0\text{ and }[\eta]\text{ are in the same component of }(\Hyp^2\cup\partial_\infty\Hyp^2)\setminus \bar l \\
[A^l(1)(\eta)] & \text{otherwise}
\end{cases}
\end{equation}
for any $\eta\in N$, where $A^l(t)\in\SO(2,1)$ induces the hyperbolic isometry of $\Hyp^2$ which translates on the left (as seen from $x_0$) along the geodesic $l$ by lenght $t$.

Hence the 1-homogeneous function $H$ associated to the Zygmund field $\dot E _l=\ddt E_{t l}$ (as in Lemma \ref{lemma campo e funzione omogenea}) has the following expression, for any section $\gamma:S^1\to N$:
\begin{equation}
H(\gamma(\xi))=
\begin{cases}
0 & \text{if }x_0\text{ and }[\gamma(\xi)]\text{ are in the same component of }(\Hyp^2\cup\partial_\infty\Hyp^2)\setminus \bar l \\
\langle\dot A^l(\gamma(\xi)),v(\gamma(\xi))\rangle & \text{otherwise}
\end{cases}
\end{equation}
where $v$ is any unit spacelike tangent vector field to $N$, in the counterclockwise orientation. Under the standard identification of $S^1$ with $\partial\D$, we obtain, for $\eta \in\partial \D$
$$\dot E _l(\eta)=\langle\dot A^l(\eta),v\rangle v\,.$$
where $v$ is now the unit tangent vector to $\partial\D$.
\end{example}

We will use the following result by Gardiner-Hu-Lakic, see \cite{garhulakic} or \cite[Appendix]{saricweak}.

\begin{theorem} \label{teorema saric}
Given a bounded measured geodesic lamination $\mu$ and a fixed point $x_0$ which does not lie on any weighted leaf of $\mu$, the integral
\begin{equation} \label{integral formula infinitesimal earthquake}
\dot E^\mu(\eta)=\int_{\mathcal G}\dot E_l(\eta)d\mu(l)
\end{equation}
converges for every $\eta\in \partial\D$ and defines a Zygmund field $\hat\varphi$ on $S^1$, which corresponds to the infinitesimal earthquake
$$\dot E^\mu=\ddt E^{t\mu}\,.$$
Conversely, for every Zygmund field $\hat\varphi$ on $S^1=\partial\D$, there exists a bounded measured geodesic lamination $\mu$ such that $\hat\varphi$ is the infinitesimal earthquake along $\mu$, namely $$\hat\varphi=\ddt E^{t\mu}\,,$$
up to an infinitesimal M\"obius transformation.
\end{theorem}

Analogously to the case of earthquakes, although the infinitesimal earthquake $\dot E^\mu$ is not continuous in $\Hyp^2$, its boundary value is a continuous field. Theorem \ref{teorema saric} can be regarded as the infinitesimal version of the Theorem (\cite{garhulakic}, anticipated by Thurston) which states that every quasi-symmetric homeomorphism is the extension on the boundary of an earthquake with bounded measured geodesic lamination.

\section{The Minkowski problem in Minkowski space} \label{sec:convex existence}


The aim of this section is to prove that, for every domain of dependence in $\R^{2,1}$ contained in the cone over a point, there exists a unique smooth Cauchy surface with prescribed (\`a la Minkowski) negative curvature, which is an entire graph. Equivalently, the main statement is the following.

\begin{theorem} \label{theorem minkowski problem lsc}
Given a bounded lower semicontinuous function $\varphi:\partial\D\rar \R$ and a smooth function $\psi:\D\to[a,b]$ for some $0<a<b<+\infty$, there exists a unique smooth spacelike surface $S$ in $\R^{2,1}$ with support function at infinity $\varphi$ and curvature $K(G^{-1}(x))=-\psi(x)$. Moreover, $S$ is an entire graph and is contained in the past of the $(1/\sqrt{\inf{\psi})}$-level surface of the cosmological time of the domain of dependence with support function $h=\co(\varphi)$.

\end{theorem}

The proof will be split in several steps. In Subsection \ref{subsec existance} we construct a solution to the Monge-Amp\`ere equation 
\begin{equation} \label{monge ampere constan curvature}
\det D^2u(z)=\frac{1}{\psi(z)}(1-|z|^2)^{-2} \tag{MA}
\end{equation}
\noindent with the prescribed boundary condition at infinity
\begin{equation} \label{monge ampere boundary condition}
u|_{\partial\D}=\varphi\,. \tag{BC}
\end{equation}
 In Subsection \ref{cauchy and uniqueness} we study the behavior of  Cauchy surfaces in terms of the support functions, and we use this condition to prove uniqueness by applying the theory of Monge-Amp\`ere equations. Finally, in Subsection \ref{no light rays} we prove that the surface is not tangent to the boundary of the domain of dependence, and hence is a spacelike entire graph.

\subsection{Existence of solutions} \label{subsec existance}

The surface $S$ will be obtained as a limit of surfaces $S_\Gamma$ invariant under the action of discrete groups $\Gamma<\isom(\R^{2,1})$, isomorphic to the fundamental group of a closed surface, acting freely and properly discontinuously on some future-convex domain in $\R^{2,1}$ for which $S_\Gamma$ is a Cauchy surface. Indeed, such a surface $S_\Gamma$ can be obtained as the lift to the universal cover of a closed Cauchy surface $S_\Gamma/\Gamma$ in a maximal globally hyperbolic spacetime $D(S_\Gamma)/\Gamma$, and the existence of surfaces with prescribed curvature in such spacetimes is guaranteed by results of Barbot-B\'eguin-Zeghib in \cite{barbotzeghib}. 

In this subsection we prove the following existence result for the Monge-Amp\`ere equation \eqref{monge ampere constan curvature}.

\begin{theorem} \label{theorem existence lsc}
Given a bounded lower semicontinuous function $\varphi:\partial\D\rar \R$ and a smooth function $\psi:\D\to[a,b]$ for some $0<a<b<+\infty$, there exists a smooth solution $u:\D\rar\R$ to the equation
\begin{equation}
\det D^2u(z)=\frac{1}{\psi(z)}(1-|z|^2)^{-2}  \tag{\ref{monge ampere constan curvature}}
\end{equation}
such that $u$ extends to a lower semicontinuous function on $\overline \D$ with 
\begin{equation} 
u|_{\partial\D}=\varphi\,.  \tag{\ref{monge ampere boundary condition}}
\end{equation}
Moreover, $u$ satisfies the inequality
\begin{equation} \label{inequality support functions constant curvature surface}
h(z)-C\sqrt{1-|z|^2}\leq u(z)\leq h(z)\,, \tag{CT}
\end{equation}
where $h$ is the convex envelope of $\varphi$
and $C=1/\sqrt{\inf\psi}$.
\end{theorem}


There are several notions of convergence of measured geodesic laminations, as discussed for instance in \cite{saricweak}. Recall in Definition \ref{defi mgl} we defined a measured geodesic lamination as a locally finite Borel measure on the set of (unoriented) geodesics $\mathcal{G}$ of $\Hyp^2$, with support a closed set of pairwise disjoint geodesics. 

\begin{defi}
A sequence $\{\mu_n\}_n$ of measured geodesics laminations converges in the weak* topology to a measured geodesic lamination, $\mu_n\war\mu$, if
$$\lim_{n\rar\infty}\int_{\mathcal{G}}f d\mu_n=\int_{\mathcal{G}}f d\mu$$
for every $f\in C_0^0(\mathcal{G})$.
\end{defi}

A stronger notion of convergence is given by the Fr\'echet topology on the space of measured geodesic laminations. This is defined by interpreting a measured geodesic lamination as a linear functional on the space of $\alpha$-H\"older functions (for every $\alpha\in(0,1)$) with compact support. The Fr\'echet topology hence comes from a family of $\alpha$-seminorms obtained by considering the supremum of the linear functional on $\alpha$-H\"older functions supported in a box of geodesics $Q=[a,b]\times[c,d]$, where $(a,b,c,d)$ is a symmetric quadruple of points in $\partial\D$. See \cite{saricweak} for more details.

We are going to approximate a measured geodesic lamination in the weak* topology by measured geodesic laminations which are invariant under the action of a cocompact Fuchsian group. 

\begin{lemma} \label{fuchsian weak approximation}
Given a measured geodesic lamination $\mu$, there exists a sequence of measured geodesic laminations $\mu_n$ such that $\mu_n$ is invariant under a torsion-free cocompact Fuchsian group $G_n<\isom(\Hyp^2)$ and $\mu_n\war \mu$.
\end{lemma}
\begin{proof} Recall that $\mathcal{G}_B$ denotes the set of geodesics of $\Hyp^2$ which intersect the subset $B\subseteq \Hyp^2$. We construct the approximating sequence in several steps.

\textit{Step 1.} We show there is a sequence of discrete measured geodesic laminations $\mu_n$ which converge to $\mu$ in the weak* topology. In \cite[§7]{saricweak} it was proved that, if $\mu$ is bounded, there exists a sequence of discrete measured geodesic laminations which converge to $\mu$ in the Fr\'echet topology, which implies weak* convergence. So, assume $\mu$ is not bounded. We define $\nu_n$ by $\nu_n(A)=\mu(A\cap\mathcal{G}_{B(0,n)})$, i.e. the support of $\nu_n$ consists of the geodesics of $\mu$ which intersect $B(0,n)$. By the results in \cite[§7]{saricweak}, for every $n$, there exists a sequence $(\nu_{n,m})_m$ which converges to $\nu_n$ in the Fr\'echet sense. As a consequence of Fr\'echet convergence, for every $n$ we can find $m=m(n)$ such that 
$$\sup_{f\in C_0^0(\mathcal{G}_{B(0,n)})}\left|\int_\mathcal{G} fd\nu_n-\int_\mathcal{G} fd\nu_{n,m(n)}\right|\leq\frac{1}{n}\,.$$
It follows that, for every $f$ compactly supported in $\mathcal{G}$, if $supp(f)\subset \mathcal{G}_{B(0,n_0)}$, then for $n\geq n_0$
$$\left|\int_\mathcal{G} fd\mu-\int_\mathcal{G} fd\nu_{n,m(n)}\right|=\left|\int_\mathcal{G} fd\nu_n-\int_\mathcal{G} fd\nu_{n,m(n)}\right|\xrightarrow{n\to\infty} 0\,.$$
Hence $\mu_n:=\nu_{n,m(n)}$ gives the required approximation.

\textit{Step 2.} We now modify the sequence $\mu_n$ to obtain a sequence $\mu'_n\war\mu$ of finite measured laminations with ultraparallel geodesics. This step is necessary because we would like to construct a fundamental polygon $P_n$ for a cocompact Fuchsian group, so that the edges of $P_n$ intersect the geodesics of $\mu_n$ orthogonally. We will choose edges of $P_n$ so as to ``separate'' the endpoints of the geodesics of $\mu_n$. However, for this purpose one needs that the geodesics of $\mu_n$ are ultraparallel, and thus we first need to modify $\mu_n$ to a new sequence $\mu_n'$.  We can assume the discrete laminations $\mu_n$ constructed in Step 1 are finite (namely they consist of a finite number of weighted geodesics), by taking the intersection with $\mathcal{G}_{B(0,n)}$. Let $d_{\mathcal{G}}$ be the distance induced by a Riemannian metric on $\mathcal{G}$. Suppose the leaves of $\mu_n$ are $l_1^n,\ldots,l_{p(n)}^n$ with weights $a_1^n,\ldots,a^n_{p(n)}$. Then we construct a finite lamination $\mu'_{n}$ by replacing $l_1^n,\ldots,l_{p(n)}^n$ by leaves $k_1^n,\ldots,k_{p(n)}^n$ so that
\begin{itemize}
\item $k_i^n$ and $k_j^n$ are ultraparallel for every $i\neq j$;
\item $d_{\mathcal{G}}(l^n_i,k^n_i)\leq 1/n$;
\item The weight of $k^n_i$ is $a_i^n$.
\end{itemize}
Let us show that $(\mu'_n)_n$ converges weak* to $\mu$. For this purpose, fix a function $f$ with $supp(f)\subset \mathcal{G}_{B(0,n_0)}$ for some $n_0$. Fix $\epsilon>0$. Since $f$ is uniformly continuous, there exists $n_1$ such that if $d_{\mathcal{G}}(l,k)<1/n_1$, then $|f(l)-f(k)|<\epsilon$.
We have
$$\left|\int_\mathcal{G}fd\mu-\int_\mathcal{G}fd\mu_n'\right|\leq\left|\int_\mathcal{G}fd\mu-\int_\mathcal{G}fd\mu_n\right|+\left|\int_\mathcal{G}fd\mu_n-\int_\mathcal{G}fd\nu'_{n,n}\right|\,.$$
By construction, there exists $n_2$ such that the first term in the RHS is smaller than $\epsilon$ provided $n\geq n_2$. Now for every $n$, if $m\geq\max\left\{n_0,n_1\right\}$, 
$$\left|\int_\mathcal{G}fd\mu_n-\int_\mathcal{G}fd\nu'_{n,m}\right|=\sum_{i=1}^{p(n)}\left(f(l^n_i)-f(k^n_i)\right)a_i^n\leq \epsilon \mu_n(\mathcal{G}_{(B(0,n_0))})\,.$$
Since $\mu_n\war\mu$, there exists a constant $C$ such that $\mu_n(\mathcal{G}_{(B(0,n_0))})\leq C$ for $n\geq n_3$. In conclusion, if $n\geq \max\left\{n_0,n_1,n_2,n_3\right\}$, then 
$$\left|\int_\mathcal{G}fd\mu-\int_\mathcal{G}fd\mu_n'\right|\leq(1+C)\epsilon\,.$$

\textit{Step 3.} We claim it is possible to find a polygon $P_n$ with the following properties:
\begin{itemize}
\item $P_n$ contains the ball $B(0,n)$;
\item The angles of $P_n$ are $\pi/2$;
\item $P_n$ intersects the leaves of $\mu'_n$ orthogonally.
\end{itemize}

\begin{figure}[htb]
\centering
\begin{minipage}[c]{.45\textwidth}
\centering
\includegraphics[height=5.5cm]{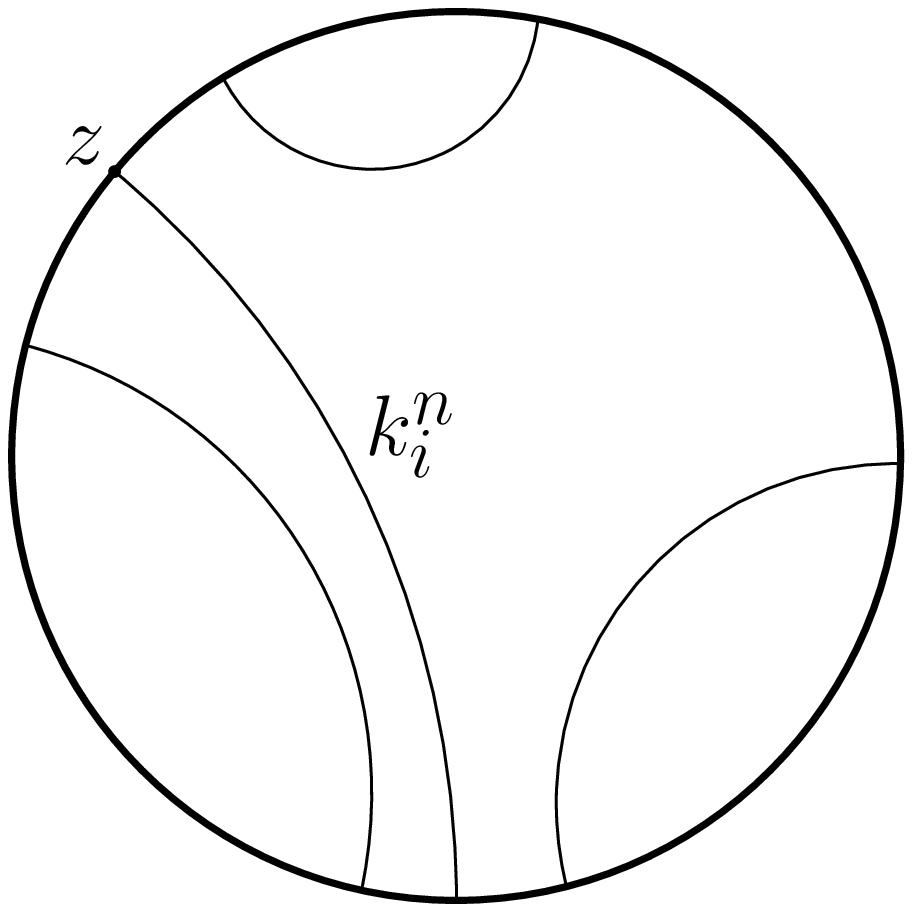}
\caption{Start from a finite geodesic laminations with leaves $k_i^n$.} \label{fig:mgl1}
\end{minipage}%
\hspace{5mm}
\begin{minipage}[c]{.45\textwidth}
\centering
\includegraphics[height=5.5cm]{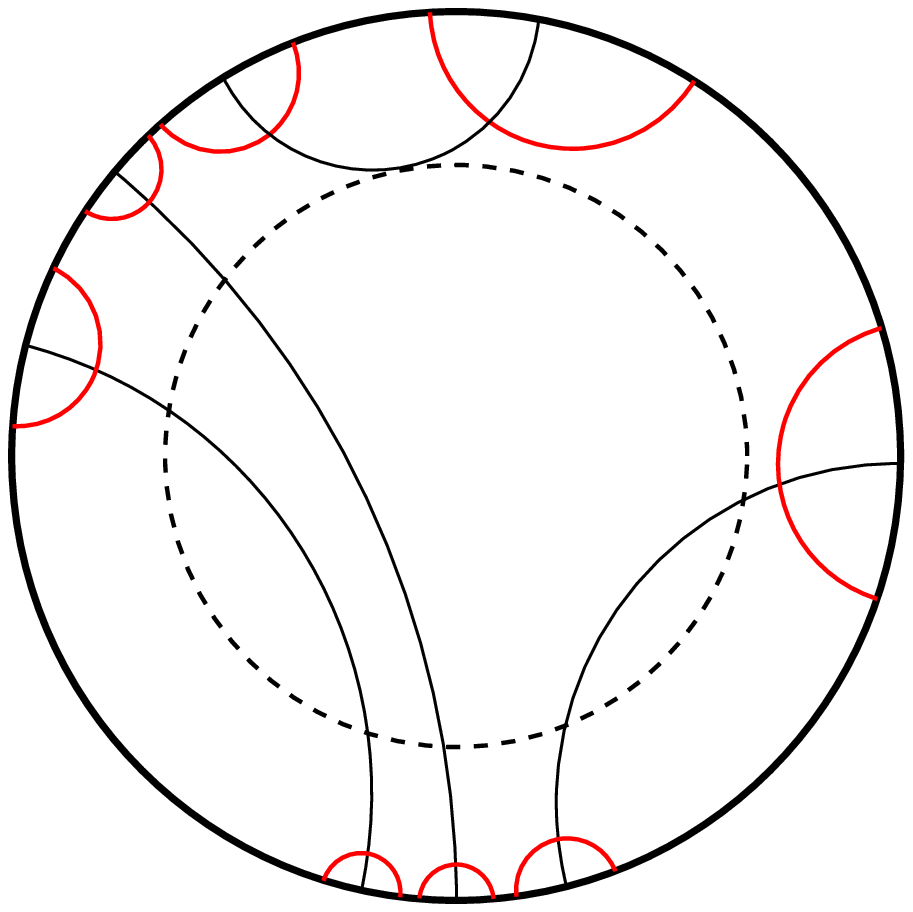}
\caption{The construction of the geodesics $g_1,\ldots,g_p$.} \label{fig:mgl2}
\end{minipage}
\end{figure}

We construct the polygon $P_n$ in the following way. For every point $z\in\partial\D$ which is limit of a leaf $k^n_i$ of $\mu_n'$, we pick a geodesic orthogonal to $k^n_i$ which separates $z$ from $B(0,n)$ and from all the other limit points of $\mu_n'$. Let $\left\{g_1,\ldots,g_p\right\}$ be the geodesics obtained in this way. Replacing the $g_i$ by other geodesics further from $B(0,n)$, we can assume the geodesics $g_1,\ldots,g_p$ are pairwise ultraparallel. See Figures \ref{fig:mgl1} and \ref{fig:mgl2}.

 We can now extend the family of geodesics $\{g_1,\ldots,g_p\}$ to a larger family $\{g'_1,\ldots,g'_{p'}\}$ satisfying:
\begin{itemize}
\item $\left\{g_1,\ldots,g_p\right\}\subset\{g'_1,\ldots,g'_{p'}\}$;
\item $g'_1,\ldots,g'_{p'}$ are contained in $\Hyp^2\setminus B(0,n)$
\item $g'_1,\ldots,g'_{p'}$ are pairwise ultraparallel;
\item No geodesic $g'_i$ separates two geodesics $g'_j$ and $g_k'$ in the family (so we can assume the indices in $\{g'_1,\ldots,g'_{p'}\}$ are ordered counterclockwise);
\item The geodesics $h_i$ orthogonal to $g'_i$ and $g'_{i+1}$ (if the indices $i$ are considered $\mathrm{mod} \,p'$) are contained in $\Hyp^2\setminus B(0,n)$.
\end{itemize}

\begin{figure}[htb]
\centering
\begin{minipage}[c]{.45\textwidth}
\centering
\includegraphics[height=5.5cm]{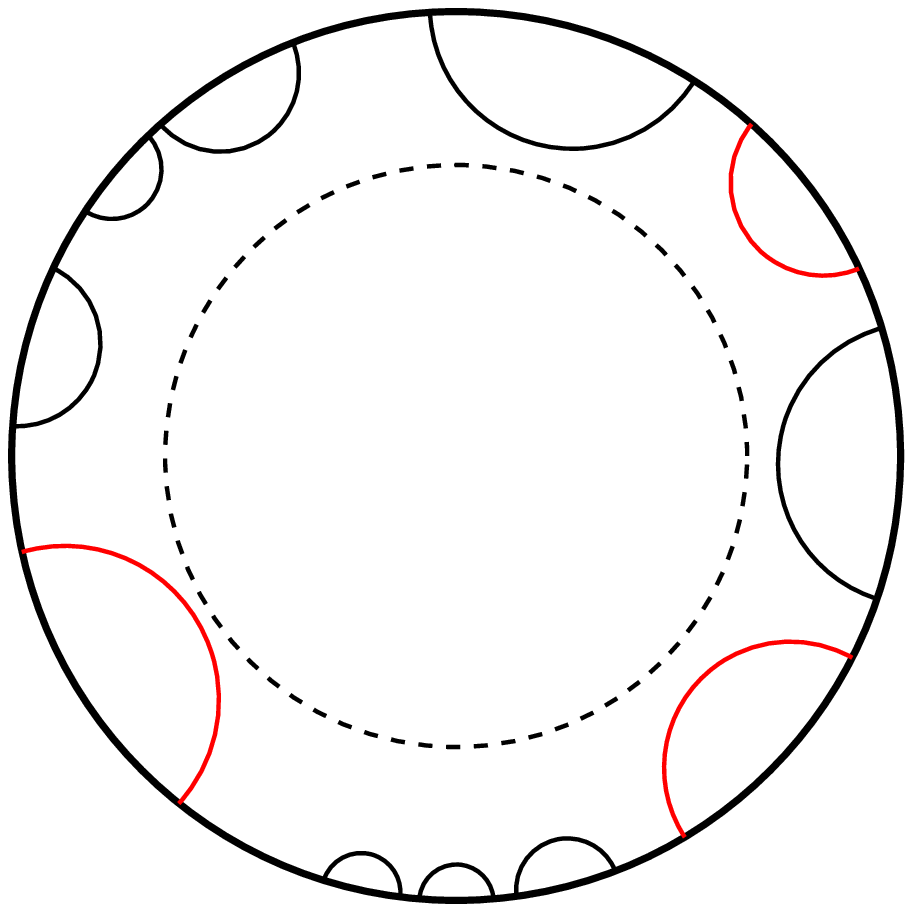}
\caption{The geodesics in $\left\{g'_1,\ldots,g'_{p'}\right\}\setminus \{g_1,\ldots,g_p\}$.} \label{fig:mgl3}
\end{minipage}%
\hspace{5mm}
\begin{minipage}[c]{.45\textwidth}
\centering
\includegraphics[height=5.5cm]{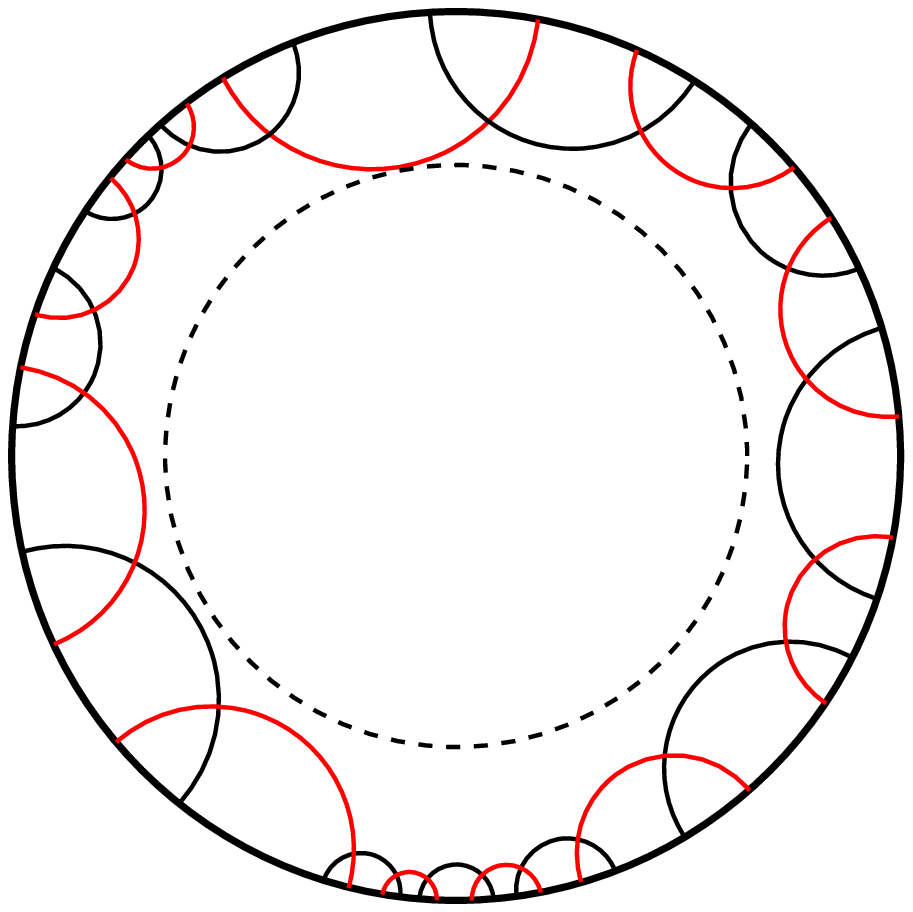}
\caption{The geodesics $h_1,\ldots,h_{p'}$ orthogonal to $g'_1,\ldots,g'_{p'}$.} \label{fig:mgl4}
\end{minipage}
\end{figure}

The reader can compare with Figures \ref{fig:mgl3} and \ref{fig:mgl4}. It is clear that the polygon $P_n$ (Figure \ref{fig:mgl5}) given by the connected component of $\Hyp^2\setminus\{g'_1,\ldots,g'_{p'},h_1,\ldots,h_{p'}\}$ containing $B(0,n)$ satisfies the given properties.

\begin{figure}[htb]
\centering
\includegraphics[height=5.5cm]{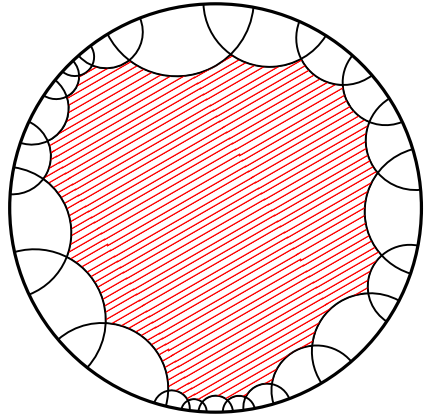}
\caption{The polygon $P_n$. \label{fig:mgl5}}
\end{figure}

\textit{Step 4.} We finally construct a sequence $\mu''_n\war\mu$ with $\mu''_n$ invariant under the action of a torsion-free cocompact Fuchsian group. We consider the discrete group of isometries of $\Hyp^2$ generated by reflections in the sides of the polygon $P_n$ constructed in Step 3. The index 2 subgroup $\overline G_n$ of orientation-preserving isometries  is a discrete cocompact group and it is well-known that $\overline G_n$ contains a finite index torsion-free cocompact Fuchsian group $G_n$. We define $\mu_n''$ as the $G_n$-orbit of $\mu_n'$. It turns out that $\mu_n''$ is a measured geodesic lamination since the leaves of $\mu_n'$ intersect the sides of $P_n$ orthogonally.
Note that $\mu''_n$ is obtained by modifying $\mu'_n$ only in the complement of $B(0,n)$, since by construction $B(0,n)\subset P_n$. Hence it is clear that, if $n_0$ is such that $supp(f)\subset \mathcal{G}_{B(0,n_0)}$, then $\int_\mathcal{G}fd\mu_n'=\int_\mathcal{G}fd\mu_n''$ for $n>n_0$, and thus $\mu''_n\war \mu$. 
\end{proof}

\begin{lemma} \label{limit weak uniform compact}
Given a sequence $\mu_n$ of measured geodesic laminations converging to $\mu$ in the weak* sense, let $D_n=D(\mu_n,x_0,y_0)$ and $D=D(\mu,x_0,y_0)$ be the domains of dependence having $\mu_n$ and $\mu$ respectively as dual laminations, and $P=y_0+x_0^\perp$ as a support plane tangent to the boundary at $y_0$, where $x_0$ does not belong to a weighted leaf of $\mu$. Let $h_n$ and $h$ be the support functions of $D_n$ and $D$. Then $h_n$ converges uniformly on compact sets of $\D$ to $h$.
\end{lemma}
\begin{proof}
It suffices to prove that the convergence is pointwise, since the functions $h_n$ and $h$ are convex on $\D$, and therefore pointwise convergence implies uniform convergence on compact sets. In fact, it suffices to prove pointwise convergence for almost every point.

We will actually prove that the support functions $\bar h_n$ restricted to $\Hyp^2$ converge pointwise to $\bar h$ almost everywhere, which is clearly equivalent to the claim. Let us assume $x_0$ and $x$ are points which do not lie on weighted leaves of the lamination $\mu$. Let $\mathcal{G}[x_0,x]$ be the set of geodesics which intersect the closed geodesic segment $[x_0,x]$ and  $\boldsymbol{\sigma}:\mathcal{G}[x_0,x]\rar\R^{2,1}$ be the function which assigns to a geodesic $l$ the corresponding point in $\dS^2$, namely, the spacelike unit vector in $\R^{2,1}$ orthogonal to $l$ with respect to Minkowski product, pointing outward with respect to the direction from $x_0$ to $x$. The support function of $D$ can be written as (compare expression (\ref{support function formula}) in Proposition \ref{support function from lamination}):
$$\bar h_n(x)=\langle x,y_0\rangle+\int_{\mathcal{G}[x_0,x]} \langle x,\boldsymbol{\sigma}\rangle d\mu_n\,.$$
We thus want to show that
$$\int_{\mathcal{G}[x_0,x]} \langle x,\boldsymbol{\sigma}\rangle d\mu_n\xrightarrow{n\to\infty} \int_{\mathcal{G}[x_0,x]} \langle x,\boldsymbol{\sigma}\rangle d\mu\,.$$
Note that $\mathcal{G}[x_0,x]$ is compact in $\mathcal{G}$; define $\varphi_i$ a smooth function such that $\varphi_i(l)=1$ for every $l\in\mathcal{G}[x_0,x]$ and $supp(\varphi_i)\subset\mathcal{G}(x_0-\frac{1}{i},x+\frac{1}{i})$. Here $(x_0-\frac{1}{i},x+\frac{1}{i})$ denotes the open geodesic interval which extends $[x_0,x]$ of a lenght $1/i$ on both sides. Hence we have:
\begin{align}
\left| \int_{\mathcal{G}} \langle x,\boldsymbol{\sigma}\rangle \chi_{\mathcal{G}[x_0,x]}d\mu_n- \int_{\mathcal{G}} \langle x,\boldsymbol{\sigma}\rangle \chi_{\mathcal{G}[x_0,x]} d\mu \right| & \leq  \left| \int_{\mathcal{G}} \langle x,\boldsymbol{\sigma}\rangle \chi_{\mathcal{G}[x_0,x]}d\mu_n- \int_{\mathcal{G}} \langle x,\boldsymbol{\sigma}\rangle \varphi_i d\mu_n \right| \tag{$\star$} \label{1 term} \\
&+ \left| \int_{\mathcal{G}} \langle x,\boldsymbol{\sigma}\rangle \varphi_i d\mu_n- \int_{\mathcal{G}} \langle x,\boldsymbol{\sigma}\rangle \varphi_i d\mu \right| \tag{$\star\star$} \label{2 term} \\
&+\left| \int_{\mathcal{G}} \langle x,\boldsymbol{\sigma}\rangle \varphi_i d\mu- \int_{\mathcal{G}} \langle x,\boldsymbol{\sigma}\rangle \chi_{\mathcal{G}[x_0,x]} d\mu \right|\,. \tag{$\star\star\star$} \label{3 term}
\end{align}
Let $\mathcal{F}=\overline{\mathcal{G}(x_0-\frac{1}{i},x+\frac{1}{i})\setminus \mathcal{G}[x_0,x]}$. Since we have assumed $x$ and $x_0$ are not on weighted leaves of $\mu$, we have
$\eqref{3 term}\leq K\mu(\mathcal{F})\leq K\epsilon$
if $i\geq i_0$, for some fixed $i_0$.
By definition of weak* convergence, the term numbered (\ref{2 term}) converges to zero as $n\to\infty$ for $i=i_0$ fixed, so $\eqref{2 term}\leq\epsilon$ for $n\geq n_0$. Finally, $\limsup_{n\to\infty}\mu_n(\mathcal{F})\leq \mu(\mathcal{F})$ by Portmanteau Theorem (which in this case can be easily proved again by an argument of enlarging the interval and approximating by bump functions). Hence there exists $n_0'$ such that $\eqref{1 term}\leq K\mu_n(\mathcal{F})\leq 2K\epsilon$ if $n\geq n_0'$ and $i\geq i_0$. Choosing $n\geq\max\left\{n_0,n_0'\right\}$, the proof is concluded.
\end{proof}

Let us now consider an arbitrary measured geodesic lamination $\mu$ and take the sequence $\mu_n\war\mu$ constructed as in Lemma \ref{fuchsian weak approximation}. Let $D(\mu_n,x_0,y_0)$ be the domain of dependence having dual lamination $\mu_n$ and $P=y_0+x_0^\perp$ as a support plane tangent at $y_0$. Since $\mu_n$ is invariant under the action of a Fuchsian cocompact group $G_n$, $D(\mu_n,x_0,y_0)$ is a domain of dependence invariant under a discrete group $\Gamma_n$. The linear part of $\Gamma_n$ is $G_n$ and the translation part is determined (up to conjugacy) by $\mu_n$ (see Proposition \ref{proposition translation part mess}). This means that $D(\mu_n)/\Gamma_n$ is a maximal globally hyperbolic flat spacetime.

\begin{theorem}[\cite{barbotzeghib}] \label{theorem bbz}
Let $G_0$ be a Fuchsian cocompact group and $\mu_0$ be a $G_0$-invariant measured geodesic lamination. Let $\psi_0:\Hyp^2\rar(0,\infty)$ be a $G_0$-invariant smooth function. Let $\Gamma_0$ be a subgroup of $\isom(\R^{2,1})$ whose linear part is $G_0$ and whose translation part is determined by $\mu_0$. Then there exists a unique smooth Cauchy surface $S_0$ of curvature $K(G^{-1}(x))=-\psi_0(x)$ in the maximal future-convex domain of dependence $D_0$ invariant under the action of the group $\Gamma_0$.
\end{theorem}

We will construct a Cauchy surface $S$ for $D=D(\mu,x_0,y_0)$ with prescribed curvature as a limit of Cauchy surfaces $S_n$ in $D_n=D(\mu_n,x_0,y_0)$. Let $h_n$ be the support function of $D_n$. By Lemma \ref{limit weak uniform compact}, $h_n$ converges uniformly on compact sets of $\D$ to $h$, the support function of $D$. Let $u_n$ be the support function of $S_n$.

\begin{lemma} \label{lemma inequality support functions}
Let $S_0$ be a smooth strictly convex Cauchy surface in a domain of dependence $D_0$ invariant under the action of a discrete group $\Gamma_0<\isom(\R^{2,1})$, such that $D_0/\Gamma_0$ is a maximal globally hyperbolic flat spacetime. Let $K:S_0\rar(-\infty,0)$ be the curvature function of $S_0$. Then the support functions $u_0$ of $S_0$ and $h_0$ of $\partial_s D_0$ satisfy
\begin{equation} \label{inequality support functions}
h_0(z)-C\sqrt{1-|z|^2}\leq u_0(z)\leq h_0(z)
\end{equation}
\noindent for every $z\in\D$. Moreover, one can take $C=1/\sqrt{\inf|K|}$.
\end{lemma}
\begin{proof}
Since $S_0\subset \Ip(\partial_s D_0)$, it is clear that $u_0\leq h_0$. For the converse inequality, recall that we denote by $\bar u_0$ the restriction to $\Hyp^2$ of the 1-homogeneous extension $U_0$ of $u_0$, and analogously for $\bar h_0$. Let us consider $x\in\Hyp^2$ and show that $\bar u_0(x)\geq \bar h_0(x)-C$, for $C=1/\sqrt{\inf|K|}$. The inequality (\ref{inequality support functions}) then follows, since $u_0(z)=U_0(z,1)=\sqrt{1-|z|^2}\bar u_0(x)$ and $h_0(z)=H_0(z,1)=\sqrt{1-|z|^2}\bar h_0(x)$, for $x\in\Hyp^2$ which projects to $(z,1)$.
Let us consider the foliation of $D_0$ by leaves of the cosmological time, namely the surfaces $L_{T}$ whose support function is $\bar h_T(x)=\bar h_0(x)-T$ if $x\in \Hyp^2$, with $T\in(0,\infty)$. The surface $S_0$ descends to a compact surface in $D_0/\Gamma_0$ and therefore the time function $T$ on $S_0$ achieves a maximum $T_{\max}$ at a point $p$ (actually, a full discrete $\Gamma_0-$orbit) on $S_0$. It follows that the level surface $L_{T_{\max}}$ is entirely contained in the future of $S_0$ and tangent to $S_0$ at $p$.

By the construction of the cosmological time, for any point $p\in L_T$ there exists a hyperboloid of curvature $-1/T^2$ which is tangent to $L_T$ and contained in the future of $L_T$.
Hence, the surface $S_0$ is contained in the past of a hyperboloid of curvature $-1/T^2$, and tangent to such hyperboloid at some point, which implies that $\inf|K|\leq 1/T^2$. Therefore $T\leq C$ for $C=1/\sqrt{\inf|K|}$. This shows that the surface $S_0$ is contained in the past of the level surface $L_{C}$ and its support function on $\Hyp^2$ satisfies $\bar h_0(x)-C\leq \bar u_0(x)$. 
\end{proof}



We are now ready to conclude the proof.
\begin{proof}[Proof of Theorem \ref{theorem existence lsc}]
Given the lower semicontinuous function $\varphi:\partial\D\rar \R$, let us consider the dual lamination $\mu$ of the domain of dependence $D$ defined by $\varphi$. Hence $D=D(\mu,x_0,y_0)$ for some $x_0,y_0$ and the support function $h$ of $D$ is 
the convex envelope of $\varphi$.

By Lemma \ref{fuchsian weak approximation}, there exist measured geodesic laminations $\mu_n$, invariant under the action of torsion-free cocompact Fuchsian groups $G_n$, which converge weakly to $\mu$. 
Recall from the proof of Lemma \ref{fuchsian weak approximation} that the Fuchsian group $G_n$ has a fundamental domain $P_n'$ which contains the ball $B(0,n)$ for the hyperbolic metric. Let us define a $G_n$-invariant function $\psi_n:\Hyp^2\rar\R$, which approximates $\psi$. We take a partition of unity $\{\rho_n,\varrho_n\}$ subordinate to the covering $\{B(0,n),P'_n\setminus B(0,n/2)\}$ of $P'_n$. We define
$$\psi_n(x)=\rho_n(x)\psi(x)+\varrho_n(x)(\inf\psi)$$
and we extend $\psi_n$ to $\Hyp^2$ by invariance under the isometries in $G_n$.
It is clear that the sequence $\psi_n$ converges to $\psi$ uniformly on compact sets of $\Hyp^2$, since $\psi_n$ agrees with $\psi$ on $B(0,n/2)$. Since $\psi_n$ is constant on $P'_n\setminus B(0,n/2)$, the $G_n$-invariant extension is smooth. Finally, $\inf\psi_n=\inf\psi$.
Applying Theorem \ref{theorem bbz}, we obtain a solution $u_n$ to the equation 
$$\det D^2u_n(z)=\frac{1}{\psi_n(z)}(1-|z|^2)^{-2}$$
for every $n$. 

Let $D_n=D(\mu_n,x_0,y_0)$ be the domain of dependence associated with $\mu_n$, so that $y_0+x_0^\perp$ is a support plane of the domain $D_n$. 
We can assume $x_0$ is a point of $\Hyp^2$ which does not belong to a weighted leaf of any $\mu_n$. Let $H_n$ be the extended support function of $D_n$ and let $h_n$ be its usual restriction to $\D$. By Lemma \ref{limit weak uniform compact}, $h_n$ converges uniformly on compact sets of $\D$ to $h$. Moreover by inequality (\ref{inequality support functions}) of Lemma \ref{lemma inequality support functions}, the convex functions $u_n$ are uniformly bounded on every compact set of $\D$. Hence, by convexity, the $u_n$ are equicontinuous on compact sets of $\D$ and therefore, by the Ascoli-Arzel\`a Theorem, there exists a subsequence converging uniformly on compact sets to a convex function $u$. The limit function $u$ is a generalized solution of the equation
\begin{equation}
\det D^2u(z)=\frac{1}{\psi(z)}(1-|z|^2)^{-2}\,.  \tag{\ref{monge ampere constan curvature}}
\end{equation}

By Theorem \ref{solution strictly convex dimension 2}, $u$ is strictly convex and therefore is smooth by Theorem \ref{solution smooth}. Moreover, the functions $u_n$ satisfy the inequality in (\ref{inequality support functions}) for every $z\in\D$:
$$h_n(z)-C\sqrt{1-|z|^2}\leq u_n(z)\leq h_n(z)\,,$$
hence for the limit function $u$ we have
\begin{equation} \label{inequality support functions constant curvature surface}
h(z)-C\sqrt{1-|z|^2}\leq u(z)\leq h(z)\,, \tag{CT}
\end{equation}
where $h$ is the limit of the support functions $h_n$ of $D_n$ and is the support function of $D$ by Lemma \ref{limit weak uniform compact}. 
Since both $u$ and $h$, extended to $\overline \D$, are lower semicontinuous and convex functions, and the value on a point $z\in\partial\D$ coincides with the limit along a radial geodesic (see Lemma \ref{lemma limit radial geodesics}), we have the boundary condition
 $u|_{\partial\D}=h|_{\partial\D}=\varphi$. This shows that the condition \eqref{monge ampere boundary condition} holds, and concludes the proof.
\end{proof}

\subsection{Uniqueness of solutions} \label{cauchy and uniqueness}

In this subsection, we discuss the uniqueness of the solution of Equation \eqref{monge ampere constan curvature}, for which the existence was proved in Theorem \ref{theorem existence lsc}. More precisely, we prove the following:

\begin{prop} \label{theorem uniqueness lsc}
Given a bounded lower semicontinuous function $\varphi:\partial\D\rar \R$ and a smooth function $\psi:\Hyp^2\to[a,b]$ for some $0<a<b<+\infty$, the smooth solution $u:\D\rar\R$ to the equation
\begin{equation}
\det D^2u(z)=\frac{1}{\psi(z)}(1-|z|^2)^{-2}  \tag{\ref{monge ampere constan curvature}}
\end{equation}
satisfying
\begin{equation} 
u|_{\partial\D}=\varphi\,.  \tag{\ref{monge ampere boundary condition}}
\end{equation}
is unique.
\end{prop}

The claim in Proposition \ref{theorem uniqueness lsc} holds if $\varphi$ is continuous, by a direct application of Theorem \ref{comparison principle}. The rest of this subsection will be devoted to the proof of the claim when $\varphi$ is only assumed to be lower semicontinuous. The key property is that every solution of \eqref{monge ampere constan curvature} with boundary value $\varphi$, for $\psi>a>0$, satisfies the condition \eqref{inequality support functions constant curvature surface}. Geometrically, this means that every Cauchy surface with curvature bounded away from zero has bounded cosmological time. 

\begin{prop} \label{estimate cosmological any cc surface}
Given a smooth function $\psi:\D\rar[a,b]$ for some $0<a<b<+\infty$, any smooth solution $u:\D\rar\R$ to the equation
\begin{equation}
\det D^2u(z)=\frac{1}{\psi(z)}(1-|z|^2)^{-2}  \tag{\ref{monge ampere constan curvature}}
\end{equation}
with 
\begin{equation} 
u|_{\partial\D}=\varphi\,.  \tag{\ref{monge ampere boundary condition}}
\end{equation}
satisfies
\begin{equation} \label{inequality support functions constant curvature surface}
h(z)-C\sqrt{1-|z|^2}\leq u(z)\leq h(z)\,, \tag{CT}
\end{equation}
for some constant $C>0$, where $h=\co\varphi$.
\end{prop}
\begin{proof}
The statement is true if $\varphi$ is continuous, as we have already observed that in that case the solution is unique, and in Theorem \ref{theorem existence lsc} we have proved the existence of a solution satisfying the required condition.
Let us now consider the general case. Since $u$ is convex, it is clear that $u\leq h$. We show the other inequality. Let $r\in(0,1]$ and $u_r:\D\rar\R$ be defined as
$$u_r(z)=u(rz)\,.$$
Since $u$ is continuous (actually, smooth) on $\D$, $u_r$ converges uniformly on compact sets of $\D$ to $u$ as $r\rar 1$. Let $\psi_r$ be such that  $$\det D^2u_r(z)=\frac{1}{\psi_r(z)}(1-|z|^2)^{-2}\,.$$
We have
$$\det D^2 u_r(z)=r^4 \det D^2u(rz)=\frac{r^4}{\psi(rz)}(1-r^2|z|^2)^{-2}\leq \frac{1}{\inf\psi}(1-|z|^2)^{-2}$$
and therefore $\psi_r(z)\geq\inf\psi$.
Since $u_r$ is continuous on $\overline \D$, by the continuous case and the above inequality we obtain
\begin{equation} \label{stima limite da dentro}
h_r(z)-\frac{1}{\sqrt{\inf\psi}}\sqrt{1-|z|^2}\leq u_r(z)\,,
\end{equation}
where $h_r=\co (u_r|_{\partial\D})$.

Fix a point $z_0\in\D$. We claim that $h(z_0)\leq \liminf_r h_r(z_0)$. Indeed, given an arbitrary affine function $f:\D\rar\R$ such that $f|_{\partial \D}< \varphi$, the set $\{z:u(z)\leq f(z)\}$ is compact in $\D$. Suppose it is contained in a ball of radius $r_0$.
Hence if $r>r_0$, for every $z_1\in\partial\D$
$$u_r(z_1)=u(rz_1)>f(rz_1)=f_r(z_1)\,,$$
where of course we defined $f_r(z)=f(rz)$.
This shows that $u_r|_{\partial\D}> f_r|_{\partial\D}$, and therefore $h_r\geq f_r$. 
As $r\to 1$, we obtain $f(z_0)\leq \liminf_r h_r(z_0)$, and the claim follows since $f$ is arbitrary.

Now taking limits in Equation \eqref{stima limite da dentro}, we conclude that $h(z_0)-C\sqrt{1-|z_0|^2}\leq u(z_0)$ for $C=1/\sqrt{\inf\psi}$. As the point $z_0$ is arbitrary, we conclude the proof.
\end{proof}

\begin{proof}[Proof of Proposition \ref{theorem uniqueness lsc}]
Let $u_1,u_2$ be two solutions with $u_1|_{\partial\D}=u_2|_{\partial\D}=\varphi$. By Proposition \ref{estimate cosmological any cc surface}, there exists a constant $C$ such that
$$-C\sqrt{1-|z|^2}\leq u_1(z)-u_2(z)\leq C\sqrt{1-|z|^2}\,.$$
Hence the function $u_1-u_2$ extends continuously to zero at the boundary $\partial\D$. Therefore $u_1-u_2$ has a minimum on $\overline\D$. By Theorem \ref{comparison principle}, the minimum cannot be achieved at an interior point. Therefore the minimum is achieved on $\partial\D$, which means that $u_1\geq u_2$. By exchanging the roles of $u_1$ and $u_2$, one can conclude that $u_1\equiv u_2$.
\end{proof}


\subsection{The solution is an entire graph} \label{no light rays}

In this subsection we prove that the solutions constructed in Theorem \ref{theorem existence lsc} are the support functions of spacelike entire graphs in $\R^{2,1}$. We will make use of barriers which are constant curvature surfaces invariant under a parabolic group. The proof of the following proposition is given in Appendix \ref{appendix parabolic}.

\begin{prop} \label{surface parabolic graph}
For every $K<0$, $C>0$ and every null vector $v_0\in\R^{2,1}$ there exists 
an entire graph $S_{C,v_0}(K)$ of constant curvature $K$ which is a Cauchy surface in the domain of dependence with support function at infinity
$$\varphi_{C,v_0}(z)=\begin{cases} -\sqrt{C} & [z]=[v_0] \\ 0 & [z]\neq[v_0] \end{cases}\,.$$
\end{prop}

In Appendix \ref{appendix parabolic} the surfaces of Proposition \ref{surface parabolic graph} are described in more detail. For instance, we remark that the surface $S_{C,v_0}(K)$ is not complete. The domain of dependence with support function equal to $\co(\varphi_{C,v_0})$ is the future of a parabola, as pictured in Figure \ref{fig:parabola}, see also the discussion in the proof of of Theorem \ref{theorem minkowski problem lsc} below. In order to use the surface $S_{C,v_0}(K)$ as a barrier, we need to prove a technical lemma.

\begin{lemma} \label{confronto sci}
Let $\varphi_1,\varphi_2:\partial\D\rar \R$ be two bounded lower semicontinuous functions  and let $\psi_1,\psi_2:\Hyp^2\to[a,b]$ be two smooth functions, for some $0<a<b$. If $\varphi_1\leq\varphi_2$ and $\psi_1\leq\psi_2$, then the smooth solutions $u_i:\D\rar\R$ (for $i=1,2$) to the equation
$$
\det D^2u_i(z)=\frac{1}{\psi_i(z)}(1-|z|^2)^{-2}
$$
with  
$$u_i|_{\partial\D}=\varphi_i$$ 
satisfy $u_1\leq u_2$ on $\D$.
\end{lemma}
\begin{proof}
Suppose first $\varphi_1$ is continuous. Therefore also the solution $u_1$ is continuous on $\overline\D$, since it satisfies the condition $$h_1(z)-C\sqrt{1-|z|^2}\leq u_1(z)\leq h_1(z)\,,$$
where $h_1$ is the convex envelope of $\varphi_1$. Then the function $u_2-u_1$ is lower semicontinuous and is positive on the boundary, therefore it achieves a minimum. By Theorem \ref{comparison principle}, the minimum has to be on the boundary, hence $u_2\geq u_1$ on $\overline\D$.

Now for the general case, let $\varphi_1$ be lower semicontinuous and let $\varphi_n$, $n\geq 3$, be a sequence of continuous functions which converge to $\varphi_1$ monotonically from below, namely $\varphi_n\leq \varphi_{n+1}$ and $\varphi_n\leq \varphi_1$ for every $n\geq 3$. Let $u_n$ be the solution of the equation
$$
\det D^2u_n(z)=\frac{1}{\psi_1(z)}(1-|z|^2)^{-2}
$$
with $$u_n|_{\partial\D}=\varphi_n\,.$$ By the previous case, we know (if $n\geq 3$) that $u_n\leq u_{n+1}$ and $u_n\leq u_1$. Hence the $u_n$ are uniformly bounded and convex, thus by convexity the sequence $u_n$ converges uniformly on compact sets (up to a subsequence) to a generalized solution $u_\infty$ of the same equation:
$$
\det D^2u_\infty(z)=\frac{1}{\psi_1(z)}(1-|z|^2)^{-2}\,.
$$
It is clear that $u_\infty\leq u_1$. Let $z\in\partial\D$. Recall the value of $u_\infty$ on $z$ coincides with the limit on radial geodesics. Hence we have $u_\infty(z)=\lim_{r\to 1} u_\infty(rz)\geq \lim_{r\to 1} u_n(rz) =\varphi_n(z)$ for every $n$. Therefore $u_\infty|_{\partial\D}\equiv\varphi_1$. By the uniqueness proved in Proposition \ref{theorem uniqueness lsc}, $u_\infty\equiv u_1$. Since $u_n(z)\leq u_2(z)$ for $n\geq 3$ and for every $z\in \D$, we conclude that $u_1\leq u_2$.
\end{proof}

We are finally ready to conclude the proof of Theorem \ref{theorem minkowski problem lsc}.
\begin{proof}[Proof of Theorem \ref{theorem minkowski problem lsc}]
We have showed in Theorem \ref{theorem minkowski problem lsc} and Proposition \ref{theorem uniqueness lsc} that there exists a unique solution $u$ to Equation \eqref{monge ampere constan curvature}, hence having the required curvature function. Moreover, the solution satisfies Equation \eqref{inequality support functions constant curvature surface}, which ensures that the surface $S$ with support function $u$ is a Cauchy surface and satisfies the estimate on the cosmological time.

It only remains to show that $S$ is a spacelike entire graph. Suppose it is not. Therefore $S$ is tangent to the boundary of the domain of dependence (recall Subsection \ref{preliminaries subsec convex surfaces}) and develops a lightlike ray $R$ at the tangency point. Suppose the lightlike ray is parallel to the null vector $v_0$ of $\R^{2,1}$. Let $\alpha$ be such that $\varphi(z)\leq \alpha$ for every $z\in\partial\D$.

We consider the function
$$\varphi_0(z)=\begin{cases} \varphi(z) & [z]=[v_0] \\ \alpha & [z]\neq[v_0] \end{cases}\,.$$
From Proposition \ref{surface parabolic graph}, there exists a $K_0$-surface $S_0$ with support function at infinity $\varphi_0$, for $K_0=\inf K=-\sup|K|$, which is obtained by translating vertically in $\R^{2,1}$ a suitably chosen surface $S_{C,v_0}(K_0)$. Geometrically, this is equivalent to choosing the domain of dependence whose boundary is the future of a parabola (see Figure \ref{fig:parabola}). The parabola is obtained by intersecting  the plane containing the lightlike ray $R$ with a cone $\Ip(p)$ over a point $p$ on the $z$-axis, sufficiently in the past, so as to contain the original surface $S$.

Applying Lemma \ref{confronto sci}, we see that $S$ is in the future of $S_0$. However, $S_0$ is an entire graph (see Appendix \ref{appendix parabolic}) and both $S$ and $S_0$ have the same lightlike support plane with normal vector $v_0$. This gives a contradiction and concludes the claim that $S$ is an entire graph.
\end{proof}

\begin{figure}[htb]
\centering
\includegraphics[height=5.5cm]{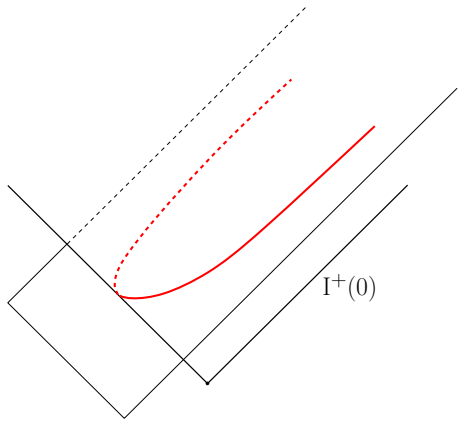}
\caption{The future of a parabola is a domain of dependence invariant for the parabolic group. Its support function at infinity is lower-semicontinuous and is affine on the complement of a point of $\partial\D$. \label{fig:parabola}}
\end{figure}

\begin{remark} \label{remark counterexample constant curvature not entire}
In the above proof of Theorem \ref{theorem minkowski problem lsc}, we have actually showed that every smooth convex bounded solution $u:\D\rar\R$ of
\begin{equation}
\det D^2u(z)=\frac{1}{\psi(z)}(1-|z|^2)^{-2}  \tag{\ref{monge ampere constan curvature}}
\end{equation}
with $\psi:\D\rar(0,\infty)$, $\sup\psi<\infty$, corresponds to a spacelike entire graph in $\R^{2,1}$ with curvature $K(G^{-1}(z))=-\psi(z)$. On the other hand, the hypothesis that $\sup\psi<\infty$ is essential. We give an explicit counterexample. Let us consider the function $u:\D\rar\R$ defined by $u(z)=|z|^2/2$. It is easily checked that the dual surface $S$ defined by $u$ is the graph of $u$ itself, and is only defined over $\D$. One can see that $S$ is tangent to the lightcone centered at $(0,0,-1/2)$ and that its curvature function is $K(G^{-1}(z))=-\psi(z)=-\frac{1}{(1-|z|^2)^2}$, hence is unbounded. See also Figure \ref{fig:counterexample}.
\end{remark}
\begin{figure}[htb]
\centering
\includegraphics[height=4.5cm]{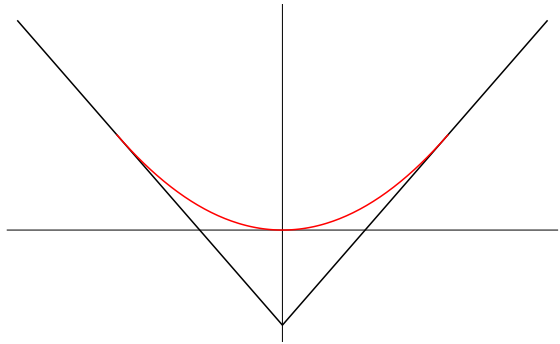}
\caption{A Cauchy surface in the cone over a point, of unbounded negative curvature, which is not a spacelike entire graph. The surface (red) is obtained by revolution around the vertical axis. \label{fig:counterexample}}
\end{figure}

\section{Foliations by constant curvature surfaces}

In this section we prove that every domain of dependence defined by a bounded support function at infinity is foliated by $K$-surfaces, as $K$ varies in $(-\infty,0)$. The main statement is the following.

\begin{theorem} \label{theorem foliation bounded}
Every domain of dependence $D$, with bounded support function at infinity $\varphi:\partial \D\rar \R$, is foliated by smooth spacelike entire graphs of constant curvature $K\in(-\infty,0)$.
\end{theorem}

The existence of such $K$-surfaces follows from Theorem \ref{theorem minkowski problem lsc}, by choosing the constant function $\psi\equiv |K|$. We now show that the $K$-surfaces foliate the domain of dependence $D$.

\begin{lemma} \label{convergence support functions implies convergence graphs}
Let $S_n=\mathrm{graph}(f_n)$ and $S_\infty=\mathrm{graph}(f_\infty)$ be spacelike entire graphs in $\R^{2,1}$
 with $C^1$ support functions $u_n:\D\rar\R$ and $u_\infty:\D\rar\R$. 
 If $u_n$ converges to $u_\infty$ uniformly on compact sets, then $f_n$ converges uniformly on compact sets to $f_\infty$.
\end{lemma}
\begin{proof}
By a slight abuse of notation, we consider here the Gauss map $G_n:S_n\to\D$ using the canonical identification $\pi:\Hyp^2\to \D$.
As by convexity  we have that $Du_n(z)\to Du_\infty (z)$ for any $z\in\D$, Formula \eqref{formula inverse gauss map}  implies that
$$ G_n^{-1}(z)\to G_\infty^{-1}(z)$$
for all $z\in\D$.
Let us set  $G_n^{-1}(z)=(p_n(z), f_n(p_n(z))$, where $p_n(z)$ is the vertical projection to $\R^2$.
We have that 
\begin{equation}\label{eq:fff}
p_n(z)\to p_\infty(z)\quad \textrm{ and }\quad f_n(p_n(z))\to f_\infty(p_\infty(z))\,.
\end{equation}
Using that $f_n$'s are $1$-Lipschitz, by a standard use of Ascoli-Arzel\`a Theorem, we get that
up to a subsequence, $f_n$ converges uniformly on compact subset of $\R^2$ to some function $g$.

In order to prove that $g=f_\infty$, let us use again \eqref{eq:fff}. 
We get that $f_\infty(p_\infty(z))=\lim f_n(p_n(z))=g(p_\infty(z))$.
So $f$ and $g$ coincide on the image of $p_\infty$. As we are assuming that $S_\infty$ is a spacelike entire
graph we conclude that they coincide everywhere.
\end{proof}

Recall that a $K_0$-surface $S(K_0)$ in $D$ is constructed as limit of $K_0$-surfaces $S_n(K_0)$ invariant under the action of a surface group. By the work of \cite{barbotzeghib}, the $K$-surfaces $S_n(K)$ foliate the domain of dependence $D_n$ of $S_n(K_0)$, as $K$ varies in $(-\infty,0)$.

\begin{theorem}[\cite{barbotzeghib}] \label{theorem bbz foliation}
Let $G_0$ be a Fuchsian cocompact group 
and let $\Gamma_0$ be a discrete subgroup of $\isom(\R^{2,1})$ whose linear part is $G_0$. 
Then the Cauchy surfaces $S_0(K)$ of constant curvature $K\in(-\infty,0)$ foliate the maximal domain of dependence $D_0$ invariant under the action of the group $\Gamma_0$, in such a way that if $K_1<K_2$, then $S_0(K_2)$ is contained in the future of $S_0(K_1)$.
\end{theorem}

\begin{proof}[Proof of Theorem \ref{theorem foliation bounded}]
The proof is split in several steps. First we prove the constant curvature surfaces are pairwise disjoint, then that the portion contained between two constant curvature surfaces is filled by other constant curvature surfaces, and finally that one can find a constant curvature surface arbitrarily close to the boundary of the domain of dependence and to infinity. 

\textit{Step 1.} Let us show that, if $K_1<K_2$, then the constant curvature surfaces $S(K_1)$ and $S(K_2)$ are disjoint, and $S(K_2)$ is in the future of $S(K_1)$. Let $S_n(K_1)$ and $S_n(K_2)$ be approximating sequences as in the proof of Theorem \ref{theorem existence lsc}, and let $u_n(K_1)$ and $u_n(K_2)$ be the corresponding support functions. From Theorem \ref{theorem bbz foliation} of \cite{barbotzeghib}, we know that $u_n(K_2)<u_n(K_1)$. Hence in the limit $u(K_2)\leq u(K_1)$, where $u(K_i)$ is the support function of $S(K_i)$. Hence $S(K_1)$ and $S(K_2)$ do not intersect transversely. Moreover $S(K_1)$ is in the closure of the past of $S(K_2)$. Finally $S(K_1)$ and $S(K_2)$ cannot be tangent at a point, since $|K_1|>|K_2|$ and thus at least one of the eigenvalues of the shape operator of $S(K_1)$ is larger than the largest eigenvalue of $S(K_2)$.

\textit{Step 2.} We show that, given two Cauchy surfaces $S(K_1)$, $S(K_2)$ in $D$ of constant curvature $K_1<K_2$, every point between $S(K_1)$ and $S(K_2)$ lies on a Cauchy surface of constant curvature. Let $x$ be a point in $\R^{2,1}$ contained in the past of $S(K_2)$ and in the future of $S(K_1)$. For $n$ large, $x$ is in the past of $S_n(K_2)$ and in the future of $S_n(K_1)$. Therefore there exists a surface $S_n(K_n)$ through $x$, with $K_1<K_n<K_2$. Up to a subsequence, let us assume $K_n\rar K_\infty$. Using the same argument we gave in the proof of Theorem \ref{theorem existence lsc}, the support functions $u_n(K_n)$ converge (up to a subsequence) uniformly of compact sets to $u_\infty(K_\infty)$, which is the support function of the $K_\infty$-surface $S(K_\infty)$ in $D$. Since $x\in S_n(K_n)$ for every $n$, Lemma \ref{convergence support functions implies convergence graphs} implies that $x\in S(K_\infty)$.

\textit{Step 3.} We show that for every point $x$ nearby the boundary of the domain of dependence there is a constant curvature Cauchy surface $S(K_0)$ such that $x\in S(K_0)$. This follows from Equation \eqref{inequality support functions constant curvature surface}, which states that the $K$-surface $S(K)$ in $D$ is contained in the past of the $(1/\sqrt{|K|})$-level surface of the cosmological time. Since the level surfaces $L_C=\{T=C\}$ of the cosmological time get arbitrarily close to the boundary of the domain of dependence as $C$ gets close to $0$, it is clear that $x$ is in the future of $S(K_0)$ if $|K_0|$ is large enough. The claim follows by Step 2. 

\textit{Step 4.} Finally, we show that every point far off at infinity lies on some constant curvature Cauchy surface. By contradiction, suppose there is a point $x$ which is in the future of every $K$-surface $S(K)$. Let $S(K)$ be the graph of $f_K:\R^2\rar\R$. Then the $f_K$ are uniformly bounded (they are all smaller than the function which defines $\Ip(x)$) and convex. Up to a subsequence $f_{K}\rar f_\infty$ uniformly on compact sets. The function $f_\infty$ defines a surface $S_\infty$ contained in the domain of dependence $D$.

Now, $f_K$ satisfies (see for instance \cite{Li})
$$\det D^2 f_K=|K|(1-||Df||^2)^2\,.$$
Therefore, taking the limit as $K\to 0$, $\det D^2 f_\infty=0$ in the generalized sense. 
The following Lemma states that $f_\infty$ is affine along a whole line of $\R^2$,
and this gives a contradiction, since $S_\infty$ would contain an entire line and thus could not  be contained in the domain of dependence $D$.
\end{proof}

\begin{lemma} \label{lemmaaffine}
Let $f:\R^2\to \R$ be a convex function which satisfies the equation $\det D^2f=0$ in the generalized sense.
Then there exist a point $x_0\in \R^2$, a vector $v\in\R^2$, and $\alpha\in\R$ 
such that $f(x_0+tv)=f(x_0)+\alpha t$ for every $t\in\R$.
\end{lemma}
\begin{proof}
By \cite[Theorem 1.5.2]{gutierrez} for any bounded convex  domain $\Omega\subset\R^2$,
$f|_{\Omega}$ coincides with the convex envelope of $f|_{\partial\Omega}$.
It follows that for any $x\in\R^2$ there is $v=v(x)\in\R^2$ and $\alpha=\alpha(x)\in\R$ such that
$f(x+tv)=f(x)+\alpha t$ for $t\in(-\epsilon,\epsilon)$, for some $\epsilon=\epsilon(x)>0$.

Fix a point $x_1$, and set $v_1=v(x_1)$.
Up to adding an affine function we may assume that $\alpha(x_1)=0$ and that $f(x)\geq f(x_1)$ for any $x\in\R^2$.
If $f$ is affine along the whole line $x_1+\R v_1$, we have done.
Otherwise take the maximal $t_1$ such that $f(x_1+t_1v_1)=f(x_1)$ and put $x_0=x_1+t_1v_1$.
Let $v_0=v(x_0)$, clearly $v_0\neq v_1$. As $f(x)\geq f(x_0)=f(x_1)$ for every $x\in\R^2$, necessarily  $\alpha(x_0)=0$. See Figure \ref{fig:lemmaaffine1}.

\begin{figure}[htb]
\centering
\includegraphics[height=4cm]{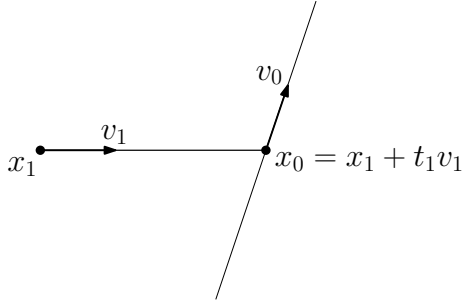}
\caption{The setting of the proof of Lemma \ref{lemmaaffine}. Composing with an affine map, we can assume $f$ is constant on the drawn segments. \label{fig:lemmaaffine1}}
\end{figure}

We claim that $f(x)>f(x_0)$ on the half-plane $P_0$ bounded by $x_0+\R v_0$ which does not contain $x_1$.
Otherwise we should have that $f\equiv f(x_0)$ on the triangle with vertices $x, x_0+\epsilon v_0, x_0-\epsilon v_0$,
but then $f$ would be  constant equal to $f(x_1)$ on some segment $[x_1, x_0+\eta v_1]$, violating he maximality of $t_1$. See Figure \ref{fig:lemmaaffine2}.

Now suppose that $f(x_0+tv_0)>f(x_0)$ for some $t$ and take the maximal $t_0$ for which $f(x_0+t_0 v_0)=f(x_0)$.
Define $x_2=x_0+t_0v_0$. As before we have that $v_2$ is different form $v_0$ and that $f\equiv f(x_0)$ on some segment
of the form $[x_2-\epsilon v_2, x_2+\epsilon v_2]$ (see Figure \ref{fig:lemmaaffine3}). As this segment contains points of $P_0$ we get a contradiction.
\end{proof}

\begin{figure}[htb]
\centering
\begin{minipage}[c]{.45\textwidth}
\centering
\includegraphics[height=4cm]{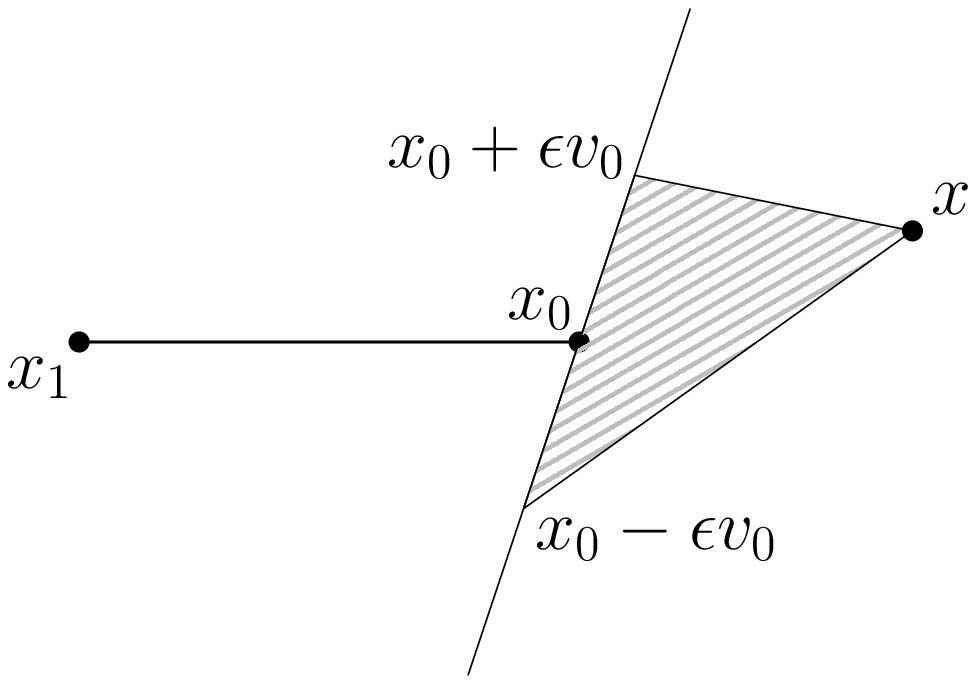}
\caption{We have $f(x)>f(x_1)$ on the right halfplane bounded by the line $x_0+\R v_0$, for otherwise $f$ would be constant on the ruled triangle, contradicting the maximality of $t_1$.} \label{fig:lemmaaffine2}
\end{minipage}%
\hspace{5mm}
\begin{minipage}[c]{.45\textwidth}
\centering
\includegraphics[height=4cm]{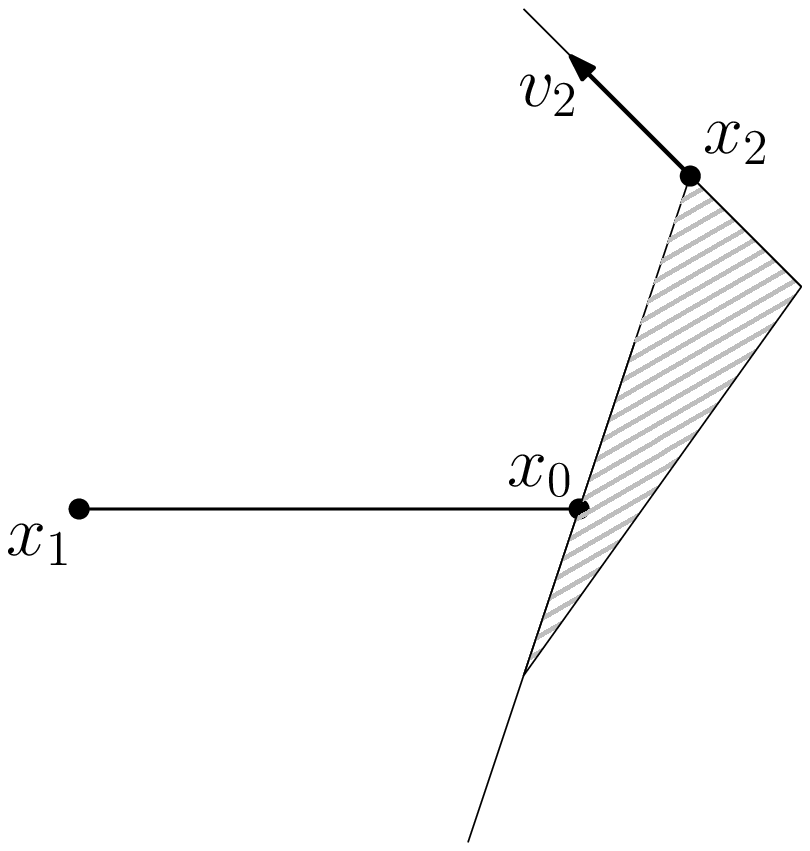}
\caption{By a similar argument, $f$ has to be constant on the entire line $x_0+\R v_0$.} \label{fig:lemmaaffine3}
\end{minipage}
\end{figure}

\section{Constant curvature surfaces} \label{sec:convex zygmund}

In this section we give a characterization of $K$-surfaces with bounded principal curvatures.

\begin{theorem} \label{big theorem zygmund finite dual lamination bounded curvatures}
Let $D$ be a domain of dependence in $\R^{2,1}$. The following are equivalent:
\begin{itemize}
\item[i)] The measured geodesic lamination $\mu$ dual to $\partial_s D$ is bounded, i.e. $||\mu||_{Th}<+\infty$.
\item[ii)] The support function at infinity $h=H|_{\partial\D}:\partial\D\rar\R$ of $D$ is in the Zygmund class.
\item[iii)] The domain of dependence $D$ contains a convex Cauchy surface with principal curvatures bounded from below by some constant $d>0$.
\item[iv)] The domain of dependence  $D$ is foliated by complete convex Cauchy surfaces of constant curvature $K$ with principal curvatures bounded from below by some constant $d=d(K)>0$, where $K\in(-\infty,0)$.
\end{itemize} 
\end{theorem}

We will give the proof in several steps. It is obvious that $iv)\Rightarrow iii)$. In Subsection \ref{zygmund and boundedness} we prove that $i)\Leftrightarrow ii)$. The existence part of $ii)\Rightarrow iv)$ follows by Theorem \ref{theorem foliation bounded}; in Subsection \ref{boundedness curvature} we complete the proof by showing that, if the dual lamination has finite Thurston norm, then the principal curvatures are bounded and the surface is complete. 
Finally, Subsection \ref{estimate norm curvatures} proves $iii)\Rightarrow i)$, by giving an explicit estimate of Thurston norm of the dual lamination in terms of the supremum of the principal curvatures, which holds for any convex Cauchy surface.

\subsection{Zygmund fields and bounded measured geodesic laminations} \label{zygmund and boundedness}

In this part, we discuss the equivalence between $i)$ and $ii)$. We prove here the key fact for this equivalence. Given a function $\varphi:S^1\to\R$, we denote by $\hat\varphi$ the vector field on $S^1$ associated to $\varphi$ by means of the standard trivialization.

\begin{prop} \label{zygmund e lamination}
Given an infinitesimal earthquake $\hat\varphi=\ddt E^{t\mu}$, the function $\varphi:S^1\rar\R$ is the support function at infinity on $\partial\D=S^1$ of a domain of dependence $D$ with dual lamination $\mu$.
\end{prop}
\begin{proof}
By composing $\hat\varphi$ with an infinitesimal M\"obius tranformation, (compare Section \ref{zygmund fields}) we can suppose the point $x_0\in \Hyp^2$ lies in a stratum of $\mu$ which is fixed by the earthquakes $E^{t\mu}$, for $t\in\R$. By Proposition \ref{support function from lamination}, the support function at infinity of the domain of dependence $D=D(\mu,x_0,0)$ which has dual lamination $\mu$ and $x_0^\perp$ as a support plane is 
\begin{equation} \label{eq: earthquake and support 1}
H(\eta)=\int_{\mathcal{G}[x_0,\eta)} \langle\eta,\boldsymbol{\sigma}\rangle d\mu\,,
\end{equation}
for every $\eta$ in $\overline{\partial\Ip(0)}$. Here $[x_0,\eta)$ denotes the geodesic ray obtained by projecting to $\Hyp^2$ the line segment connecting $x_0$ and $\eta$ (recall the convention introduced before Proposition \ref{support function from lamination}).

By Lemma \ref{lemma campo e funzione omogenea}, the vector field $\hat\varphi$ on $S^1$ defines a 1-homogeneous function $\Phi$ on $\partial\Ip(0)$. Since $H$ is 1-homogeneous, it suffices to check that $H$ and $\Phi$ agree on $\partial\D= \partial\Ip(0)\cap\{x_3=1\}$. Let $\eta\in\partial\D$ and let $v$ be the unit vector tangent to $\partial \D$ in the counterclockwise orientation. By Lemma \ref{lemma campo e funzione omogenea}, under the standard identification of $\partial\D$ with $S^1$, we have $$\varphi(\eta)=\langle \hat\varphi(\eta),v\rangle=\Phi(\eta)\,.$$
We now compute the infinitesimal earthquake $\varphi$ at a point $\eta$. If $l$ is a leaf of $\mu$, the infinitesimal earthquake along the lamination composed of the only leaf $l$ (as in Example \ref{infinitesimal earthquake one leaf}) is
$$\dot E_l(\eta)=\langle\dot A^l(\eta),v\rangle v=\langle\eta\boxtimes\boldsymbol{\sigma}(l),v\rangle v\,,$$
where $\dot A^l=\ddt A^l(t)\in\so(2,1)$ is the infinitesimal generator of the 1-parameter subgroup of hyperbolic isometries $A^l(t)$ which translate on the left (as seen from $x_0$) along the geodesic $l$ by lenght $t$. Here $\boxtimes$ denotes the Minkowski cross product. In the second equality we have used the fact that $\dot A^l(\eta)=\eta\boxtimes\boldsymbol{\sigma}(l)$ (see for instance \cite[Appendix B]{bonschlfixed}). 

Using Equation \eqref{integral formula infinitesimal earthquake} in Theorem \ref{teorema saric}, we obtain
\begin{equation} \label{eq: earthquake and support 2}
\langle \hat\varphi(\eta),v\rangle =\int_{\mathcal{G}[x_0,\eta)} \dot E_l(\eta)d\mu=\int_{\mathcal{G}[x_0,\eta)} \langle \eta\boxtimes\boldsymbol{\sigma},v\rangle d\mu\,.
\end{equation}
We will show that $\langle\eta,\boldsymbol{\sigma}(l)\rangle=\langle \eta\boxtimes\boldsymbol{\sigma}(l),v\rangle$, from which the claim follows, by comparing Equations \eqref{eq: earthquake and support 1} and \eqref{eq: earthquake and support 2}. For this purpose, let $p=(0,0,1)$ and $\eta=p+w$, so that $(p,w,v)$ gives an orthonormal oriented triple. Suppose $\boldsymbol{\sigma}(l)=a p+b w+c v$. Then 
$$\langle\eta,\boldsymbol{\sigma}(l)\rangle=a\langle p,p\rangle+b \langle w,w\rangle=b-a$$
whereas
$$\eta\boxtimes\boldsymbol{\sigma}(l)=b(p\boxtimes w)+a(w\boxtimes p)+c(p\boxtimes v+w\boxtimes v)$$
and in conclusion $\langle \eta\boxtimes\boldsymbol{\sigma}(l),v\rangle=b-a=\langle\eta,\boldsymbol{\sigma}(l)\rangle$.
\end{proof}

\begin{remark} We observe that if a different base point $x_0'$ (out of any weighted leaf of $\mu$) is chosen, we obtain the 1-homogeneous function $H'$ such that 
$$
H'(\eta)=\int_{\mathcal{G}[x_0',\eta)} \langle\eta,\boldsymbol{\sigma}\rangle d\mu=
\int_{\mathcal{G}[x_0',x_0]} \langle\eta,\boldsymbol{\sigma}\rangle d\mu+
\int_{\mathcal{G}[x_0',\eta)} \langle\eta,\boldsymbol{\sigma}\rangle d\mu
$$
where this equality follows from the fact that $\mu$ is a measured geodesic lamination, and hence in Equation \eqref{eq: earthquake and support 1} the interval $[x_0,x]$ can be replaced by any path from $x_0$ to $x$ transverse to the support of the lamination (see also \cite{Mess}). Clearly, if a suitable normalization is chosen, also the infinitesimal earthquake $\hat\varphi'$ changes in the same way. Therefore $H'$ agrees with the function $\Phi'$ obtained from $\hat\varphi'$.

Finally, recall that a tangent element to Universal Teichm\"uller space is defined as a vector field on $S^1$ satisfying the Zygmund condition \eqref{cross-ratio infinitesimal} up to a first-order deformation by M\"{o}bius transformations. Therefore a different representative $\hat\varphi'$ in the same equivalence class of $\hat\varphi$ differs by
$$\hat\varphi'(\eta)=\hat\varphi(\eta)+\ddt A(t)(\eta)=\hat\varphi(\eta)+\dot A(\eta)=\hat\varphi(\eta)+\eta\boxtimes y_0$$
where $A(t)\in\isom(\R^{2,1})$, $A(0)=\id$ and $\dot A=\ddt A(t)\in\so(2,1)$. 
By the same computation as above, $\langle \hat\varphi'(\eta),v\rangle=H'(\eta)$ where $H'$ is the support function of the domain of dependence $D(\mu,x_0,y_0)$.
\end{remark}

By applying Proposition \ref{zygmund e lamination} and results presented in \cite{garhulakic} or \cite{saricweak} on the convergence of the integral in Equation \eqref{integral formula infinitesimal earthquake}, one deduces that, if $D$ is a domain of dependence whose dual lamination has finite Thurston norm, then the support function at infinity of $D$ is finite. We will give a quantitative version of this fact in Proposition \ref{uniform boundedness support function bounded lamination} - which will be useful because it gives a uniform bound on the support function in terms of the Thurston norm. Here we draw another consequence of Proposition \ref{zygmund e lamination}, namely the equivalence of conditions $i)$ and $ii)$.

\begin{cor} \label{cor support function extends equivalence 12} 
Given a domain of dependence $D$, the dual lamination $\mu$ has finite Thurston norm if and only if the support function $h$ of $D$ extends to a Zygmund field on $\partial \D$.
\end{cor}
\begin{proof}
If $||\mu||_{Th}<+\infty$, from Proposition \ref{zygmund e lamination} we know that $h|_{\partial\D}$ coincides with the infinitesimal earthquake along $\mu$, hence is a Zygmund field. Viceversa, if $h|_{\partial\D}$ is a Zygmund field, by Theorem \ref{teorema saric} there exists a bounded lamination $\mu$ such that $h|_{\partial\D}$ is the infinitesimal earthquake along $\mu$, and we conclude again by Proposition \ref{zygmund e lamination}.
\end{proof}

\subsection{Boundedness of curvature} \label{boundedness curvature}
To prove the implication $i)\Rightarrow iv)$, we have showed in Theorem \ref{theorem minkowski problem lsc} the existence of constant curvature Cauchy surfaces $S(K)$, while in Theorem \ref{theorem foliation bounded} we proved that the surfaces $S(K)$ foliate the domain of dependence as $K\in(-\infty,0)$. It remains to show that the principal curvatures of the Cauchy $K$-surfaces $S(K)$ we constructed are bounded provided the dual measured geodesic lamination has finite Thurston norm. This will also imply that $S(K)$ is complete, since (by boundedness of the curvature and of the principal curvatures) the Gauss map is bi-Lipschitz with respect to the induced metric on $S$ and the hyperbolic metric of $\Hyp^2$.

\begin{prop} \label{boundedness curvature boundedness lamination}
Given a $K$-surface $S$, if the lamination $\mu$ dual to the domain of dependence $D(S)$ has finite Thurston norm, then the principal curvatures of $S$ are uniformly bounded.
\end{prop}




We will prove the proposition by contradiction. If the statement did not hold, there would exist a sequence of points $x_n\in S$ such that the principal curvatures diverge (since the product of the principal curvatures is constant, necessarily one principal curvature will tend to zero and the other to infinity). Roughly speaking, we will choose isometries $A_n$ so that the points $x_n$ are sent to a compact region of $\R^{2,1}$, and consider the surfaces $S_n=A_n(S)$. Essentially, a contradiction will be obtained by showing that the sequence $S_n$ contains a subsequence converging to a constant curvature smooth surface $S_\infty$ - using the boundedness of the dual lamination - and that this gives bounds on the principal curvatures at $x_n$. Hence it is not possible that principal curvatures diverge.

In order to apply the above argument, we need to prove a uniform bound on the support functions, depending only on the Thurston norm of the dual lamination.

\begin{prop} \label{uniform boundedness support function bounded lamination}
Let $D_0=D(\mu_0,x_0,y_0)$ be a domain of dependence whose dual lamination $\mu_0$ has Thurston norm $||\mu_0||_{Th}<M$, such that $P=y_0+x_0^\perp$ is a support plane tangent to $\partial_s D_0$ at $y_0$. Let $h_0$ be the support function of $D_0$. Then $h_0\leq C$ on $\overline\D$ for a constant $C$ which only depends on $M,x_0,y_0$.

\end{prop}
\begin{proof}
The support function of $D_0$ restricted to $\Hyp^2$, under the hypothesis, is given by (see Proposition \ref{support function from lamination}):
$$\bar h_0(x)=\langle x,y_0\rangle+\int_{\mathcal{G}[x_0,x]} \langle x,\boldsymbol{\sigma}\rangle d\mu_0\,.$$
It is harmless to assume that $x_0=(0,0,1)$ and $y_0=0$; indeed, composing with an isometry of $\R^{2,1}$, the support function $h$ changes by an affine map on $\overline\D$.
Hence we give an estimate of the integral term in $\bar h_0$. Let $\gamma$ be a unit speed parametrization of the geodesic segment $[x_0,x]$. Note that if ${\gamma}(s)$ is on a geodesic $l$, for every $x\in\Hyp^2$, $$|\langle x,\boldsymbol{\sigma}(l)\rangle|=\sinh d_{\Hyp^2}(x,l)\leq\sinh d_{\Hyp^2}(x,{\gamma}(s)).$$
Hence, consider the partition ${\gamma}(0)=x_0,{\gamma}(1),\ldots,{\gamma}(N),{\gamma}(d_{\Hyp^2}(x_0,x))=x$, for $N$ the integer part of $d_{\Hyp^2}(x_0,x)$. We have
\begin{align*}
\left| \int_{\mathcal{G}[x_0,x]} \langle x,\boldsymbol{\sigma}\rangle d\mu_0\right| &\leq \sum_{i=1}^{N+1} \sinh(i)\mu_0([{\gamma}(i-1),{\gamma}(i)]) \\
&\leq \sum_{i=1}^{N+1}\frac{M}{2} e^i=\frac{M}{2}e\frac{e^{N+1}-1}{e-1}\leq \frac{M}{2} \frac{e^2}{e-1}e^{d_{\Hyp^2}(x_0,x)}\,.
\end{align*}
We can finally give a bound for the support function $h_0$ on $\D$. If $\pi(x)=z\in\D$,
$$h_0(z)=\frac{\bar h_0(x)}{\cosh d_{\Hyp^2}(x_0,x)}\leq\frac{M}{2} \frac{e^2}{e-1}\frac{e^{d_{\Hyp^2}(x_0,x)}}{\cosh d_{\Hyp^2}(x_0,x)}\leq M \frac{e^2}{e-1}\,.$$
This shows that the function $h_0$ is bounded by a constant which only depends on $M$ and on $x_0,y_0$. By Lemma \ref{lemma limit radial geodesics}, the bound also holds on $\partial\D$.
\end{proof}

\begin{proof}[Proof of Proposition \ref{boundedness curvature boundedness lamination}]
Let $A_n\in\SO_0(2,1)$ be a linear isometry such that $A_n(G(x_n))=(0,0,1)$, where $G:S\rar\Hyp^2$ is the Gauss map of $S$. Let $D'_n$ be the domain of dependence of the surface $A_n(S)$. Now let $t_n\in\R^{2,1}$ be such that the $t_n$-translate of $\partial_s D'_n$ has a support plane with normal vector $(0,0,1)$ with tangency point the origin.

Let $S_n=A_n(S)+t_n$. Let $u_n$ be the support function on $\D$ of $S_n$ and $h_n$ be the support function of its domain of dependence $D_n=D_n'+t_n$.

The support functions $u_n$ are uniformly bounded on any $\Omega$ with compact closure in $\D$, since we have $h_n-(1/\sqrt{|K|})\sqrt{1-|z|^2}\leq u_n\leq h_n$ and by Proposition \ref{uniform boundedness support function bounded lamination} the support functions $h_n$ of $D_n$ are uniformly bounded on $\Omega$. Hence $||u_n||_{C^0(\D)}<C$ for some constant $C$. Let $B_n$ be the shape operator of $S_n$. Equation (\ref{formula shape operator euclidean hessian}) in Lemma \ref{lemma formulae minkowski}, for $z=0$, gives $B^{-1}_n=\Hess(u_n)$. Applying Lemma \ref{lemma boundedness second derivatives}, the inverse of the principal curvatures of $S_n$, which are the eigenvalues of $B_n^{-1}$, cannot go to infinity at the origin. This concludes the proof that principal curvatures of $S_n$ cannot become arbitrarily small.
\end{proof}

\subsection{Cauchy surfaces with bounded principal curvatures have finite dual lamination} \label{estimate norm curvatures}

In this part, we will show that a Cauchy surface $S$ with principal curvatures bounded below, $\lambda_i\geq d$ for some $d>0$, are such that the measured geodesic lamination $\mu$ dual to $D(S)$ has finite Thurston norm. This shows the implication $iii)\Rightarrow i)$. More precisely, we prove:

\begin{prop} \label{prop thurston norm and principal curvatures}
Let $B$ be the shape operator of a convex spacelike surface $S$ in $\R^{2,1}$ such that the Gauss map is a homeomorphism. Let $\mu$ be the measured geodesic lamination dual to $D(S)$, where $D(S)$ is the domain of dependence of $S$. Then
\begin{equation} \label{estimate thurston norm operator norm}
||\mu||_{Th}\leq 2\sqrt{2(1+\cosh(1))}{||B^{-1}||_{op}}
\end{equation}
\noindent where $||B^{-1}||_{op}=\sup\{||B^{-1}(v)||/||v||:v\in TS\}$ is the operator norm.
\end{prop}

Note that $||B^{-1}||_{op}$ is the supremum of the inverse of the principal curvatures of $S$; alternatively, it is the inverse of the infimum of the principal curvatures of $S$. To prove Proposition \ref{prop thurston norm and principal curvatures}, we consider $x_1,x_2\in\Hyp^2$ with $d_{\Hyp^2}(x_1,x_2)\leq 1$ and take points $y_1$, $y_2$ on $\partial_s D(S)$ such that $P_1=y_1+x_1^\perp$ and $P_2=y_2+x_2^\perp$ are support planes of $D$. Recall $\mu(\mathcal{G}[x_1,x_2])$ denotes the value taken by the dual lamination $\mu$ on the geodesic segment $[x_1,x_2]$ which joins $x_1$ and $x_2$. We will also denote $||v||_-=\sqrt{\langle v,v\rangle}$ if $v\in\R^{2,1}$ is spacelike.

\begin{lemma} \label{inequality weight arc minkowski distance}
Let $y_1,y_2\in \partial_s D(S)$ and $P_1=y_1+x_1^\perp$ and $P_2=y_2+x_2^\perp$ be support planes for $\partial_s D(S)$ tangent to $\partial_s D(S)$ at $y_1$ and $y_2$. If $x_1$ and $x_2$ do not lie on any weighted leaf of the dual lamination $\mu$ of $D(S)$, then $$\mu(\mathcal{G}[x_1,x_2])\leq ||y_1-y_2||_-\,.$$
\end{lemma}
\begin{proof}
Assume first $\mathrm{supp}\mu\cap \mathcal{G}[x_1,x_2]$ determines a finite lamination, i.e. $\mu$ restricted to $\mathcal{G}[x_1,x_2]$ is composed of a finite number of weighted leaves $g_1,\ldots,g_p$ with weights $a_1,\ldots,a_p$. Then we have (compare Proposition \ref{support function from lamination})
$$y_2-y_1=\int_{\mathcal{G}[x_1,x_2]} \boldsymbol{\sigma} d\mu=\sum_{i=1}^p a_i \boldsymbol{\sigma}(g_i).$$
Since the geodesics $g_1,\ldots,g_p$ are pairwise disjoint, the unit normal vectors $\boldsymbol{\sigma}(g_1),\ldots,\boldsymbol{\sigma}(g_p)$ are such that $\langle \boldsymbol{\sigma}(g_i),\boldsymbol{\sigma}(g_j)\rangle\geq 1$. Hence
$$\langle y_2-y_1,y_2-y_1\rangle=\sum_{i=1}^p a_i^2+2\sum_{i< j}a_i a_j\langle \boldsymbol{\sigma}(g_i),\boldsymbol{\sigma}(g_j)\rangle\geq \left(\sum_{i=1}^p a_i\right)^2=\mu(\mathcal{G}[x_1,x_2])^2.
$$
This shows that the following inequality holds for a finite lamination $\mu$:
\begin{equation} \label{inequality weight arc minkowski distance for finite laminations}
\left(\int_{\mathcal{G}[x_1,x_2]} d\mu\right)^2\leq \langle \int_{\mathcal{G}[x_1,x_2]} \boldsymbol{\sigma} d\mu,\int_{\mathcal{G}[x_1,x_2]} \boldsymbol{\sigma} d\mu\rangle\,.
\end{equation}
In general, if $\mu$ restricted to $\mathcal{G}[x_1,x_2]$ is not a finite lamination, we can approximate in the weak* topology the lamination $\mu$ by finite laminations $\mu_n$ (compare Lemma \ref{fuchsian weak approximation}). As in Lemma \ref{limit weak uniform compact}, one can show 
\begin{equation} \label{approximation by finite laminations 1}
\int_{\mathcal{G}[x_1,x_2]} \boldsymbol{\sigma} d\mu_n\xrightarrow{n\rar\infty} \int_{\mathcal{G}[x_1,x_2]} \boldsymbol{\sigma} d\mu=y_2-y_1
\end{equation}
and 
\begin{equation} \label{approximation by finite laminations 2}
\int_{\mathcal{G}[x_1,x_2]} d\mu_n\xrightarrow{n\rar\infty} \int_{\mathcal{G}[x_1,x_2]} d\mu=\mu(\mathcal{G}[x_1,x_2])\,.
\end{equation}
Since (\ref{inequality weight arc minkowski distance for finite laminations}) holds for the finite laminations in the LHS of Equations \eqref{approximation by finite laminations 1} and \eqref{approximation by finite laminations 2}, the proof is complete.
\end{proof}

\begin{lemma} \label{surface bounded curvature hyperboloid}
Let $S$ be a convex surface in $\R^{2,1}$ such that the principal curvatures $\lambda_i$ of $S$ are bounded below, $\lambda_i\geq d>0$. 
For every point $p\in S$, the surface $S$ is contained in the future of the hyperboloid through $p$, tangent to $T_p S$ and with curvature $-d^2$.
\end{lemma}
\begin{proof}
Let $u:\D\rar \R$ and $\bar u:\Hyp^2\rar \R$, as usual, denote the support function of $S$ restricted to $\D$ and to $\Hyp^2$. Analogously, let $v:\D\to\R$ and $\bar v:\Hyp^2\rar \R$ be the support function of the hyperboloid $p+(1/d)\Hyp^2$, as in the hypothesis. Composing with an isometry, we can assume $p=(0,0,1/d)$ and $T_p S$ is the horizontal plane $x_3=1/d$. Hence $\bar v\equiv -1/d$, 
while from Equation \eqref{formula inverse gauss map} in Lemma \ref{lemma formulae minkowski} for the inverse of the Gauss map of $S$
$$G^{-1}(x)=\grad \bar u(x)-\bar u(x)x\,,$$
one can deduce $\bar u((0,0,1))=-1/d$ and $\grad \bar u((0,0,1))=0$. 
From the hypothesis, the eigenvalues of the shape operator of $S$ are larger than $d$ at every point. On the other hand the shape operator of the hyperboloid of curvature $-d^2$ is $d\id$.
Therefore $\Hess \bar u-\bar u\id<\Hess \bar v-\bar v\id$ and it follows (recalling the technology from Subsection \ref{preliminaries subsec support} of the support function restricted on $\Hyp^2$ and $\D$) that $v-u$ is a convex function on $\D$ with a minimum at $0\in\D$, where $(v-u)(0)=0$. Therefore $v-u$ is positive on $\D$, which shows that $v\geq u$ and thus proves the statement by Lemma \ref{lemma disuguaglianza funzioni supporto}.
\end{proof}

The following Lemma is a direct consequence.

\begin{lemma} \label{surface bounded curvature cone}
Let $S$ be a convex spacelike surface in $\R^{2,1}$ such that the principal curvatures of $S$ are bounded below, $\lambda_i\geq d>0$, and that the Gauss map G is a homeomorphism. Then for every $p\in S$, $D(S)\subset \Ip(r_d(p))$, where 
$$r_d(p)=p-\frac{1}{d}G(p).$$
\end{lemma}
\begin{proof}
By Lemma \ref{surface bounded curvature hyperboloid}, $S\subset \Ip(r_d(p))$ for every $p\in S$. It follows that the entire domain of dependence of $S$ is contained in $\Ip(r_d(p))$.
\end{proof}

\begin{proof}[Proof of Proposition \ref{prop thurston norm and principal curvatures}]
Assume the principal curvatures of $S$ are bounded below by $d>0$, and $d$ is the infimum of the principal curvatures of $S$. Hence $||B^{-1}||_{op}= 1/d$. 

Let us take $x_1,x_2\in\Hyp^2$, which do not lie on any weighted leaf of the dual lamination $\mu$ of $D(S)$, with $d_{\Hyp^2}(x_1,x_2)\leq 1$. 
Suppose $y_1,y_2\in \partial_s D(S)$ are such that  $P_1=y_1+x_1^\perp$ and $P_2=y_2+x_2^\perp$ are tangent planes for $\partial_s D(S)$. We will show that $||y_1-y_2||_-\leq 2\sqrt{2(1+\cosh(1))}/d$ and thus the estimate \eqref{estimate thurston norm operator norm} will follow by Lemma \ref{inequality weight arc minkowski distance}.

Suppose moreover $p_1,p_2\in S$ are such that $p_1+x_1^\perp$ and $p_2+x_2^\perp$ are tangent planes for $S$. Let us denote $U_i=\Ip(r_d(p_i))\cap\overline{\Ipm(T_{p_i}S)}$ and $V_i=\Ip(r_d(p_i))\cap\overline{\Ipm(P_i)}$, for $i=1,2$. See Figure \ref{fig:definitions} and \ref{fig:definitions2}. Note that $y_1,y_2\in\Ip(r_d(p_i))$, for $i=1,2$, by Lemma \ref{surface bounded curvature cone}.

\begin{figure}[htb]
\centering
\includegraphics[height=5.5cm]{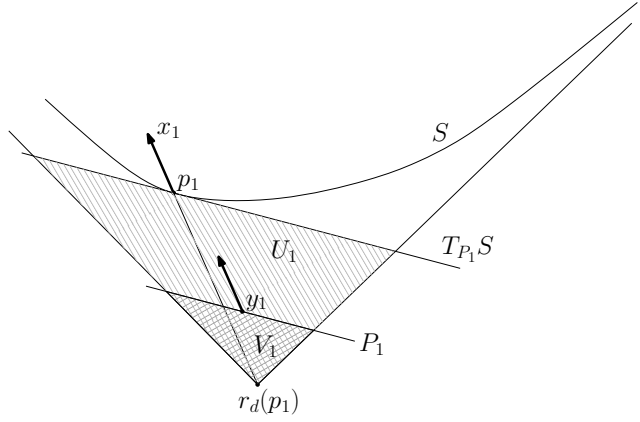}
\caption{The setting of the proof and the definitions of the sets $U_1$ and $V_1$. \label{fig:definitions}}
\end{figure}

\begin{figure}[htb]
\centering
\includegraphics[height=8.5cm]{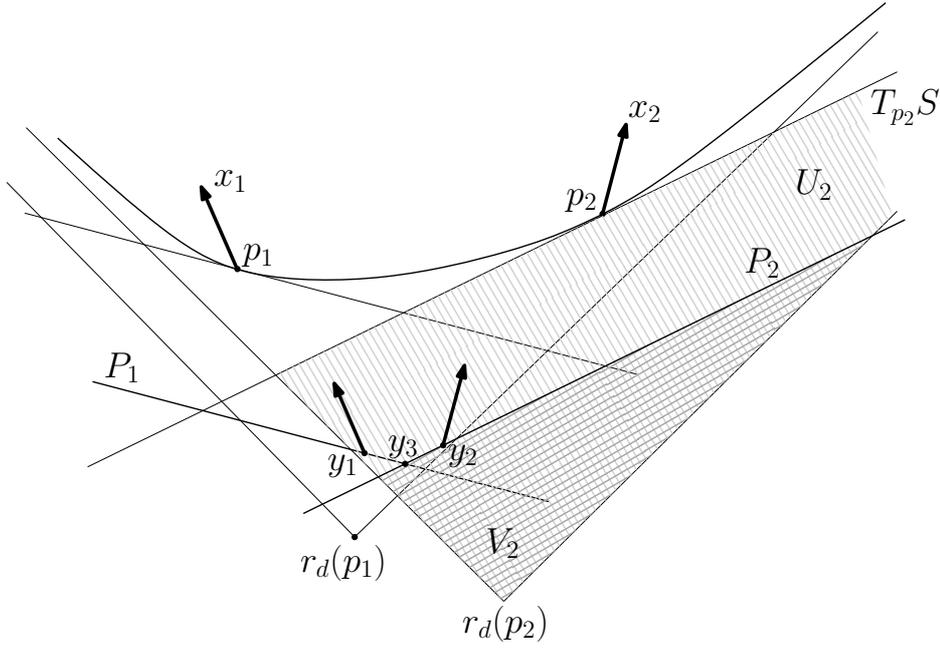}
\caption{Analogously, the definitions of $U_2$ and $V_2$. The case pictured is for $y_1\notin V_2$ and $y_2\notin V_1$, hence there is a point $y_3$ in $P_1\cap P_2$, which lies in $\Ip(r_d(p_1))\cap\Ip(r_d(p_2))$. \label{fig:definitions2}}
\end{figure}

\begin{figure}[htb]
\centering
\includegraphics[height=4cm]{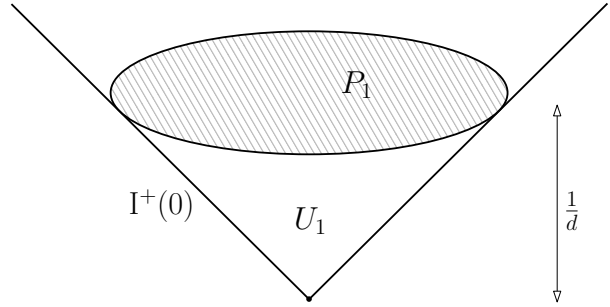}
\caption{An isometric image of $U_1$. \label{fig:imageU1}}
\end{figure}

By construction, $y_1\in V_1\subset U_1$ and $y_2\in V_2\subset U_2$. Let us consider separately three cases:

\emph{Case 1}: $y_2\in V_1$. Then both $y_1$ and $y_2$ are contained in $U_1$. Since $U_1$ can be mapped isometrically to the region $\left\{(x_1,x_2,x_3): x_3^2\geq x_1^2+x_2^2,\, x_3\leq 1/d\right\}$ (see Figure \ref{fig:imageU1}), it is easy to see that a spacelike segment contained in $U_1$ can have lenght at most $2/d$, which gives the statement in this particular case. 

\emph{Case 2}: $y_1\in V_2$. The estimate $||y_1-y_2||_-\leq 2/d$ is obtained in a completely analogous way.

\emph{Case 3}: $y_1\notin V_2$ and $y_2\notin V_1$. We claim that in this case $P_1\cap P_2$ contains a point $y_3$ in $\Ip(r_d(p_1))\cap\Ip(r_d(p_2))$. Indeed, if $P_1\cap P_2$ did not contain such a point, then the line $P_1\cap P_2$ would be disjoint from $\Ip(r_d(p_1))\cap\Ip(r_d(p_2))$ and there would be two possibilities. Either $P_2\cap \Ip(r_d(p_1))\cap\Ip(r_d(p_2))$ is contained in $\Ip(P_1)\cap \Ip(r_d(p_1))\cap\Ip(r_d(p_2))$, or $P_1\cap \Ip(r_d(p_1))\cap\Ip(r_d(p_2))$ is contained in $\Ip(P_2)\cap \Ip(r_d(p_1))\cap\Ip(r_d(p_2))$. The former case implies that $V_2$ contains $P_1\cap \Ip(r_d(p_1))\cap\Ip(r_d(p_2))$, which is not possible since by assumption $y_1\in P_1$ is not in $V_2$. The latter case analogously contradicts $y_2\notin V_1$.

Consider now the geodesic segments $y_3-y_1$ and $y_2-y_3$. Since $y_3$ and $y_1$ are both contained in $V_1$, we have $||y_3-y_1||_-\leq 2/d$. Analogously $||y_2-y_3||_-\leq 2/d$. If the plane $Q$ containing $y_1,y_2,y_3$ is spacelike or lightlike, then $$||y_2-y_1||_-\leq||y_2-y_3||_-+||y_3-y_1||_-\leq \frac{4}{d}\,.$$
If $Q$ is timelike (meaning that the induced metric on $Q$ is a Lorentzian metric), then
$$\langle y_2-y_1,y_2-y_1\rangle=||y_2-y_3||_-^2+||y_3-y_1||_-^2+2\langle y_2-y_3,y_3-y_1\rangle\,.$$

Let $v_1$ and $v_2$ be the future unit vectors in $Q$ orthogonal to $y_3-y_1$ and $y_2-y_3$. It is easy to check that $$|\langle y_2-y_3,y_3-y_1\rangle|=||y_2-y_3||_-||y_3-y_1||_-|\langle v_1,v_2\rangle|=||y_2-y_3||_-||y_3-y_1||_-\cosh d_{\Hyp^2}(v_1,v_2)\,.$$
We claim that $v_i$ is the orthogonal projection in $\Hyp^2$ of $x_i$ to the geodesic determined by $Q$ (namely, the geodesic through $v_1$ and $v_2$). Composing with an isometry, we can assume $y_3=0$, the direction spanned by $y_3-y_1$ is the line $x_2=x_3=0$ and $Q=\left\{x_2=0\right\}$. Then $v_1=(0,0,1)$ and $x_1=(0,\sinh t,\cosh t)$, where $t=d_{\Hyp^2}(v_1,x_1)$, and thus the claim holds. Of course the proof for $x_2$ and $v_2$ is analogous. This concludes the proof by Lemma \ref{inequality weight arc minkowski distance}, since
$$||y_2-y_1||_-^2\leq \frac{4}{d^2}(2+2\cosh(1))\,,$$
where we have used that $y_3$ and $y_2$ are contained in $P_2\subset U_2$, $y_1$ and $y_3$ are contained in $P_1\subset U_1$ and so $||y_2-y_3||_-,||y_3-y_1||_-\leq 2/d$ as above, and (since the projection to a line in $\Hyp^2$ is distance-contracting) $d_{\Hyp^2}(v_1,v_2)\leq d_{\Hyp^2}(x_1,x_2)\leq 1$.
\end{proof}

\subsection{Proof of Theorem \ref{big theorem zygmund finite dual lamination bounded curvatures}}

All the elements to conclude the proof of Theorem \ref{big theorem zygmund finite dual lamination bounded curvatures} have been obtained. Let us summarize the necessary steps. We must show the equivalence of the four points $i)$-$iv)$. The equivalence of $i)$ and $ii)$ is the content of Corollary
\ref{cor support function extends equivalence 12}. 

To show that $ii)$ implies $iv)$, the existence of a foliation by surfaces of constant curvature was proved in Theorem \ref{theorem foliation bounded}, while the fact that the principal curvatures are bounded was proved in Proposition \ref{boundedness curvature boundedness lamination}. As observed, this also implies completeness of the induced metric.

The condition of point $iv)$ is obviously stronger than the condition of $iii)$. Hence to conclude, it suffices to show that $iii)$ implies either $i)$ or $ii)$. In fact, Proposition \ref{prop thurston norm and principal curvatures} proves that if $iii)$ holds, namely if the principal curvatures are bounded from below by a positive constant, then $i)$ holds, i.e. the Thurston norm of the dual lamination is bounded (and actually Proposition \ref{prop thurston norm and principal curvatures} provides a quantitative version of this fact).

\appendix
\section{Constant curvature surfaces invariant under a parabolic group} \label{appendix parabolic}

In the appendix we construct some explicit solutions to the Monge-Amp\`ere equation associated to surfaces with constant curvature $K<0$, namely
\begin{equation} \label{monge ampere K}
\det D^2u(z)=\frac{1}{|K|}(1-|z|^2)^{-2}\,.
\end{equation}
We study constant curvature surfaces invariant under a one-parameter parabolic subgroup of isometries. In order to have such a surface, the subgroup must necessarily fix the origin. 

Hence, let us  denote by $A_\bullet:\R\rar\isom(\R^{2,1})$ the representation associated to the linear parabolic subgroup. Let us choose a basis $\{v_0,v_1,v_2\}$ of $\R^{2,1}$ such that $v_0$ is the null vector fixed by the parabolic group, $v_1$ is a null vector with $\langle v_0,v_1\rangle=-1$, and $v_2$ is a spacelike unit vector orthogonal to both $v_0$ and $v_1$.

The parabolic group is acting by
\begin{align*}
A_t(v_0)=&v_0\,; \\
A_t(v_1)=&(t^2/2)v_0+v_1+tv_2\,; \\
A_t(v_2)=&v_2+tv_0\,.
\end{align*}

Let $\gamma_0(s)=\frac{\sqrt{2}}{2}(e^s v_0+e^{-s}v_1)$ be the unit speed geodesic of $\Hyp^2$ with endpoints $[v_1]$ (for $s\to -\infty$) and $[v_0]$ (for $s\to +\infty$). Let us consider the following parametrization of $\Hyp^2$:
$$\sigma(t,s)=A_t(\gamma_0(s))\,,$$
namely, the levels $\{s=c\}$ are horocycles, while the levels $\{t=c\}$ are geodesics asymptotic to $[v_0]$. In these coordinates, the metric of $\Hyp^2$ takes the form $ds^2+(e^{-2s}/2)dt^2$.

We consider support functions restricted to $\Hyp^2$, which we denote as usual by $\bar u:\Hyp^2\to\R$, corresponding to surfaces of constant curvature. Hence we want to find solutions of the equation
\begin{equation} \label{constant curvature det hess - id}
\det(\Hess\,\bar u-\bar u\,\id)=\frac{1}{|K|}\,,
\end{equation}
where $K$ is a negative constant. Since are imposing that the surface dual to $\bar u$ is invariant for the parabolic group (with no translation), recalling Equation \eqref{transf rule support function} for the transformation of support functions under isometries of $\R^{2,1}$, we look for a solution which only depends on $s$, namely a solution of the form $\bar u(t,s)=f(s)$.

By a direct computation, one can see that the gradient and the Hessian of $\bar u$ for the hyperbolic metric in this coordinate frame has the form:
\begin{align*}
\grad \bar u=&f'(s)\partial_s \\
\Hess (\bar u)(\partial_s)=&f''(s)\partial_s \\
\Hess (\bar u)(\partial_t)=&-f'(s)\partial_t\,,
\end{align*}
therefore the constant curvature condition \eqref{constant curvature det hess - id} gives
\begin{equation} \label{equation constant curvature parabolic}
(f''(s)-f(s))(-f'(s)-f(s))=1/|K|\,.
\end{equation}
We now solve Equation \eqref{equation constant curvature parabolic}. By convexity, we impose that both eigenvalues $(f''(s)-f(s))$ and $(-f'(s)-f(s))$ are positive.
Let us perform the change of variables 
\begin{equation} \label{change variables}
g(s)=-f'(s)-f(s)\,,
\end{equation} 
so that Equation \eqref{equation constant curvature parabolic} becomes
\begin{equation} \label{equation constant curvature parabolic change}
g(s)(g(s)-g'(s))=1/|K|\,,
\end{equation}
whose general positive solution is, as $C$ varies in $\R$,
\begin{equation}
g(s)=\sqrt{|K|^{-1}+Ce^{2s}}\,.
\end{equation}

\subsection{Solutions for $C=0$}
We observe that the case $C=0$ gives the trivial solution, namely the hyperboloid. Indeed $f$ can be recovered by integrating \eqref{change variables}, hence obtaining
\begin{equation} \label{explicit solution parabolic}
f(s)=e^{-s}\left(D-\int_0^s e^xg(x)dx\right)\,.
\end{equation}
Observe that the term $e^{-s}D$ corresponds to a translation in the direction $-\sqrt{2}D v_0$. Hence, as the parameter $D$ varies over $\R$, the corresponding surface varies by a translation in the line spanned by $v_0$. 
If $C=0$, we have $g\equiv1/\sqrt{|K|}$. By choosing $D$ suitably, we obtain the solution
\begin{equation} \label{solution parabolic hyperboloid}
f_{0,K}(s)=-\frac{1}{\sqrt{|K|}}\,,
\end{equation} 
which is the support function of a hyperboloid of curvature $K$ centered at the origin. 

\subsection{Solutions for $C>0$}

If $C> 0$, from Equation \eqref{explicit solution parabolic} we obtain the solution (for a constant $D$ which we will fix later)
\begin{equation}
f_{C,K}(s)=-\frac{1}{2}\sqrt{|K|^{-1}+Ce^{2s}}-\frac{1}{2|K|\sqrt{C}}e^{-s}\log\left(\sqrt{C}\sqrt{|K|^{-1}+Ce^{2s}}+Ce^s\right)+e^{-s}D\,.
\end{equation}
We now describe some of the properties of the surface $S_{C,v_0}(K)$ whose support function is $\bar u_{C,K}(t,s)=f_{C,K}(s)$, for any fixed curvature $K<0$.

First, we want to determine the value on $\partial \D$ of the support function $ u_{C,K}$, namely the restriction to $\D$ of the 1-homogeneous extension of $\bar u_{C,K}$. Let us denote by $\sigma(t,s)_z$ the vertical component of $\sigma(t,s)$. Then we have
$$u_{C,K}(t,s)=\frac{f_{C,K}(s)}{\sigma(t,s)_z}\,.$$
Using the invariance for the parabolic group in $\isom(\R^{2,1})$, it suffices to consider the case $t=0$. We have 
$$\sigma(0,s)_z=-\langle \sigma(0,s), \frac{v_0+v_1}{\sqrt{2}}\rangle=\cosh(s)\,.$$
Observe that the term $e^{-s}D/\cosh(s)$ tends to $0$ when $s\to\infty$ and to $2D$ when $s\to -\infty$. By an explicit computation, choosing $D$ suitably, we can obtain
$$\lim_{s\rar-\infty}u_{C,K}(t,s)=\lim_{s\rar-\infty}\frac{f_{C,K}(s)}{\sigma(t,s)_z}=0\,,$$
while
$$\lim_{s\rar-\infty}u_{C,K}(t,s)=\lim_{s\rar+\infty}\frac{f_{C,K}(s)}{\sigma(t,s)_z}=-\sqrt{C}\,,$$
hence the support function at infinity is
$$u_{C,K}|_{\partial\D}(z)=\begin{cases} -\sqrt{C} & [z]=[v_0] \\ 0 & [z]\neq[v_0] \end{cases}\,.$$
Geometrically, this means that the domain of dependence of the surface is the future of a parabola, obtained as the intersection of the cone centered at the origin (whose support function is identically 0) with a lightlike plane with normal vector $v_0$ (whose intercept on the $z$-axis is $\sqrt{C}$). See Figure \ref{fig:parabola} in subsection \ref{cauchy and uniqueness}.
It is easy to see that the dual lamination is the measured geodesic lamination of $\Hyp^2$ whose leaves are all geodesics asymptotic to $[v_0]$, with a measure invariant under the parabolic group (a multiple of the Lebesgue measure, in the upper half-space model, pictured in Figures \ref{fig:laminationparabolic} and \ref{fig:laminationhalfspace}). 

\begin{figure}[htb]
\centering
\begin{minipage}[c]{.45\textwidth}
\centering
\includegraphics[height=4cm]{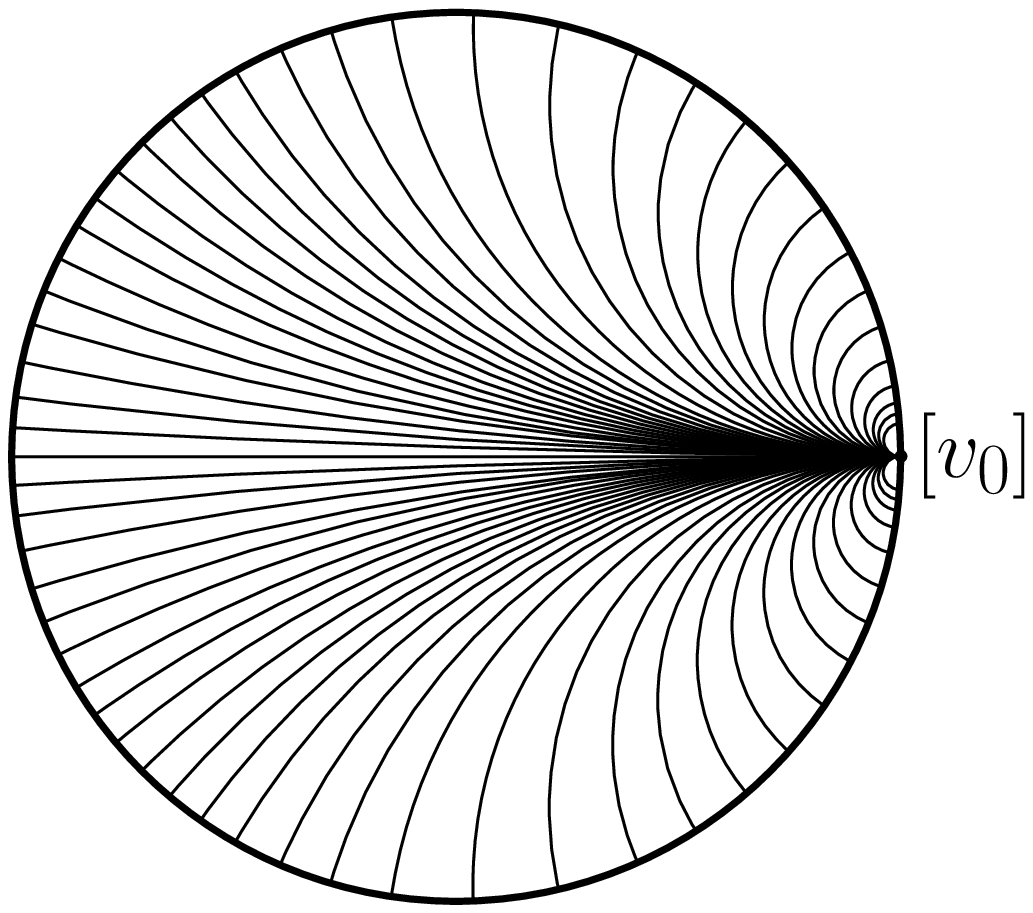}
\caption{The dual lamination to the boundary of the domain of dependence, which is the future of a parabola.} \label{fig:laminationparabolic}
\end{minipage}%
\hspace{5mm}
\begin{minipage}[c]{.45\textwidth}
\centering
\includegraphics[height=4cm]{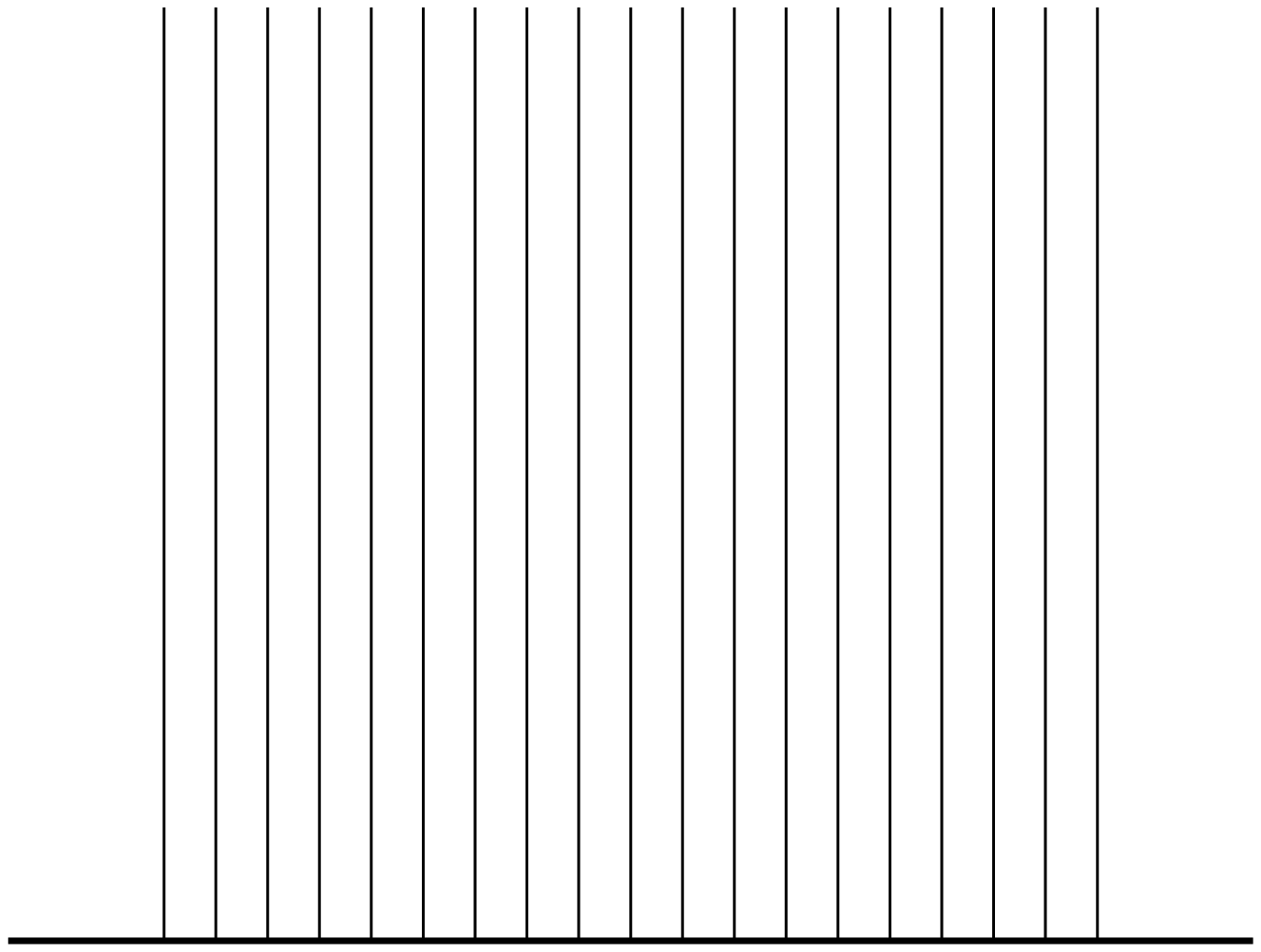}
\caption{In the upper half space model, with fixed point at infinity, the measure is the Lebesgue measure.} \label{fig:laminationhalfspace}
\end{minipage}
\end{figure}

The surface $S_{C,v_0}(K)$ dual to $u_{C,K}$ can be described by an explicit parametrization, recalling that the inverse of the Gauss map of the surface is $G^{-1}(x)=\grad \bar u_{C,K}(x)-\bar u_{C,K}(x)x$. In these coordinates, 
\begin{align} \label{inverse gauss map surface parabolic group c positive}
G^{-1}(\sigma(t,s))&=\grad \bar u_{C,K}(t,s)-\bar u_{C,K}(t,s)\sigma(t,s)\\
&=\frac{\sqrt{2}}{2}f'(s)(e^s v_0-e^{-s}A_t(v_1))-\frac{\sqrt{2}}{2}f(s)(e^s v_0+e^{-s}A_t(v_1))\\
&=\frac{\sqrt{2}}{2}\left((f'(s)-f(s))e^s v_0-(f'(s)+f(s))e^{-s}A_t (v_1)\right)\\
&=\frac{\sqrt{2}}{2}\left(-(g(s)+2f(s))e^s v_0+g(s)e^{-s}A_t (v_1)\right)\,.
\end{align}

We now want to show that the surface is an entire graph. For this purpose, we will use the following criterion.
\begin{lemma} \label{criterion spacelike graph properness}
Let $u:\D\rar\R$ be a $C^2$ support function with positive Hessian, and suppose that the inverse of the Gauss map $G^{-1}:\D\rar\R^{2,1}$ is proper.
Then the boundary of the future-convex domain $D$ defined by $u$ is a spacelike entire graph.
\end{lemma}
\begin{proof}
As $u$ is $C^1$, we can use Equation \eqref{formula inverse gauss map} and get $G^{-1}(x)=\grad\bar u(x)-\bar u(x)x$ for every $x\in\Hyp^2$.
It can be readily shown that this implies that the vertical projection of $G^{-1}(x)$ is $Du(z)$, where $z=\pi(x)\in\D$, see Lemma 2.8 of \cite{bonfill}.
As the Hessian of $u$ is positive, the image of the gradient map of $u$ is an open subset of $\R^2$, so it follows that $\partial_s D$ is open in $\partial D$.
Since $G^{-1}$ is proper, $\partial_s D$ is also closed in $\partial D$, and this concludes the proof.
\end{proof}

We will actually show that the height function given by
$$z(t,s)=-\langle G^{-1}(\sigma(t,s)),\frac{v_0+v_1}{\sqrt{2}}\rangle$$
is proper. By a direct computation,
\begin{align*}
2z(t,s)&=-(g(s)+2f(s))e^s+g(s)e^{-s}\left(\frac{t^2}{2}+1\right)\\
&=\frac{1}{|K|{\sqrt{C}}}\log\left(\sqrt{C}\sqrt{|K|^{-1}+Ce^{2s}}+Ce^s\right)+e^{-s}\sqrt{|K|^{-1}+Ce^{2s}}\left(\frac{t^2}{2}+1\right)\\
&\geq \frac{1}{|K|{\sqrt{C}}}s+e^{-s}|K|^{-1/2}\left(\frac{t^2}{2}+1\right)-C_0
\,,
\end{align*}
for some constant $C_0$. Observe that $z(t,s)$ tends to infinity for $s\to\pm\infty$. It is easily checked that on a sequence $\sigma(t_n,s_n)$ which escapes from every compact, $t_n^2+s_n^2\to\infty$, and thus $z(t_n,s_n)\to\infty$. This concludes the claim that $S_{C,v_0}(K)$ is a spacelike entire graph, by Lemma \ref{criterion spacelike graph properness}.

Finally, we briefly discuss the isometry type of the induced metric. 
By an explicit computation using the expression in Equation \eqref{inverse gauss map surface parabolic group c positive}, we find the pull-back of the induced metric via $G^{-1}$:
\begin{equation} \label{equation pull-back gauss map parabolic}
(G^{-1}\circ\sigma)^*(g_{\R^{2,1}})=(f''(s)-f(s))^2 ds^2+\frac{1}{2}e^{-2s}g(s)^2 dt^2\,,
\end{equation}
where it turns out that
$$f''(s)-f(s)=\frac{1}{|K|\sqrt{|K|^{-1}+Ce^{2s}}}$$
By an explicit change of variables
$$r(s)=\frac{1}{\sqrt{|K|}}\arctanh\left(\frac{1}{\sqrt{1+|K|Ce^{2s}}}\right)\,,$$
so that $(r'(s))^2=(f''(s)-f(s))^2$, by computing
$$g(s)^2=\frac{1}{|K|}(1+|K|Ce^{2s})=\frac{1}{|K|\tanh^2(r\sqrt{|K|})}$$
and
$$e^{-2s}=|K|C\sinh^2\left(r\sqrt{|K|}\right)$$
one obtains that the induced metric is
$$dr^2+\left(\frac{C}{2}\right)\cosh^2\left({r}{\sqrt{|K|}}\right)dt^2\,.$$
Rescaling $t$, one obtains
$$dr^2+\cosh^2\left({r}{\sqrt{|K|}}\right)dt^2\,,$$
that is, the first fundamental form of $S_{C,v_0}(K)$ is isometric to a half-plane of constant curvature $K$, namely, to the region of a hyperboloid bounded by a geodesic $l$.
\begin{figure}[htb]
\centering
\includegraphics[height=5.5cm]{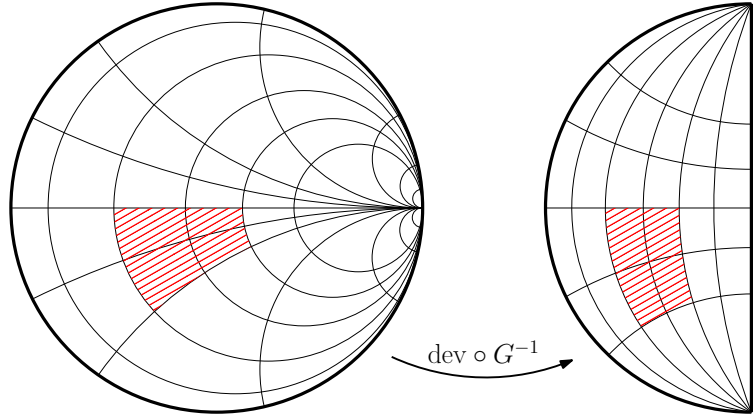}
\caption{A developing map $\dev:S_{C,v_0}(K)\rar\Hyp^2$ for the induced metric, composed with the inverse of the Gauss map $G^{-1}:\Hyp^2\rar S_{C,v_0}(K)$, in the Poincar\'e disc model of $\Hyp^2$. \label{fig:developing}}
\end{figure}

We resume the content of this subsection in the following proposition. Let us denote by $\Hyp^2(K)$ the rescaled hyperbolic plane, of curvature $K<0$, and by $\Hyp^2(K)_+$ a half-plane in $\Hyp^2(K)$. Observe that $\Hyp^2(K)_+$ has a one-parameter group of isometries $T(l)$ which consists of (the restriction of) hyperbolic translations along the geodesic $l$ which bounds $\Hyp^2(K)_+$.
\begin{prop} 
For every $K<0$, $C>0$ and every null vector $v_0\in\R^{2,1}$ there exists an isometric embedding 
$$i_{K,C,v_0}:\Hyp^2(K)_+\to\R^{2,1}$$ 
with image a Cauchy surface $S_{C,v_0}(K)$ in the domain of dependence whose support function at infinity is 
$$\varphi(z)=\begin{cases} -\sqrt{C} & [z]=[v_0] \\ 0 & [z]\neq[v_0] \end{cases}\,.$$
The surface $S_{C,v_0}(K)$ is a spacelike entire graph and $i_{K,C,v_0}$ is equivariant with respect to the group of isometries $T(l)$ of $\Hyp^2(K)_+$ and the parabolic linear subgroup of $\isom(\R^{2,1})$ fixing $v_0$.
\end{prop}

\subsection{Solutions for $C<0$}

If $C< 0$, the function $g(s)=\sqrt{|K|^{-1}+Ce^{2s}}$ is only defined for 
$$s\leq \frac{1}{2}\log\left(\frac{1}{|CK|}\right)\,.$$
From Equation \eqref{explicit solution parabolic} we can explicitly write the solution (again $D$ is a constant to be fixed):
\begin{equation} \label{explicit solution singular}
f_{C,K}(s)=-\frac{1}{2}\sqrt{|K|^{-1}+Ce^{2s}}-\frac{1}{2|K|\sqrt{|C|}}e^{-s}\arctan\left(\frac{\sqrt{|C|}e^s}{\sqrt{|K|^{-1}+Ce^{2s}}}\right)+e^{-s}D\,.
\end{equation}
Again, we study briefly the properties of the surface $S_{C,v_0}(K)$ whose support function is $\bar u_{C,K}(t,s)=f_{C,K}(s)$. Observe that the solution \eqref{explicit solution singular} is only defined in the range $s\leq \frac{1}{2}\log\left(\frac{1}{|CK|}\right)$, namely, in the complement of a horoball. Let us notice that, in the same notation as before, the limit of the support function (which only makes sense for $s\to -\infty$) is
$$\lim_{s\rar-\infty} u_{C,K}(t,s)=\lim_{s\rar-\infty}\frac{f_{C,K}(s)}{\sigma(t,s)_z}=0\,,$$
provided we choose $D=0$. On the other hand, as $s\to\frac{1}{2}\log\left(\frac{1}{|CK|}\right)$, the function $f_{C,K}(s)$ has the finite limit $-(\pi/4)\sqrt{|K|^{-1}}$. We observe that $\bar u_{C,K}(t,s)=f_{C,K}(s)$ can be extended to a convex function defined on the whole $\Hyp^2$ by declaring 
$$f_{C,K}(s)=-\frac{\pi}{4}\frac{1}{|K|\sqrt{|C|}}e^s$$
for $s\geq \frac{1}{2}\log\left(\frac{1}{|CK|}\right)$. We will now denote by $\bar u_{C,K}(t,s)=f_{C,K}(s)$ the function extended in this way. The surface $S_{C,v_0}(K)$ is thus a constant curvature surface which develops a singular point, namely it intersects the boundary of the domain of dependence, which in this case is just $\Ip(0)$. The inverse of the Gauss map sends the whole horoball $\{s\geq \frac{1}{2}\log\left(\frac{1}{|CK|}\right)\}$ to the point $\frac{\sqrt{2}}{2}\frac{1}{|K|\sqrt{|C|}}\frac{\pi}{4}v_0$.


We remark that $u_{C,K}$ is a generalized solution to the Monge-Amp\`ere equation on the disc
$\det D^2 u=\nu,$
where $\nu$ in this case is a measure which coincides with $(1/|K|)(1-|z|^2)^{-2}\mathcal{L}$ on the complement of the horoball (where $\mathcal{L}$ is the Lebesgue measure), and is $0$ inside the horoball.

By a computation analogous to the previous case,
to compute the induced metric we manipulate Equation \eqref{equation pull-back gauss map parabolic}: setting
$$r(s)=\frac{1}{\sqrt{|K|}}\arctanh(\sqrt{1+|K|Ce^{2s}})\,,$$
and replacing
$$g(s)^2=\frac{1}{|K|}(1+|K|Ce^{2s})=\frac{1}{|K|}\tanh^2(r\sqrt{|K|})$$
and
$$e^{-2s}=|K|C\cosh^2(r\sqrt{|K|})$$
we obtain the expression for the metric
$$dr^2+\left(\frac{C}{2}\right)\sinh^2\left({r}{\sqrt{|K|}}\right)dt^2\,,$$
or, after rescaling of $t$,
$$dr^2+\sinh^2\left({r}{\sqrt{|K|}}\right)dt^2\,.$$
This shows that the first fundamental form of $S_{C,v_0}(K)$ is isometric to the universal cover of the complement of a point in $\Hyp^2(K)$, which we will denote by $\widetilde{\Hyp^2(K)\setminus {p}}$. Let $R(p)$ the group of rotations of $\Hyp^2(K)$ fixing a point $p$ and let $\widetilde{R(p)}$ be its universal cover. We conclude by including all the information in the following proposition.

\begin{figure}[htb]
\centering
\includegraphics[height=5.5cm]{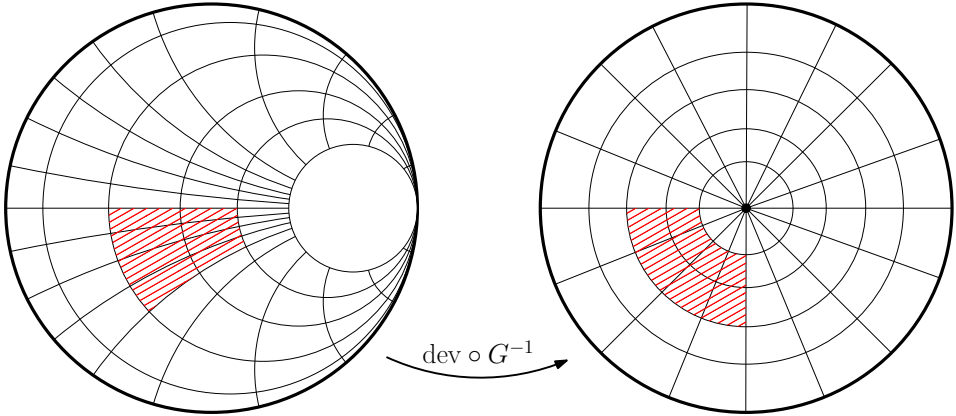}
\caption{Again, the developing map in the Poincar\'e disc model. In the case $C<0$, the inverse of the Gauss map $G^{-1}:\Hyp^2\rar S_{C,v_0}(K)$ shrinks a horoball to a point. \label{fig:developing2}}
\end{figure}

\begin{prop} 
For every $K<0$, $C<0$ and every null vector $v_0\in\R^{2,1}$ there exists an isometric embedding 
$$i_{K,C,v_0}:\widetilde{\Hyp^2(K)\setminus {p}}\to\R^{2,1}$$ 
with image a Cauchy surface $S_{C,v_0}(K)$ for $\Ip(0)$.
The closure of the surface $S_{C,v_0}(K)$ intersects the null cone $\partial\Ip(0)$ in the point $\beta v_0$, where
$$\beta=\frac{\sqrt{2}}{2}\frac{1}{|K|\sqrt{|C|}}\frac{\pi}{4}\,.$$
The inverse of the Gauss map of the closure of $S_{C,v_0}(K)$ maps a horoball of $\Hyp^2$ to $\beta v_0$. Finally $i_{K,C,v_0}$ is equivariant with respect to the group of isometries $\widetilde{R(p)}$ of $\widetilde{\Hyp^2(K)\setminus {p}}$ and the parabolic linear subgroup of $\isom(\R^{2,1})$ fixing $v_0$.
\end{prop}

\bibliographystyle{plain}
\bibliography{bs-bibliography}

\end{document}